\documentclass[8pt,CJK]{article}
\usepackage[utf8]{inputenc}
\usepackage[english]{babel}  
\usepackage[T1]{fontenc}     
\usepackage{lmodern}

\usepackage{amsfonts}
\usepackage{amsmath}
\usepackage{amssymb}
\usepackage{amsthm}
\usepackage{bm} 			
\usepackage{dsfont} 			
\usepackage{esint}			

\usepackage[pdftex]{graphicx}
\usepackage{graphics}
\usepackage{float}
\usepackage{tikz}

\usepackage{url}
\usepackage{enumitem}

\usepackage{color} 			
\usepackage{cancel}			
\usepackage{stmaryrd}			
\usepackage[sort,nocompress]{cite} 	
\usepackage{authblk} 	

\newcommand{\R}{\mathbb{R}}

\newcommand{\NN}{\mathbb{N}}

\newcommand{\LL}{{\rm{L}}}
\newcommand{\CC}{{\rm{C}}}
\newcommand{\HH}{{\rm{H}}}
\newcommand{\WW}{{\rm{W}}}

\newcommand{\lt}{\left}
\newcommand{\rt}{\right}
\def\[{[\![}
\def\]{]\!]}

\newcommand{\dd}{{\rm{d}}}

\newcommand{\Id}{{\rm Id}}

\newcommand{\argmax}{{\rm{arg\,max}}}

\newcommand{\loc}{{\rm{loc}}}
\newcommand{\comp}{{\rm{c}}}

\newcommand{\Boule}{{\rm B}}
\newcommand{\Q}{{\rm Q}}
\newcommand{\dist}{{\rm dist}}

\def\Xint#1{\mathchoice
  {\XXint\displaystyle\textstyle{#1}}%
  {\XXint\textstyle\scriptstyle{#1}}%
  {\XXint\scriptstyle\scriptscriptstyle{#1}}%
  {\XXint\scriptscriptstyle\scriptscriptstyle{#1}}%
  \!\int}
\def\XXint#1#2#3{{\setbox0=\hbox{$#1{#2#3}{\int}$}
    \vcenter{\hbox{$#2#3$}}\kern-.5\wd0}}

\def\fint{\Xint-}

\renewcommand{\epsilon}{\varepsilon}
\renewcommand{\tilde}{\widetilde}
\newcommand{\langl}{\lt\langle}
\newcommand{\rangl}{\rt\rangle}

\newcommand{\Dom}{{\rm D}}
\newcommand{\tilDom}{\tilde{\rm D}}
\newcommand{\abar}{\bar{a}}
\newcommand{\taubar}{\bar{\tau}}
\newcommand{\rhobar}{\bar{\rho}}
\newcommand{\ubar}{\bar{u}}
\newcommand{\phitilde}{\tilde{\phi}^{\mathfrak{C}}}
\newcommand{\phiDtilde}{\tilde{\phi}^{\mathfrak{D}}}

\newcommand{\phiC}{\phi^{\mathfrak{C}}}
\newcommand{\reg}{{\rm reg}}

\newcommand{\rstar}{r_*}
\newcommand{\rmax}{R}
\newcommand{\Cstar}{C_*}

\newcommand{\radius}{r}
\newcommand{\phiD}{\phi^{\mathfrak{D}}}

\newcommand{\phiup}{\phi^{\mathfrak{D}, {\rm up}}}
\newcommand{\sigup}{\sigma^{\mathfrak{D}, {\rm up}}}
\newcommand{\phidown}{\phi^{\mathfrak{D}, {\rm down}}}
\newcommand{\sigdown}{\sigma^{\mathfrak{D}, {\rm down}}}
\newcommand{\etaup}{\eta_{\rm up}}
\newcommand{\etadown}{\eta_{\rm down}}
\newcommand{\etabulk}{\eta_{\rm bulk}}
\newcommand{\sigmaD}{\sigma^{\mathfrak{D}}}
\newcommand{\tisigmaD}{\tilde{\sigma}^{\mathfrak{D}}}

\newcommand{\Reps}{\mathcal{R}}

\newcommand{\Energ}{\mathcal{E}}
\newcommand{\Exc}{{\rm Exc}}
\newcommand{\ueps}{u_\epsilon}
\newcommand{\utieps}{\tilde{u}_\epsilon}
\newcommand{\et}{\;\;\text{and}\;\;}
\newcommand{\dans}{\;\;\text{in}\;\;}
\newcommand{\si}{\;\;\text{if}\;\;}

\newcommand{\pour}{\;\;\text{for}\;\;}
\newcommand{\pourtout}{\;\;\text{for all}\;\;}
\newcommand{\sur}{\;\;\text{on}\;\;}
\newcommand{\lhs}{l.~h.~s.\ }
\newcommand{\rhs}{r.~h.~s.\ }

\newtheorem{theorem}{Theorem}[section]
\newtheorem{corollary}[theorem]{Corollary}
\newtheorem{lemma}[theorem]{Lemma}

\newtheorem{proposition}[theorem]{Proposition}
\newtheorem*{proposition*}{Proposition}

\newtheorem*{theorem*}{Theorem}
\newtheorem*{lemma*}{Lemma}
\newtheorem*{corollary*}{Corollary}
\newtheoremstyle{TheoremNum}
{\topsep}{\topsep}              		
{\itshape}                      		
{}                              		
{\bfseries}                     		
{.}                             		
{ }                             		
{\thmname{#1}\thmnote{ \bfseries #3}}	
\theoremstyle{TheoremNum}

\theoremstyle{remark}
\newtheorem{remark}{Remark}

\title{Stochastic homogenization and geometric singularities : a study on corners}

\author[1,*]{Marc Josien}
\author[2]{Claudia Raithel}
\author[3]{Mathias Schäffner}

\affil[1]{CEA, DES, IRESNE, DEC, Cadarache, F-13108, Saint-Paul-Lez-Durance, France}
\affil[2]{TU Wien, Wiedner Hauptstrasse 8-10, 1040 Wien, Austria}
\affil[3]{TU Dortmund, Vogelpothsweg 87,44227 Dortmund, Germany}
\affil[*]{correspondence to : \texttt{marc.josien@cea.fr}}
\begin{document}

\maketitle

\begin{abstract}
In this contribution we are interested in the quantitative homogenization properties of linear elliptic equations with homogeneous Dirichlet boundary data in polygonal domains with corners.
To begin our study of this situation, we consider the setting of an angular sector in $2$ dimensions : Unlike in the whole-space, on such a sector there exist non-smooth harmonic functions (these depend on the angle of the sector).
Here, we construct extended homogenization correctors corresponding to these harmonic functions and prove growth estimates for these which are quasi-optimal, namely optimal up to a logarithmic loss.
Our construction of the \textit{corner correctors}  relies on a large-scale regularity theory for $a$-harmonic functions in the sector, which we also prove and which, as a by-product, yields a Liouville principle. 
We also propose a nonstandard 2-scale expansion, which is adapted to the sectoral domain and incorporates the \textit{corner correctors}.
Our final result is a quasi-optimal error estimate for this adapted 2-scale expansion.
\end{abstract}



\newpage

\tableofcontents

\newpage
\section{Introduction}

This article investigates the interplay between microstructures and a nonsmooth macroscopic geometry from the point of view of theoretical stochastic homogenization.
The motivation comes from the need to tackle realistic objects made of heterogeneous materials (\textit{e.g.} a building made of concrete).
Indeed, these objects, manufactured or not, typically display non-smooth boundaries at the macroscopic scale
and can be made of different heterogeneous materials, or may have defects.
In all of these frameworks, the classical assumption in stochastic homogenization of an infinite stationary medium \cite{Varadhan_1979, JKO} breaks in a specific geometric region :
near the boundary, the interface or the defects.
As a consequence, in order to obtain estimates one has to go beyond the classical approach, not only at the microscopic scale but also at the macroscopic scale.
In particular, stationarity has to be replaced by a more flexible assumption.

Here, we consider the simple case of a linear elliptic equation that is posed in a corner-like domain (\textit{cf.} Figure~\ref{NiceCorner}) consisting of a heterogeneous material. 
This setting goes beyond the currently available theory of stochastic homogenization, in which equations are usually posed either in the whole space or in domains that are smooth at the macroscopic scale \cite{Allaire,Gloria_Neukamm_Otto_2019, Armstrong_book_2018}. 
Indeed, even in the setting of homogeneous materials, harmonic functions are known to display a singularity near the tip of the corner \cite{Dauge_Livre_1988} --this makes the elliptic regularity theory for domains with corners deviate from the classical theory on smooth domains. 
In the current contribution we show that this singular behavior which occurs at the tip of the corner in the setting of a homogeneous material has a natural analogue in heterogeneous materials. 
Towards this end, we unveil the role of \textit{corner correctors} that are related to the aforementioned singularities. 
Our main goal is to introduce, build, and precisely estimate the growth rate of these corner correctors. 
Equipped with them, we extend the well-known Lipschitz regularity theory of Avellaneda and Lin \cite{AvellanedaLin} to the situation of corner-like domains by providing a Taylor-like expansion.
As a corollary, we establish a Liouville principle.
Last, we propose a (nonstandard) 2-scale expansion that is well-adapted to the corner-like domain, for which we establish quasi-optimal error estimates.
This latter result we also illustrate through means of numerical experiments.

\begin{figure}[H]
\includegraphics[width=\textwidth]{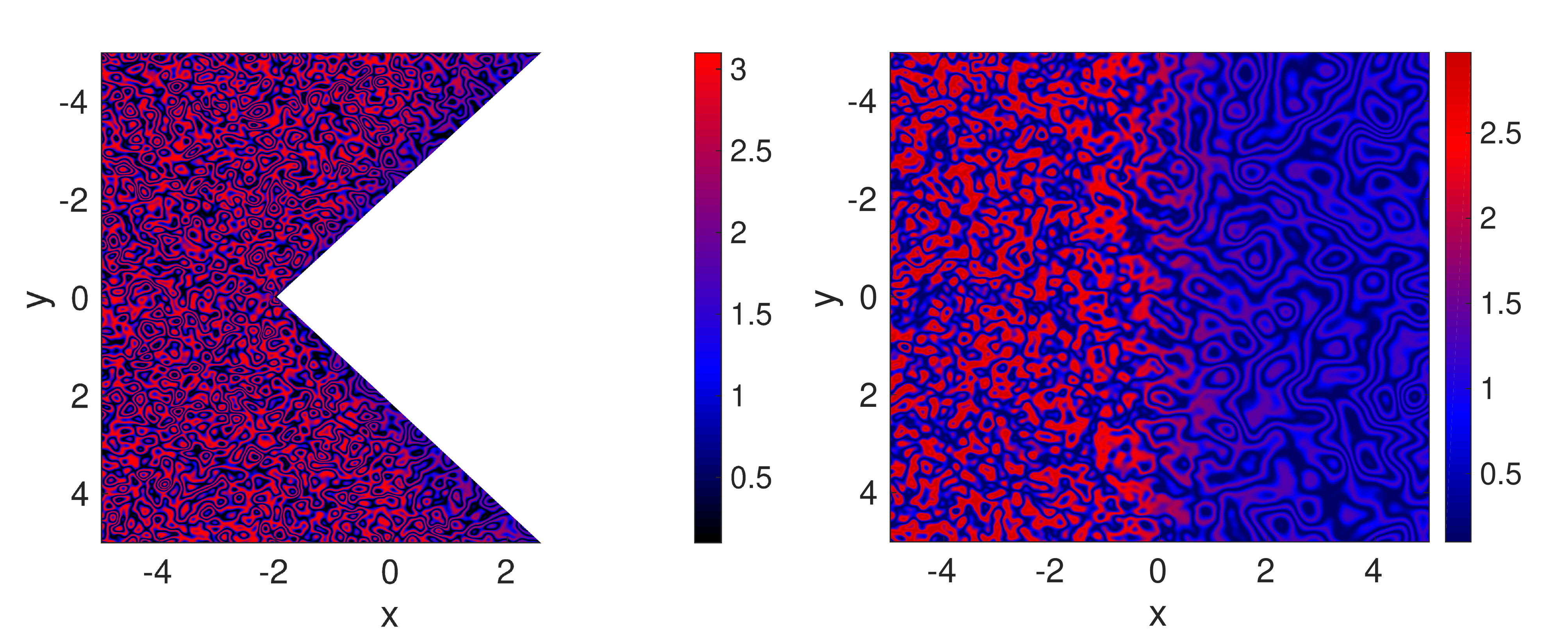}
\caption{On the left, a corner inside a heterogeneous material; on the right, a flat interface between two different heterogeneous materials (see \cite{JosienRaithel_2019}). The level of color represents the value of the (scalar) coefficient field.}\label{NiceCorner}
\end{figure}

\subsection{The problem under study}

The main purpose of this article is to find a method for obtaining fine estimates for the homogenization phenomenon of a linear elliptic equation set on a domain that is non-smooth at the macroscopic scale.
Towards this end, we choose to consider the simplest setting; namely, an infinite $2$-dimensional corner (see Figure \ref{Fig :Geom} below)
\begin{equation}\label{Def:Dom}
\begin{aligned}
\Dom :=\{x=r(\sin\theta,\cos\theta) \in \R^2, \theta \in (0,\omega), r > 0\} && \pour \omega \in (0, 2\pi],
\end{aligned}
\end{equation}
on which the following elliptic equation with Dirichlet boundary conditions is set\footnote{To keep notation light, we always omit the parentheses after the divergence operator $\nabla \cdot$ and, with a slight abuse of notation, denote $\nabla \cdot a g$ for $\nabla \cdot (a g)$.} :
\begin{equation}
\label{Divagrad-0-bis}
\lt\{
\begin{aligned}
&-\nabla \cdot a\lt(\frac{\cdot}{\epsilon} \rt) \nabla \ueps = \nabla \cdot f && \dans \Dom,
\\
&u=0 &&\sur \partial\Dom.
\end{aligned}
\rt.
\end{equation}
Notice that the domain $\Dom$ is not convex if $\omega > \pi$.
Here, $a$ is a heterogeneous elliptic coefficient field that represents physical characteristics of the material under consideration and $f$ is a compactly supported forcing term.
In \eqref{Divagrad-0-bis}, we explicitly write the ratio $\epsilon \ll 1$ between the small scale on which $a(\cdot/\epsilon)$ varies and the large scale, $1$, on which $f$ varies.
Rigorously speaking, \eqref{Divagrad-0-bis} also requires conditions at infinity, which may conveniently be fixed by imposing $\nabla \ueps \in \LL^2(\Dom)$ (or by assuming a growth rate of $\ueps$).
However simple this model is, it may accurately represent various physical phenomena in a domain made of a heterogeneous material displaying a corner.
Just to cite a few concrete examples : a thermal equilibrium in a sharp-edged artifact built by additive manufacturing, a mechanical equilibrium around a crack in a composite material (see, \textit{e.g.}, \cite{Hossain_2014} for fractures in heterogeneous media), electrostatics in a tetrahedral object made of metamaterials \cite{BBCCC16} --all problems which are of wide interest among physicists and engineers.
Moreover, there are incentives to develop a refined understanding of what happens near the tip of the corner : for example, in mechanics, fractures are often initiated by notches or cracks.

Even when $\Dom$ is a smooth bounded domain, equation \eqref{Divagrad-0-bis} is usually difficult to solve numerically.
Indeed, the microscopic scale on which the coefficient field $a$ oscillates is far smaller than the macroscopic scale, that is typically fixed by the scale on which the forcing term $f$ varies (in other cases, this may be fixed by the size of the domain $\Dom$ or the scale on which the boundary conditions vary).
In other words, seen from the macroscopic scale, the coefficient $a(\cdot/\epsilon)$ and therefore $\nabla \ueps$ oscillate at a high frequency.
Hence, numerical brute-force approaches are heavy and may even prove intractable.
This motivated the rise of the theory of homogenization :
Under suitable assumptions on the coefficient $a$ --including periodicity or ergodic stationarity-- it is possible to approximate the solution $\ueps$ of \eqref{Divagrad-0-bis} without solving frontally the equation.
Indeed, in these cases, there exists a constant (or homogeneous) coefficient $\abar$ independent of $f$ such that the solution $\ubar$ of
\begin{equation}
\label{Divagrad-1}
\lt\{
\begin{aligned}
&-\nabla \cdot \abar \nabla \ubar = \nabla \cdot f && \dans \Dom,
\\
&\ubar=0 &&\sur \partial\Dom,
\end{aligned}
\rt.
\end{equation}
well approximates the macroscopic behavior of $\ueps$.
Moreover, the oscillating function $\nabla \ueps$ can be approximated by means of the so-called 2-scale expansion, that we will discuss later on.
For the classical theory of homogenization, we refer to \cite{Allaire, JKO, Tartar}.

From a theoretical point of view, one of the major difficulties of \eqref{Divagrad-0-bis} is that it does not enjoy strong regularity estimates that are uniform in the limit $\epsilon \downarrow 0$.
By this,
we mean that,
in general, one cannot hope for estimates that are drastically better than the energy estimates provided by the Lax-Milgram theory --even though a slight gain in integrability is obtained by means of the Meyers estimate \cite{Meyers_Estim}.
One major finding of Avellaneda and Lin \cite{AvellanedaLin} is the following :
In the case where the domain $\Dom$ is smooth, if the coefficient $a$ is periodic, then the equation \eqref{Divagrad-0-bis} inherits some regularity properties from the constant-coefficient equation \eqref{Divagrad-1}; in particular, Lipschitz regularity.
This is a key ingredient in previous investigations of the fine regularity properties of \eqref{Divagrad-0-bis}; \textit{e.g.} to obtain estimates on the Green function \cite{KLSGreenNeumann}, $\LL^p$ estimates \cite{Shen}, or annealed estimates \cite{Marahrens_Otto}.
These results may then, in turn, be used to get error estimates on the approximation of $u_{\epsilon}$ via the 2-scale expansion.
Moreover, regularity properties are also inherently related to the kernel of the operator $-\nabla \cdot a \nabla$;
therefore, they naturally yield Liouville principles \cite{CRAS_AvellanedaLin}.
Extending the Lipschitz regularity result of Avellaneda and Lin to random settings was, for a long time, out of reach and has only recently been managed by Armstrong and Smart \cite{Armstrong} (see also \cite{Armstrong_book_2018})  and Gloria, Neukamm and Otto \cite{Gloria_Neukamm_Otto_2019}.
This strategy for obtaining error estimates has been used in various settings : see, \textit{e.g.}, \cite{KLSGreenNeumann, Shen} for the periodic setting, \cite{Bella_Giunti_Otto_2017, DuerinckxGO_2016} for the stochastic setting, \cite{Article_Homog} for the case of defects in a periodic medium, and \cite{Josien_InterfPer_2018, JosienRaithel_2019} for the case of a flat interface.

In this article, we adapt Avellaneda and Lin's Lipschitz regularity results to the case of a random coefficient $a$ with the equation \eqref{Divagrad-0-bis} set on the sectoral domain $\Dom$ defined by \eqref{Def:Dom};
as a corollary, we establish a Liouville principle and an error estimate for a nonstandard  2-scale expansion that is adapted to our framework.
This plan requires us to build \textit{corner correctors}, that differ from the usual homogenization correctors in various aspects (notably, they scale according to a non-integer exponent).
A somewhat circular difficulty that is inherent in our strategy is that well-behaved correctors are crucial for establishing regularity properties, which, in turn, are required in order to build correctors.
We overcome this issue by observing that for our ansatz, it is not necessary to use the full power of the usual large-scale regularity result in order to construct the corner corrector.
Instead, we use H-convergence to get a deterministic version of the large-scale regularity result and use a weak version of this in order to construct the corner correctors, and also to obtain quasi-optimal growth rates (in a spatial norm that is weaker than in our final result).
By analogy with the classical whole-space case, this amounts to using the weaker large-scale $\CC^{0,\alpha}$ estimate --which does not involve the classical correctors-- as opposed to the large-scale Lipschitz estimate.
We then consolidate the large-scale regularity result to get $\LL^\infty$-like annealed estimates, and derive the desired quasi-optimal estimates on the corner correctors (in the stronger spatial norm).
Lastly, these are used to get quasi-optimal error estimates on our (nonstandard) 2-scale expansion.
This strategy simplifies the approach used in  \cite{Fischer_Raithel_2017, Raithel_2017,JosienRaithel_2019}; in particular, we are able to avoid the iterative construction of the corner corrector on increasingly large scales.

\subsection{Interplay of microstructure and macroscopic geometries}

We now situate the current contribution within the broader perspective of bridging the gap between the macroscopic geometry of a given object and the microscopic structure of the underlying materials.

We first discuss the case of a smooth domain.
There, the interplay between a boundary that is smooth on the macroscopic level but with small-scale oscillations (either related to the boundary itself or caused by oscillating boundary conditions), has attracted much attention; notably, in the periodic setting, \textit{cf.} \cite{KenigPrange_2015, GerardVaretMasmoudi, ArmstrongKuusiMourratPrange}.
The stochastic setting remained unexplored for a long time, until the work of Fischer and the second author \cite{Fischer_Raithel_2017, Raithel_2017}, who proposed a way to build correctors adapted to Dirichlet or Neumann boundary conditions, but with suboptimal estimates.
In \cite{BellaFischerJosienRaithel}, we go further, showing quasi-optimal estimates as well as adapting recent results of stochastic homogenization to the case of a flat boundary (such as bounds on the fluctuations). 

The case of a flat interface between two different heterogeneous materials has been studied in the context of periodic homogenization by means of asymptotic expansion in \cite[Chap.\ 9 p.\ 312]{Bakhvalov_Panasenko}, but saw a recent renewal of interest initiated by \cite{BLLcpde}.
Solving an issue raised in \cite{BLLcpde}, the first author proposed in \cite{Josien_InterfPer_2018} a generalization of the notion of correctors and of the 2-scale expansion to the special case of the interface between two periodic heterogeneous materials.
Once suitable estimates were proved for such correctors, adaptations of the classical results followed (\textit{e.g.} Lipschitz regularity, quasi-optimal error estimates, an expansion of the heterogeneous Green function and estimates for such).
Taking advantage of the approach of \cite{Fischer_Raithel_2017}, we extended these results in \cite{JosienRaithel_2019} to heterogeneous media satisfying very general assumptions (including classical random media).
Very recently, inspired mainly by techniques from \cite{Shen_Boundary_2017} (see also \cite{Shen_Zhuge_2017, Armstrong_Shen_2016, Armstrong}), an alternative approach in the periodic setting with a $C^{k,\alpha}$ interface was proposed in \cite{Zhuge_2020}. 
Both \cite{JosienRaithel_2019} and \cite{Zhuge_2020} take advantage of $C^{1,\alpha}$ estimates for composite domains
--a problem first considered in \cite{Vogelius_2000} and then in \cite{LiNirenberg_2003}, in which (assuming a $C^{1,\beta}$ interface) $C^{1,\alpha}$ estimates were proved for $\alpha \leq \beta / (2 (1+\beta))$ (see also \cite{Neukamm_2019} for an extension of \cite{Vogelius_2000,LiNirenberg_2003} to nonlinear elliptic systems with applications to nonlinear elasticity).

To the best of our knowledge, the interplay between corners and quantitative homogenization is mainly uncharted, even in the framework of periodic homogenization.
However, we underline that this question has been studied in more qualitative or empirical ways; \textit{e.g.}, recently, the question of fracture propagation in heterogeneous media, which combines the problems of geometric singularities and heterogeneous materials, saw a renewal of interest in the mechanics community \cite{Hossain_2014, Schneider_2020_Fracture, lebihain_2020}.
 In the mathematical literature, we are only aware of a few articles dealing with very specific frameworks :
In \cite{Dauge_2006}, a corner made of a homogeneous material coated with a thin skin of another homogeneous material was studied and an asymptotic expansion for the solution was proposed.
In contrast to the present article, the only multi-scale feature of \cite{Dauge_2006} was related to the width of the skin coated on the corner.
However, as will be seen in Section \ref{Sec:GenDisc}, there are similarities between the expansion in \cite{Dauge_2006} and the adapted 2-scale expansion that we propose here.
There are also works concerning the Dirichlet problem in planar sectors with holes of size $\epsilon$ (possibly shrinking towards the corner), \textit{cf.}  \cite{DTV_2009, Costabel_2017} or \cite[Chapter 2]{MNP_2000}.

\subsection{General discussion of the corner correctors and the nonstandard 2-scale expansion}\label{Sec:GenDisc}

Before introducing our precise results and giving a detailed account of the strategies specific to this paper, we first proceed to a general discussion of the ideas used here and in \cite{Josien_InterfPer_2018,JosienRaithel_2019} to tackle situations where both the microscopic and the macroscopic structures of the material come into play.
In these cases, as previously discussed, classical approaches fail, and we are required to tailor nonstandard 2-scale expansions to the settings in question.
We hope that a more general exposition of this process is of interest to a practitioner, who might face situations different from the ones exposed here or in \cite{Josien_InterfPer_2018,JosienRaithel_2019}, but which may be treated with similar methods.
(In particular, we hope our approach may be used for mechanical fractures in heterogeneous materials, and it would also be interesting to investigate waves hitting a polyhedric obstacle made of a metamaterial \cite{BBCCC16}.)
For ease of the exposition, we purposely choose to be rather general in this section, the precise results being written in the sequel.

Let us first discuss the requirements on the coefficient $a$.
For simplicity, it is comfortable to assume that this coefficient field is actually defined on the whole-space $\R^d$ and then restricted to the corner domain $\Dom$.
Our main assumption is that the whole-space coefficient field $a$ enjoys two desirable properties from the point of view of (classical) homogenization.
First, $a$ is uniformly elliptic and bounded.
Second, $a(\cdot/\epsilon)$ H-converges to a constant matrix $\abar$ when $\epsilon\downarrow0$
(again, for simplicity, w.~l.~o.~g. we assume henceforth that $\bar{a}=\Id$), and this convergence is monitored by means of the control of the growth rate of the extended correctors\footnote{The \textit{extended correctors} are both the usual correctors and the flux-correctors, \textit{cf.} Assumption \ref{AssumpId} below.}.
These assumptions apply to most homogenization settings (see Section \ref{Sec:Assump} for examples).
Coming from this perspective, the underlying message of our results is as follows :
Quantitative homogenization results can be transferred from the whole-space to a corner domain.

The cornerstone of the approach pioneered by Avellaneda and Lin \cite{AvellanedaLin}, is the transfer of regularity properties from the homogeneous problem \eqref{Divagrad-1} to the heterogeneous problem \eqref{Divagrad-0-bis}.
Hence, emulating this method, the first step naturally requires us to understand the regularity properties of \eqref{Divagrad-1}.
In the easier case of $f=0$ near $0$,\footnote{The more general case of $f$ being nonzero near $0$ involves other singular functions that we disregard here. An interested reader may consult \cite[Th.\ (5.11)]{Dauge_2} for the Laplacian.} the regularity properties may be encapsulated in form of an expansion of the solution $\ubar$ into
\begin{equation}
\label{Decompose_ubar}
\ubar(x) = \ubar_{\reg}^N(x) + \sum_{n=1}^N \bar{\gamma}_n \taubar_n(x),
\end{equation}
where $\ubar_{\reg}^N$ is a regular contribution that is small near $0$, whereas the rightmost sum is a linear combination of singular $\abar$-harmonic functions $\taubar_n$; \textit{i.e.}, $\taubar_n$ solves
\begin{equation}\label{Def:taubar}
\begin{cases}
-\nabla \cdot \abar \nabla \taubar_n=0 & \dans \Dom,
\\
\taubar_n=0 & \sur \partial \Dom.
\end{cases}
\end{equation}
As will be seen in Section \ref{Sec:ClassRes}, these $\taubar_n$ (or rather, their derivatives) display a singularity at the tip of the corner and scale like $|x|^{\rhobar_n}$, where $\rhobar_n =\frac{n\pi}{\omega}$.
The larger $N \in \mathbb{N}$ is, the more regular $\ubar_{\reg}^N$ is.
Somehow, \eqref{Decompose_ubar} is analogous to a Taylor expansion for $\abar$-harmonic functions (both actually coincide when considering non-singular geometries, such as the bulk of a regular domain or a flat boundary).
As far as a Liouville principle goes : Any $\abar$-harmonic function in $\Dom$ with an algebraic growth rate at infinity is a linear combination of functions $\taubar_n$, for $n \in \{1,\cdots,N\}$, with $N$ being determined by the order of the growth rate.

Transitioning now to the heterogeneous problem, having in mind the expansion \eqref{Decompose_ubar}, we also expect such a decomposition to hold on the level of \eqref{Divagrad-0-bis}.
However, rather than involving $\abar$-harmonic functions, we shall make use of $a$-harmonic functions.
This gives rise to the corner correctors $\phiC_n$, which are (non-compact) small-scale perturbations of $\taubar_n$.
In particular, these turn the $\abar$-harmonic functions $\taubar_n$ into the $a$-harmonic functions $\taubar_n + \phiC_n$; namely, $\phiC_n$ solves
\begin{equation}\label{Def:psi}
\begin{cases}
-\nabla \cdot a( \nabla \phiC_n + \nabla \taubar_n)=0 & \dans \Dom,
\\
\taubar_n + \phiC_n=0 & \sur \partial \Dom.
\end{cases}
\end{equation}
Notice that \eqref{Def:psi} becomes the classical equation for the usual whole-space homogenization corrector, $\phi_i$, if we replace $\Dom$ by $\R^2$ and $\taubar_n$ by the coordinate map $x \mapsto x_i$.
In analogy with the homogeneous Liouville principle, we prove that the functions $\taubar_n + \phiC_n$ for $n \leq N$ span the space of $a$-harmonic functions in $\Dom$ with a prescribed growth rate (depending on~$N$).
Also, we may replace \eqref{Decompose_ubar} by the following decomposition of the solution $\ueps$ of \eqref{Divagrad-0-bis}, and establish that
\begin{equation}\label{Decompose_u}
\ueps(x) =  v(x) + \sum_{n=1}^N \gamma_n \lt(\taubar_n(x)+\epsilon^{\rhobar_n}\phiC_n\lt(\frac{x}{\epsilon}\rt)\rt),
\end{equation}
where $v$ is small near $0$.
The nonstandard terms of this 2-scale expansion, \textit{i.e.} those that scale like $\epsilon^{\rhobar_n}$ (typically $\rhobar_n \notin \mathbb{Z}$), were also observed within the specific context of \cite{Dauge_2006}.
These are incompatible with the classical asymptotic expansion (see, \textit{e.g.} \cite[Sec.\ 1.1]{Allaire}), in which one \textit{a priori} postulates that the solution $\ueps$ of \eqref{Divagrad-0-bis} may be expanded as
\begin{align}\label{Usual}
\ueps(x) = \sum_{n=0}^\infty \epsilon^n u_n\left(x,\frac{x}{\epsilon}\right),
\end{align}
and then successively recovers equations for the $u_n$.

If we wish for an approximation of $\ueps$ in the full domain $\Dom$ and not only near the tip of the corner,
we shall also take into consideration the oscillations in the bulk of the sectoral domain.
This leads to the following hybrid 2-scale expansion\footnote{In \eqref{2scale}, we make the choice $\gamma_n :=\bar{\gamma}_n$, for $\bar{\gamma}_n$ appearing in \eqref{Decompose_ubar}.} :
\begin{equation}\label{2scale}
\utieps^N(x) := \lt(1+\epsilon \phiD_i\lt(\frac{x}{\epsilon}\rt) \partial_i\rt) \ubar_{\reg}^N(x) + \sum_{n=1}^N \gamma_n \lt(\taubar_n(x)+\epsilon^{\rhobar_n}\phiC_n\lt(\frac{x}{\epsilon}\rt)\rt),
\end{equation}
where the usual Dirichlet correctors $\phiD_i$ also appear.
These  solve
\begin{equation}\label{Def:phiD}
\begin{cases}
-\nabla \cdot a \nabla ( \phiD_i + x_i)=0 & \dans \Dom,         \\
\phiD_i=0                         & \sur \partial \Dom.
\end{cases}
\end{equation}
Notice that in \eqref{2scale} we have used the Einstein summation convention for the indices $i$; 
we will do so throughout the rest of the article without further notice.
In \eqref{2scale}, the leftmost part $(1+\epsilon \phiD_i (\cdot /\epsilon) \partial_i ) \ubar_{\reg}^N$ is the classical 2-scale expansion, but only applied to the regular part $\ubar_{\reg}^N$ of $\ubar$, whereas the rightmost part $\sum_{n=1}^N \gamma_n \lt(\taubar_n+\epsilon^{\rhobar_n}\phiC_n\lt(\cdot/ \epsilon\rt)\rt)$ deals with the singularities induced by the corner.
We prove that $\nabla\utieps$ is an accurate approximation for $\nabla \ueps$ and also illustrate this through means of numerical simulations.
The 2-scale expansion \eqref{2scale} is an important motivation for precisely quantifying the growth rate of the corner correctors.

\begin{remark}[Comparison with \cite{Dauge_2006}]
Notice that \eqref{2scale} is close to \cite[(1.7)]{Dauge_2006}.
Indeed, the latter employs profiles denoted $\mathfrak{K}$ (see \cite[(P$\infty$)]{Dauge_2006}) which are nothing but our corrected singularities $\taubar_n + \phiC_n$. 
\end{remark}

\subsection{Summary}

In Section \ref{main_results} we first review in more detail some results from the theory of harmonic functions on sectoral domains.
Then, we give and discuss our assumptions, which requires us to also further discuss some aspects from the classical theory of homogenization.
Section \ref{Sec:Res} is devoted to the exposition of our main results : a quenched large-scale regularity theory (Theorem \ref{ThAL}) and associated Liouville principle (Corollary \ref{ThLiouville}), the existence of corner correctors that satisfy an annealed quasi-optimal pointwise growth rate (Theorem \ref{Th:OptiGR}), and, lastly, an $\LL^\infty$-like annealed error estimate for the 2-scale expansion given in \eqref{2scale} (Theorem \ref{Th:Error}) \footnote{See Section \ref{sec:quasi} for a discussion of \textit{quenched} vs. \textit{annealed} estimates.}.
We end Section \ref{main_results} with numerical simulations, which suggest that in our setting of the sectoral domain the adapted (nonstandard) 2-scale expansion \eqref{2scale} significantly out-performs the standard (with Dirichlet correctors) 2-scale expansion.
Section \ref{strategy} contains a detailed overview of the strategy that we use to obtain our main results.
In Section \ref{proof_lemma3.3} we prove Lemma \ref{Lem:heterog}, which is a purely deterministic version of Theorem \ref{ThAL}.
Section \ref{proof_prop_3.4} is devoted to the argument for Proposition \ref{Prop:Quench}, which, under deterministic assumptions on the extended whole-space corrector and the extended half-space correctors associated to the ``edges'' of the sectoral domain, constructs the corner correctors and gives the quasi-optimal growth rate, but in a weaker spatial norm, of Theorem \ref{Th:OptiGR}.
In Section \ref{main_thms} we use Lemma \ref{Lem:heterog} and Proposition \ref{Prop:Quench} in order to prove Theorems \ref{ThAL} and \ref{Th:OptiGR}, respectively.
Finally, Section \ref{error_est_sec} contains our argument for the quasi-optimal error estimate, Theorem \ref{Th:Error}, which, in particular, also requires the construction of \textit{corner flux correctors} in Lemma \ref{Lem:sig}.

\section{Main results}
\label{main_results}

\subsection{Classical theory for the homogenized problem}\label{Sec:ClassRes}

For simplicity, we introduce the following geometrical notations : $\Boule_R(x)$ is the ball of radius $R$ centered at $x$, $\Gamma :=\partial\Dom$, $\Dom_R(x) :=\Dom \cap \Boule_R(x)$, $\Gamma_R(x) :=\Gamma \cap \Boule_R(x)$ (which differs from $\partial \Dom_R(x)$), in which we omit the parameters $x$ if $x=0$ and $R$ if $R=1$.
(For the ambiguous case $x=0$ and $R=1$, we still write explicitly the ``$1$'' as in $\Dom_1 = \Dom_1(0)$.)
We define the smooth cut-off function
\begin{equation}\label{eta_rescale}
	\eta_{\Boule,R} :=\eta_{\Boule,1}(R\cdot)	
\end{equation}
where $\eta_{\Boule,1}\in \CC_\comp^\infty(\Boule_1,[0,1])$ satisfies $\eta_{\Boule,1}\equiv1$ in $\Boule_{1/2}$ and $|\nabla\eta_{\Boule,1}|\lesssim1$ (in Figure~\ref{Fig :Geom} the set $\{ x \in \Dom : \eta_{\Boule,R}(x) = 1\}$ is shown in violet).

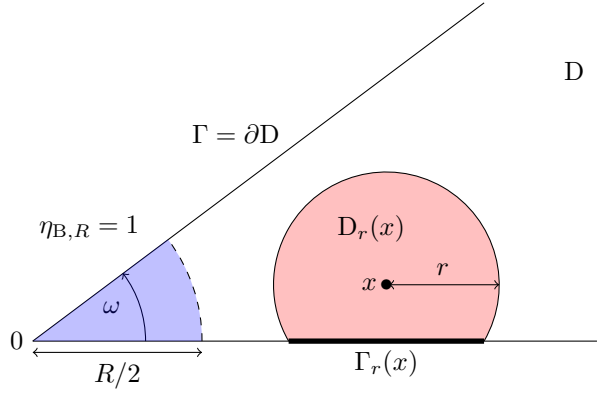
\begin{figure}[h]
\begin{center}
\begin{tikzpicture}[scale=1.5]
\coordinate (O) at (0,0);
\coordinate (A) at (5,0);
\coordinate (a) at (1,0);
\coordinate (B) at (4,3);
\coordinate (Oprime) at (1.895,0.635);
\coordinate (Aprime) at (4.9,0.635);
\coordinate (Bprime) at (4.295,2.442);
\coordinate (atilde) at (1.5,0);

\draw (O)--(A);
\draw (O)--(B);
\draw[->] (a) arc (0 :37 :1);
\draw (O) node [left] {$0$};
\draw (0.7,0.3) node {$\omega$};
\draw (4.8,2.4) node {${\rm D}$};
\draw (1.8,2) node[below] {$\Gamma=\partial {\rm D}$};
\fill[color=blue, opacity=0.25] (O)--(atilde) arc(0 :37 :1.5)--(O);

\draw[dashed] (atilde) arc(0 :37 :1.5);
\draw (0.5,1) node {$\eta_{\Boule,R}=1$};
\draw[<->] (0,-0.1)--(1.5,-0.1);
\draw (0.75,-0.3) node {$R/2$};

\fill[color = red, opacity = 0.25] (4,0) arc  (-30 :210 :1) -- cycle;
\draw (4,0) arc  (-30 :210 :1);

\draw ({4-sqrt(3)/2},0.5) node {$\bullet$};
\draw ({4-sqrt(3)/2},0.5) node[left] {$x$};
\draw[<->] ({4-sqrt(3)/2+0.025},0.5)--({5-sqrt(3)/2},0.5);
\draw ({4.5-sqrt(3)/2},0.5) node [above]{$r$};
\draw (3,1) node {$\Dom_r(x)$};
\draw[line width = 2] ({4-sqrt(3)},0)--(4,0);
\draw ({4-sqrt(3)/2},0) node[below] {$\Gamma_r(x)$};

\end{tikzpicture}
\end{center}
\caption{Geometrical setting}\label{Fig :Geom}
\end{figure}

Since the operator $\nabla \cdot \abar \nabla$ only depends on the symmetric part of $\abar$,
 thanks to a change of coordinate, we may w.~l.~o.~g.\ assume that $\abar=\Id$.
The associated family of $\abar$-harmonic functions $\taubar_n$ is defined in polar coordinates by
\begin{equation}\label{Def:taubar0}
\taubar_n(x) :=r^{\rhobar_n} \sin( \rhobar_n \theta)
\qquad \pour \rhobar_n :=\frac{n\pi}{\omega},
\end{equation}
whereby verifying \eqref{Def:taubar} is immediate.
These are associated to ``dual'' functions\footnote{See \cite[(2.1)]{Dauge_CRAS_1987} and \textit{erratum} \cite[p.\ 346]{Dauge_2}.}, that are $\abar$-harmonic in $\Dom \backslash \{0\}$
\begin{equation}\label{Def:taubarstar0}
\taubar^\star_n(x) :=-r^{-\rhobar_n} \sin( \rhobar_n \theta).
\end{equation}
The functions $\taubar_n$ and $\taubar^\star_n$ are homogeneous in the sense of
$\taubar_n(\lambda x) = \lambda^{\rhobar_n} \taubar_n(x)$ and $\taubar^\star_n(\lambda x) = \lambda^{-\rhobar_n} \taubar^\star_n(x)$.
Therefore, we obviously have
\begin{align}\label{taubar_Homog}
|\nabla^l \taubar_n(x)| \lesssim_\omega |x|^{\rhobar_n-l} \qquad \et \quad |\nabla^l \taubar^\star_n(x)| \lesssim_\omega |x|^{-\rhobar_n-l},
\end{align}
where the symbol $\lesssim_{\beta}$, for a tuple of constants $\beta$, reads ``$\leq C$ for a constant $C$ depending only on $\beta$''.

The functions $\taubar_n$ are the building blocks for decomposing the solution $\ubar$ of \eqref{Divagrad-1}.
The below theorem is a consequence of results from the literature \cite[Th.\ 2 \& Th.\ 3]{Dauge_CRAS_1987} and \cite[Th.\ (5.11)]{Dauge_2} :
\begin{theorem}[See \cite{Dauge_CRAS_1987, Dauge_2}]
\label{Th:homog}
Assume that $\ubar$ satisfies
\begin{equation}
\label{Divagradbar_1}
\begin{cases}
-\Delta \ubar = 0 &\qquad \dans \Dom_1,
\\
\ubar = 0 &\qquad \sur \Gamma_1.
\end{cases}
\end{equation}
Then, for any\footnote{We underline that we use the convention $0 \in \mathbb{N}$. In the case $N=0$, no function $\taubar_n$ is used in \eqref{Decompose_ubar}. \label{Note1}} $N \in \NN$, $\ubar$ can be decomposed as in \eqref{Decompose_ubar}, 
where $\bar{\gamma}_n$ is given by
\begin{equation}
\label{Formula :gamma}
\bar{\gamma}_n := \frac{1}{n \pi} \int_{\Dom_1} \Delta \lt(\eta_{\Boule,1} \ubar\rt) \taubar^\star_n,
\end{equation} 
and where $\ubar_{\reg}^N$ satisfies
\begin{equation}
\label{Num:001}
\lt(\fint_{\Dom_R} \lt|\ubar_{\reg}^N\rt|^2 \rt)^\frac{1}{2}\lesssim_{\omega,N} R^{\rhobar_{N+1}} \lt( \fint_{\Dom_1} |\ubar|^2 \rt)^{\frac{1}{2}} \qquad \text{provided} \quad R\leq \frac{1}{2}.
\end{equation}
\end{theorem}

Theorem \ref{Th:homog} still holds if we replace $\eta_{\Boule,1}$ 
by any function $\eta$ with compact support in $\Boule_1$, with $\eta =1$ in a neighborhood of $0$.
Relation \eqref{Formula :gamma} may also be written as
\begin{equation}\label{Formula :gamma3}
\bar{\gamma}_n=-\frac{1}{n \pi} \int_{\Dom}  \ubar\lt(2 \nabla \taubar^\star_n \cdot \nabla\eta_{\Boule,1} +\taubar^\star_n \Delta \eta_{\Boule,1}  \rt).
\end{equation}

\subsection{Our assumptions}\label{Sec:Assump}

In the sequel, we assume that the (random) coefficient field $a : \R^2 \rightarrow \R^{2 \times 2}$ is generated by an ensemble.
We denote both the ensemble and its associated expectation by $\langle \cdot \rangle$.
Below are the assumptions on the coefficients fields $a$ generated by $\langle \cdot\rangle$ :

\begin{enumerate}[label=\textbf{A.\arabic*}]
\item{
\label{Assump1}
$a$ is symmetric, uniformly $\lambda$-elliptic, and bounded. Namely, there holds
\begin{align}
\label{Ellipticite}
\xi \cdot a(x) \xi \geq \lambda |\xi|^2  \quad\et \quad \xi \cdot a^{-1}(x) \xi \geq |\xi|^2 \qquad \pourtout x, \xi \in \R^2.
\end{align}
}
\item{
\label{AssumpId}
$a$ admits the following decomposition :
\begin{align}\label{Decomposition}
a e_i&=\abar e_i- a \nabla \phi_i + \nabla \cdot \sigma_i \dans \R^2 \qquad \text{ that is }\qquad a_{ji} = \abar_{ji} - a_{jl}\partial_l \phi_i +  \partial_k \sigma_{ijk},
\end{align}
where $\abar=\Id$ is the homogenized matrix, the functions $\phi_i$ are the correctors, and the skew-symmetric fields $\sigma_i$ are the flux correctors.
The decomposition \eqref{Decomposition} is the central equation for defining the extended correctors $(\phi,\sigma)$. 
On the one hand, \eqref{Decomposition} immediately implies the familiar equation for the correctors
\begin{equation}\label{Eq :Corr}
\nabla \cdot a (\nabla \phi_i + e_i)=0 \dans \R^2.
\end{equation}
On the other hand, \eqref{Decomposition} determines the flux correctors only up to a gauge, which we prescribe by setting
\begin{equation}\label{Def:sigma*}
\sigma_{ijk} = \partial_k N_{ji} - \partial_j N_{ki} \pour \Delta N_{ji} = a_{jl}\lt( \delta_{li} + \partial_l \phi_i \rt) - \abar_{ji} \dans \R^2.
\end{equation}
\item{
\label{Assump2}
There exist exponents $\nu \in (0,1]$ and $\tilde{\nu} \geq 0$ such that the extended correctors\footnote{By convention, we denote by $\phi$ the first-order tensor $(\phi_i)_i$ and by $\sigma$ the third-order tensor $(\sigma_{ijk})_{i, j, k}$.} $(\phi,\sigma)$ are strictly sublinear with overwhelming probability in the sense of
\begin{equation}
\label{CorrSubDef} 
\begin{aligned}
&\displaystyle \sup_{x,y \in \R^2, |x-y| \leq r} \langl \lt|(\phi,\sigma) (x)- (\phi,\sigma) (y) \rt|^{p}  \rangl^{\frac{1}{p}}
\leq c_p r^{1-\nu} \ln^{\tilde{\nu}}(r+2)
\qquad \pourtout p \in [1,\infty) \et r > 0.
\end{aligned}
\end{equation}
}
}
\end{enumerate}

Notice that we may freely add constants to the extended correctors.
Therefore, for simplicity, we henceforth set
\begin{align}\label{Gauge :1}
\phi(0)=0 \et \sigma(0) = 0.
\end{align}
By Jensen's inequality, we may w.~l.~o.~g.\ assume that the constant $c_p$ in \eqref{CorrSubDef} is nondecreasing in $p \in [1,\infty)$. For simplicity, we denote
\begin{equation*}
\Xi :=\lt(\lambda,\omega,\nu,\tilde{\nu}, c\rt),
\end{equation*}
and we will not strive for the precise dependence in $\Xi$ of our estimates.
Notice that $c$ itself is a function (depending on $p$) and not a number; 
however, we use the notation $\lesssim_{\Xi,\beta}$, where $\beta$ is a tuple of constants, for designating $\lesssim_{\lambda,\omega,\nu,\tilde{\nu},c_p,\beta}$ in which the exponent $p$ itself depends on $\lambda$, $\omega$, $\nu$, $\tilde{\nu}$ and $\beta$.

\smallskip

These assumptions can be relaxed --see Section~\ref{Sec:Discuss} for a discussion on that matter.
However, we underline that Assumption \ref{Assump2} is central for quantifying the H-convergence of $a$ to $\abar$.
The somehow arbitrary form of the \rhs of \eqref{CorrSubDef} is dictated by the frameworks that we have in mind : some are deterministic (in which case the ensemble $\langl\cdot\rangl$ is a point mass) and others are random.
For example, our assumptions are suitable for
\begin{itemize}
\item periodic coefficient fields \cite{AvellanedaLin} (with $\nu=1$ and $\tilde{\nu}=0$), possibly perturbed by a defect \cite[Th.\ 3.1 \& Th.\ 4.1]{BLLcpde}  (with $\nu \in (0,1]$ depending on the integrability of the defect and $\tilde{\nu}=0$),
\item random coefficient fields generated by ensembles satisfying the spectral gap estimate \cite[Prop.\ 4.1]{JosienOtto_2019} or a finite-range-of-dependence assumption \cite[Th.\ 4.1 p.\ 124]{Armstrong_book_2018}, for $\nu = 1 \et \tilde{\nu} = \frac{1}{2}$ (in dimension $2$).
\end{itemize}
Notice that the behavior of the coefficient field $a$ near the boundary of the domain $\Dom$ is irrelevant --up to some logarithmic losses.
Hence, as in \cite[(1.7)]{JosienRaithel_2019}, we may also apply all our results to a coefficient field $\tilde{a}$ defined as follows :
\begin{equation*}
\tilde{a}(x)=
\lt\{
\begin{aligned}
&a(x) && \si \dist(x,\Gamma)>1,
\\
&b(x) && \si \dist(x,\Gamma) \leq 1,
\end{aligned}
\rt.
\end{equation*}
where $a$ satisfies Assumptions \ref{Assump1}, \ref{AssumpId}, \ref{Assump2}, but $b$ only satisfies Assumption \ref{Assump1}.
Indeed, we make use of results from \cite{JosienRaithel_2019}, the proofs of which always involve cut-off functions which vanishes along the boundary, which \textit{de facto} wipe out the influence of the extended whole-space correctors $(\phi,\sigma)$ near the boundary.

In order to apply the classical Schauder theory for obtaining pointwise estimates (see Section \ref{Sec:Discuss} for a precise discussion), we occasionally assume that $a$ is uniformly H\"older continuous with overwhelming probability in the sense of
\footnote{For simplicity, even though \eqref{CorrSubDef} and \eqref{A:holder} are of very different natures, we use the same constant $c_p$ to control them.}

\begin{equation}\label{A:holder}
\sup_{x \in \Dom}\langl \lt\| a \rt\|_{\CC^{0,\alpha}(\Boule(x))}^p \rangl^{\frac{1}{p}} \leq c_p.
\end{equation}

\subsection{Statement of the main results}\label{Sec:Res}

\subsubsection{Adaptation of the Lipschitz regularity theory to corners}

First, we construct the corner correctors $\phiC_n$ --recall that these solve \eqref{Def:psi}.
Equipped with these, we may then generalize the Lipschitz regularity theory of Avellaneda and Lin \cite{AvellanedaLin} to the case of corners.
\begin{theorem}\label{ThAL}
Let $N \in \mathbb{N}$ and $p \in [1, \infty)$.\footnote{\textit{Cf.} note \ref{Note1}.}
Under Assumptions \ref{Assump1}, \ref{AssumpId}, and \ref{Assump2}, there exists a random constant $\Cstar \geq 1$ satisfying
\begin{equation}
\label{Borne_CN-}
\langl |\Cstar|^p\rangl \lesssim_{\Xi,N,p} 1 \pourtout p \in [1,\infty),
\end{equation}
such that the following property holds :
If $u$ satisfies
\begin{equation}\label{Eq :u-aharm}
\lt\{
\begin{aligned}
&-\nabla \cdot a \nabla u = 0 && \dans \Dom_R,
\\
& u = 0 && \sur \Gamma_R,
\end{aligned}
\rt.
\end{equation}
for $R \geq 1$, then there exist coefficients $\gamma_1, \cdots, \gamma_N \in \R$ satisfying
\begin{equation}\label{Estim :gamma2}
\lt|\gamma_n\rt|\leq \Cstar R^{-\rhobar_n+1} \lt(\fint_{\Dom_{R}} |\nabla u|^2\rt)^{\frac{1}{2}} \qquad \pour n \in \{1,\dots,N\},
\end{equation} 
and such that, for all $1 \leq r \leq R$, we have 
\begin{align}\label{Num:301_L2-1}
\lt( \fint_{\Dom_r} \lt| \nabla u - \sum_{n=1}^N \gamma_n \nabla\lt( \taubar_n + \phiC_n\rt) \rt|^2 
\rt)^{\frac{1}{2}}
\leq \Cstar \lt(\frac{r}{R}\rt)^{\rhobar_{N+1}-1}
\lt(\fint_{\Dom_R} |\nabla u|^2\rt)^{\frac{1}{2}}.
\end{align}
\end{theorem}

From Theorem \ref{ThAL}, we may actually obtain a local bound on any region of the domain $\Dom$.
Indeed, if we introduce a point $x \in \Dom_{R/2} \backslash \Dom_1$, a simple two-step procedure, relying first on Theorem \ref{ThAL}, replacing $r\rightsquigarrow 2|x|$, and then on the boundary Lipschitz regularity theory
\footnote{The boundary Lipschitz theory from \cite{Fischer_Raithel_2017} may be applied thanks to Assumptions \ref{Assump1}, \ref{AssumpId}, and \ref{Assump2}.} \cite[Th.\ 1 \& Th.\ 2]{Fischer_Raithel_2017}
 --in conjunction with a covering argument-- for jumping from $\Dom_{2|x|}$ to $\Dom_{1}(x)$, yields :
\begin{align}\label{Num:301_L2}
\lt( \fint_{\Dom_1(x)} \lt| \nabla u - \sum_{n=1}^N \gamma_n \nabla\lt( \taubar_n + \phiC_n\rt) \rt|^2 
\rt)^{\frac{1}{2}}
\leq \Cstar(x) \lt(\frac{|x|}{R}\rt)^{\rhobar_{N+1}-1}
\lt(\fint_{\Dom_R} |\nabla u|^2\rt)^{\frac{1}{2}},
\end{align}
for a random field $\Cstar(x)$ satisfying
\begin{equation}
\label{Borne_CN}
\langl |\Cstar(x)|^p\rangl \lesssim_{\Xi,N,p} 1 \qquad \pourtout x \in \Dom \et p \in [1,\infty).
\end{equation}

If, in addition, \eqref{A:holder} holds, then, thanks to classical Schauder theory \cite[Th. 5.19]{Giaquinta_2},  \eqref{Num:301_L2} can be upgraded to
\begin{align}\label{Num:301}
\lt| \nabla u(x) - \sum_{n=1}^N \gamma_n \nabla\lt( \taubar_n + \phiC_n\rt)(x) \rt|
\leq \Cstar(x) \lt(\frac{|x|}{R}\rt)^{\rhobar_{N+1}-1}
\lt(\fint_{\Dom_R(x)} |\nabla u|^2\rt)^{\frac{1}{2}}.
\end{align}

\begin{remark} 
We may prove that the coefficients $\gamma_n$ are close to the actual minimizers of the \lhs of \eqref{Num:301_L2-1}, \textit{cf.} Corollary~\ref{C:excessdecay}.
From this perspective, comparing the result of Theorem \ref{ThAL} to, \textit{e.g.}, \cite[Th. 2]{JosienRaithel_2019}, \cite[Th. 2]{Fischer_Raithel_2017}, or \cite[Th. 1]{Gloria_Neukamm_Otto_2019}, we notice that \eqref{Num:301_L2-1} can be recast as an excess-decay of a certain tilt-excess.
Defining the excess of $u$ on $\Dom_r$ as 
\begin{align}
\label{excess}
\Exc_N(u; r) :=\inf_{\gamma \in \R^N} \fint_{\Dom_r} \Big| \nabla u - \sum_{n=1}^N \gamma_n \nabla ( \taubar_n + \phiC_n)\Big|^2,
\end{align}
we see that squaring both sides of \eqref{Num:301_L2-1}, using that $\taubar_n + \phiC_n$ is $a$-harmonic, and taking an infimum over $\gamma \in \R^N$ yields that 
\begin{equation*}
\Exc_N(u;r) \leq \Cstar  \lt(\frac{r}{R}\rt)^{2(\rhobar_{N+1}-1)} \Exc_N(u;R).
\end{equation*}
\end{remark}

One important point of Theorem \ref{ThAL} is that the constants $\gamma_n$ are \textit{independent} of $r$ and $R$.
As will be seen in the proof, they only depend on $u$ in $\Dom_1$.
Thus, the coefficients $\gamma_n$ can be interpreted as a projection of $u$ on the singular functions $\taubar_n+\phiC_n$, for $n \in \{1,\dots,N\}$, from the scale $\rmax$ downwards.
Hence, as a key output of the above theorem, we may obtain a Liouville principle for decomposing functions that are $a$-harmonic in $\Dom$.\footnote{Rigorously speaking, we derive Corollary \ref{ThLiouville} from a deterministic version of Theorem \ref{ThAL}, \textit{cf.} Lemma \ref{Lem:heterog}.}

\begin{corollary}[Liouville principle]\label{ThLiouville}
Let $N \in \mathbb{N}$ and $\rho \in (\rhobar_{N},\rhobar_{N+1})$.
Under Assumptions \ref{Assump1}, \ref{AssumpId}, and \ref{Assump2},
if $u$ is a subalgebraic $a$-harmonic function in $\Dom$ in the sense of
\begin{equation}\label{Hypo :ThLiouville}
\begin{cases}
-\nabla \cdot a \nabla u=0 & \dans \Dom,
\\
u=0 & \sur \Gamma,
\end{cases}
\qquad \text{along with} \quad
\limsup_{R \uparrow \infty} R^{-\rho} \lt(\fint_{\Dom_R} |u|^2 \rt)^\frac{1}{2} < \infty,
\end{equation}
then, almost-surely, there exist coefficients $\gamma_1, \ldots, \gamma_N \in \R$ such that we may decompose $u$ as :
\begin{equation*}
u=\sum_{n=1}^N \gamma_n (\taubar_n+\phiC_n).
\end{equation*}
\end{corollary}

\subsubsection{Quasi-optimal growth rates for the corner correctors}
\label{sec:quasi}

Next, we derive growth rates on the corner and Dirichlet correctors that are optimal up to a logarithmic factor.

\begin{theorem}\label{Th:OptiGR}
Under Assumptions \ref{Assump1}, \ref{AssumpId}, and \ref{Assump2},
there exist corner correctors~$\phiC_n$ and Dirichlet correctors $\phiD$
(solving \eqref{Def:psi} and \eqref{Def:phiD}, respectively) satisfying
\begin{align}
\label{Opti_Decay}
&\langl\lt|\phiC_{n}(x)\rt|^p\rangl^{\frac{1}{p}}
\lesssim_{\Xi,p,n}
(|x|+1)^{\rhobar_{n}-\nu} \ln^{\tilde{\nu}'}(|x|+2)
&& \pourtout x \in \Dom \et n \in \NN\backslash\{0\},
\\
\label{Opti_Decay_Dir}
&\langl\lt|\phiD_i(x)\rt|^p\rangl^{\frac{1}{p}}
\lesssim_{\Xi,p}
(|x|+1)^{1-\nu} \ln^{\tilde{\nu}'}(|x|+2)
&& \pourtout x \in \Dom \et i \in \{1,2\},
\end{align}
for $p \in [1, \infty)$ and $\tilde{\nu}'\lesssim_{\tilde{\nu}} 1$.
\end{theorem}

The scaling $\rhobar_n-\nu$ of \eqref{Opti_Decay} is optimal in $|x|$.
Indeed, in the bulk of the corner domain, $\phiC_n$ is close to the correction of the
(classical, even without Dirichlet correctors) 2-scale expansion of $\taubar_{n}$, 
namely $\phi_i\partial_i\taubar_n$, which generically grows like $|x|^{\rhobar_n-\nu}$ by \eqref{CorrSubDef} and \eqref{taubar_Homog}. 
By contrast, the exponent $\tilde{\nu}'$ on $\ln|x|$  is not optimal;
since our proof cannot reach the optimal logarithmic exponents, 
we will not keep precise track of these. 
(In order to avoid technicalities, we may even unnecessarily increase this exponent in our estimates.)
However, a logarithmic loss with exponent at least $1$ is unavoidable in the case 
where there is a thin layer of fixed width in which the coefficient $a$ is replaced by another coefficient $\tilde{a}$ with no prescribed structure.
We refer to Section~\ref{Sec:Discuss} where such a case is mentioned and to \cite[Prop.\ 2.6]{JosienRaithel_2019} for a counter-example.

Let us comment on Theorem \ref{Th:OptiGR} in the simple case where the (moments of the) extended correctors are bounded\footnote{Namely, assume that $\nu=1$ and $\tilde{\nu}=0$ in \eqref{CorrSubDef}. 
See Section \ref{Sec:Assump} for examples.}.
Then there is a clear analogy between the scaling of the classical correctors $\phi_i$ in relation to the scaling of $x_i$ and the scaling of the corner correctors $\phiC_n$ with respect to the scaling of $\taubar_n$.
In both cases (up to some logarithmic factors), the correctors lose one order in $|x|$ in terms of growth at infinity with respect to the function that they correct.
Seen from the perspective of the 2-scale expansion \eqref{2scale}, 
this shows that the terms corresponding to the corner correctors scale like $\epsilon \ln^{\tilde{\nu}'}(\epsilon^{-1})$ (to be compared with $\epsilon$ in standard periodic homogenization).
As is usual, the (corner) correctors are necessary to approximate the gradient~$\nabla \ueps$ in strong norms (\textit{e.g.} $\LL^p$ for $p \in [1,\infty)$), but do not appear in the approximation of $\ueps$.

Notice that Theorem \ref{ThAL} and Corollary \ref{ThLiouville} are \textit{quenched} results in the terminology of \cite{Gloria_Neukamm_Otto_2019}.
In particular, they apply to a given realization of a coefficient field $a$, the extended correctors of which satisfy a deterministic analogue of estimate \eqref{CorrSubDef}, but with a random constant.
In contrast, Theorem \ref{Th:OptiGR} is an \textit{annealed} result, because \eqref{Opti_Decay} and \eqref{Opti_Decay_Dir} involve the $\LL^p_{\langle\cdot\rangle}$ moments.
Nevertheless, Theorem \ref{Th:OptiGR} can be turned into a quenched theorem by assuming that $\langl\cdot\rangl$ is a point mass.\\

\begin{remark}
Theorem \ref{Th:OptiGR} and Corollary \ref{ThLiouville} are not sufficient for obtaining the uniqueness of corner correctors~$\phiC_n$, 
for $n$ sufficiently large.
In particular, we may replace
\begin{align*}
\phiC_n \rightsquigarrow \phiC_n + \sum_{n'=1}^{n_0} \beta_{n'} \lt( \taubar_{n'} + \phiC_{n'}\rt),
\end{align*}
where $\rhobar_{n_0} \leq \rhobar_n - \nu$, and for coefficients $\beta_{n'} \in \R$, which also satisfies \eqref{Def:psi} and \eqref{Opti_Decay}.
Thus the coefficients $\gamma_n$ in Corollary \ref{ThLiouville} depend on the particular choice of corner correctors $\phiC_n$.
Practically, this implies that, for $n \geq 1 + \nu \frac{\omega}{\pi}$, the coefficients $\gamma_n$ in \eqref{2scale} may not be equal to the coefficients $\bar{\gamma}_n$ in \eqref{Decompose_ubar}.
\end{remark}

\begin{remark}[Dirichlet correctors and maximum principle]
If the extended correctors are bounded (in a quenched sense), 
we may directly appeal to the maximum principle (applied to the equation for the boundary correction) to get \eqref{Opti_Decay_Dir} with $\nu = 1$ and $\tilde{\nu}'=0$ ; 
however, such an approach is not tractable to \eqref{Opti_Decay}.
\end{remark}

\subsubsection{Error estimate for the nonstandard 2-scale expansion}

Our last result concerns the quality of the adapted 2-scale expansion \eqref{2scale}, which is quantified by means of an error estimate in an annealed $\LL^\infty$-like norm, \textit{cf.} \eqref{Num:7034} below.
Namely, we establish the following :
\begin{theorem}\label{Th:Error}
	Suppose that $\omega \in (0,2\pi)$.
	We place ourselves under the Assumptions \ref{Assump1}, \ref{AssumpId}, and \ref{Assump2}.
	\\
	Let $f \in \WW^{2,q}(\Dom)$, for $q > 2$, be supported in $\Dom_2 \backslash \Dom_{1}$ such that $\|f\|_{\WW^{2,q}(\Dom)} \leq 1$.
	Fix $\epsilon \leq 1/2$, and let $\ueps$ and $\ubar \in \HH^1(\Dom)$ solve
	\begin{equation}\label{Num:7001}
	\lt\{
	\begin{aligned}
	&-\nabla \cdot a\lt(\frac{\cdot}{\epsilon}\rt) \nabla \ueps = \nabla \cdot f = -\nabla \cdot \abar \nabla \ubar && \dans \Dom,
	\\
	&\ueps = 0 = \ubar && \sur \Gamma.
	\end{aligned}
	\rt.
	\end{equation}
	Let $\utieps^N$ be defined by \eqref{2scale} for $\gamma_n=\bar{\gamma}_n$, with $N>N_0$ and $N_0$ defined by
	\begin{equation}
	\label{Num:7029}
	N_0 :=\argmax\{N' \in \mathbb{N}, \rhobar_{N'}<\nu\}.
	\end{equation}
	Then, there exists $\tilde{\nu}'\lesssim_{\tilde{\nu}} 1$ such that for every $x\in \Dom_1 \backslash \Dom_{2\epsilon}$ and $p \in [1,\infty)$, there holds
	\begin{equation}\label{Num:7034}
	\langl \biggl(\fint_{\Dom_\epsilon(x)}| \nabla \utieps^N  - \nabla \ueps |^{2}\biggr)^\frac{p}2 \rangl^{\frac{1}{p}}
	\lesssim_{\Xi, q, p,N}\epsilon^\nu \ln^{\tilde{\nu}'}(\epsilon^{-1})|x|^{\rhobar_1-1}.
	\end{equation}
\end{theorem}
The \lhs of \eqref{Num:7034} may be understood as the $\LL^p_{\langl\cdot\rangl}$-moments of an $\LL^\infty$ norm \textit{down to scale $\epsilon$}.
However, we have employed $\LL^2$ estimates averaged on a domain of characteristic size $\epsilon$ on $\nabla \utieps^N  - \nabla \ueps$, instead of the pointwise estimates that would be expected for $\LL^\infty$ norms.
The reason is the following : we cannot reach the scales below $\epsilon$ because we have no extra regularity of the coefficient field $a$.
On the contrary, if we had assumed \eqref{A:holder}, we would have obtained the following estimate
\begin{align*}
\langl | \nabla \utieps^N(x)  - \nabla \ueps(x) |^p \rangl^{\frac{1}{p}}
\lesssim_{\Xi, q, p,N}\epsilon^\nu \ln^{\tilde{\nu}'}(\epsilon^{-1})|x|^{\rhobar_1-1} \pour |x|>\epsilon.
\end{align*}

Remark also that the \rhs of \eqref{Num:7034} blows up as $|x| \downarrow 0$ if $\omega > \pi$.
We do not know whether this rate is optimal or not for our choice of coefficients $\gamma_n := \bar{\gamma}_n$ in \eqref{2scale}, for $\bar{\gamma}_n$ defined in \eqref{Decompose_ubar}.
Indeed, the slightest deviation from the optimal choice of coefficients $\gamma_1$ in \eqref{2scale} comes with a term $\taubar_1+\phiC_1$ the gradient of which is singular in $0$ (with a rate $|x|^{\rhobar_1-1}$).
As a consequence, our result \eqref{Num:7034} might be optimal, but describing a non-optimal choice of $\gamma_n$.
Nevertheless, we underline that the choice $\gamma_n=\bar{\gamma}_n$ is reasonable and has the advantage of being easily tractable numerically -and indeed, it performs well, \textit{cf.} Section \ref{Sec:Num}.

\begin{remark}
	For a technical reason (in Lemma \ref{Lem:sig}), we cannot achieve the proof in the case $\omega = 2\pi$.
\end{remark}

\subsection{Numerical illustration}\label{Sec:Num}

For our numerical simulations, we define a domain $\tilDom$ (see Figure \ref{Fig:u})
that has a single corner at $0$, of angle $\omega=1.95 \pi$.
We work with a scalar coefficient field $a_\epsilon$ which is $\epsilon S\Q_1$-periodic, 
where $\epsilon \ll 1$, $Q:=[0,1]^2$, and $S$ is a rotation\footnote{Due to the rotation $S$, there is no obvious symmetry or alignment between the period of the coefficient field and the corner.}.
The matrix $a$ is chosen such that the homogenized matrix, $\bar{a}$, numerically satisfies $|\bar{a} -\rm{Id}| \leq 10^{-3}$.
We consider the following equation :
\begin{equation*}
\begin{cases}
-\nabla \cdot a\lt(\frac{\cdot}{\epsilon}\rt) \nabla u^\epsilon =f & \dans \tilDom,
\\
u^\epsilon=0 & \sur \partial \tilDom,
\end{cases}
\end{equation*}
where $f$ is a smooth function that vanishes near the corner (the third image in Figure \ref{Fig:u} shows $f$).
The main objective of our simulations is to indicate the superiority of the nonstandard 2-scale expansion introduced in~\eqref{2scale} over the classical (with Dirichlet correctors) 2-scale expansion in approximating $\nabla u^{\epsilon}$.

\begin{figure}[h]
	\begin{center}
		\includegraphics[width = \textwidth]{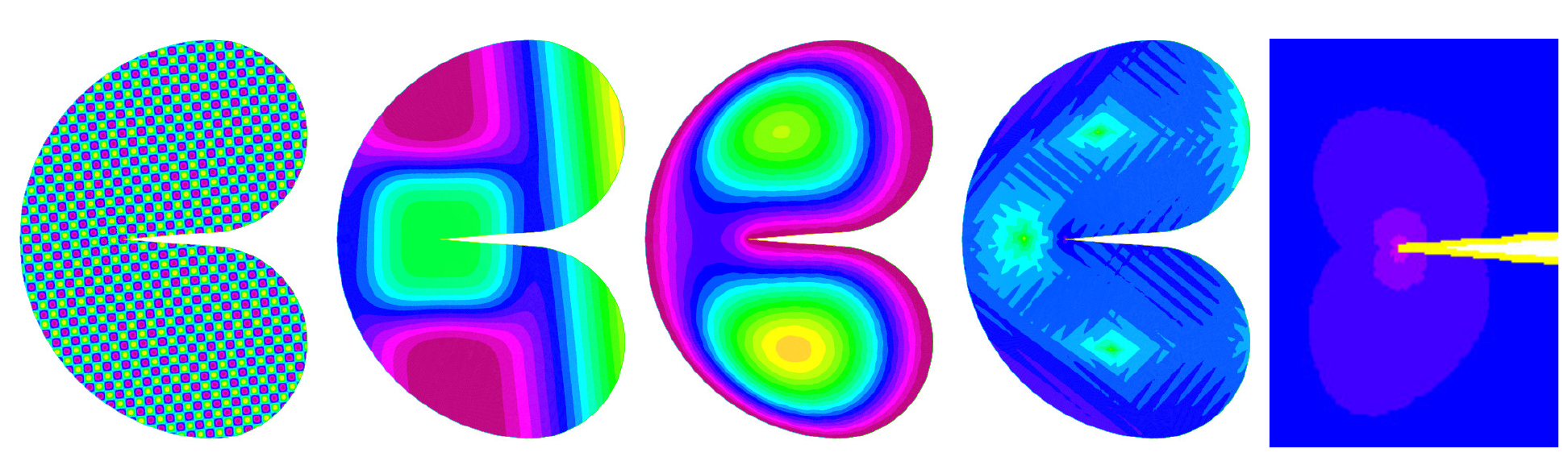}
	\end{center}
	\caption{
		From left to right : The coefficient field $a_\epsilon$ for $\epsilon=0.2$, the forcing term $f$, the solution $u^{\epsilon}$, $|\nabla u^{\epsilon}|$,
		and a detail of the latter.
		For the last image we have zoomed-in on the corner in order to highlight the singularity of $|\nabla u^{\epsilon}|$ at the corner.}
	\label{Fig:u}
\end{figure}

As a reference, we introduce the homogenized solution $\bar{u}$ to
\begin{equation*}
\begin{cases}
-\Delta \bar{u} =f & \dans \tilDom,
\\
\bar{u}=0 & \sur \partial \tilDom.
\end{cases}
\end{equation*}
As highlighted in Figure \ref{Fig:u} in the case of $\nabla u^{\epsilon}$, both $\nabla u^\epsilon$ and $\nabla \bar{u}$ display a singularity in the vicinity of the corner.
By Theorem \ref{Th:homog}, we have the following decomposition : $\ubar = \ubar_{\reg}^1 + \bar{\gamma}_1 \bar{\tau}_1$,
where $\ubar_{\reg}^1$ satisfies \eqref{Num:001} and represents the contribution of $\bar{u}$ which is more regular at the corner. 
\\
Next, we define the Dirichlet correctors $\phi_{\epsilon,i}^{\mathfrak{D},\tilDom}$ and the corner corrector $\phi_{1,\epsilon}^{\mathfrak{C},\tilDom}$ adapted to $\tilDom$ by
\begin{equation*}
\begin{cases}
-\nabla \cdot a\lt(\frac{\cdot}{\epsilon}\rt) \lt( \nabla \phi_{\epsilon,i}^{\mathfrak{D},\tilDom} + e_i\rt) =0 & \dans \tilDom,
\\
\phi_{\epsilon,i}^{\mathfrak{D},\tilDom}=0 & \sur \partial \tilDom,
\end{cases}
\qquad \et
\qquad 
\begin{cases}
-\nabla \cdot a\lt(\frac{\cdot}{\epsilon}\rt) \lt( \nabla \phi_{\epsilon,1}^{\mathfrak{C},\tilDom} + \nabla \bar{\tau}_1\rt) =0 & \dans \tilDom,
\\
\phi_{\epsilon,1}^{\mathfrak{C},\tilDom}=0 & \sur \partial \tilDom.
\end{cases}
\end{equation*}
In our simulations we use $\phi_{\epsilon,1}^{\mathfrak{C},\tilDom}$ as a proxy for the corner corrector $\epsilon^{\rhobar_1}\phiC_1(\frac{\cdot}{\epsilon})$ --which is itself out of reach numerically since it solves \eqref{Def:psi} on an infinite domain.

As already mentioned above, we compare the performance of the classical 2-scale expansion, \textit{i.e.}
\begin{equation}\label{naiveNum}
\tilde{u}^{\mathfrak{D},\epsilon} :=\ubar + \phi_{\epsilon,i}^{\mathfrak{D},\tilDom} \partial_i \ubar,
\end{equation}
in approximating $\nabla u^{\epsilon}$ to that of the hybrid 2-scale expansion proposed in \eqref{2scale}, \textit{i.e.}
\begin{equation}\label{2scaleNum}
\tilde{u}^{\mathfrak{C},\epsilon} :=\ubar_{\reg}^1 + \phi_{\epsilon,i}^{\mathfrak{D},\tilDom} \partial_i \ubar_{\reg}^1 + \bar{\gamma}_1\lt( \bar{\tau}_1 + \phi_{\epsilon,1}^{\mathfrak{C},\tilDom}\rt).
\end{equation}
For various values of $\epsilon$, we compare the relevant errors, $\Reps^{0,\epsilon} :=u^\epsilon - \tilde{u}^{\mathfrak{D},\epsilon}$ and $\Reps^{1,\epsilon} :=u^\epsilon - \tilde{u}^{\mathfrak{C},\epsilon}$, on shells by considering the energy norms.
In particular, for $i = 0,1$, we measure 
\begin{align*}
\Energ^{i,\epsilon}(R) :=\lt( \fint_{\Dom_R \backslash \Dom_{R/2}} |\nabla \Reps^{i,\epsilon}|^2 \rt)^{\frac{1}{2}} \pour R>0.
\end{align*}

These results are plotted in Figure \ref{Fig:Energ1}, where we track the ``gain'' that we obtain from using the hybrid 2-scale expansion, by showing $\Energ^{0,\epsilon}(R)/\Energ^{1,\epsilon}(R)$.
\begin{figure}[H]
	\begin{center}
		\includegraphics[width=\textwidth]{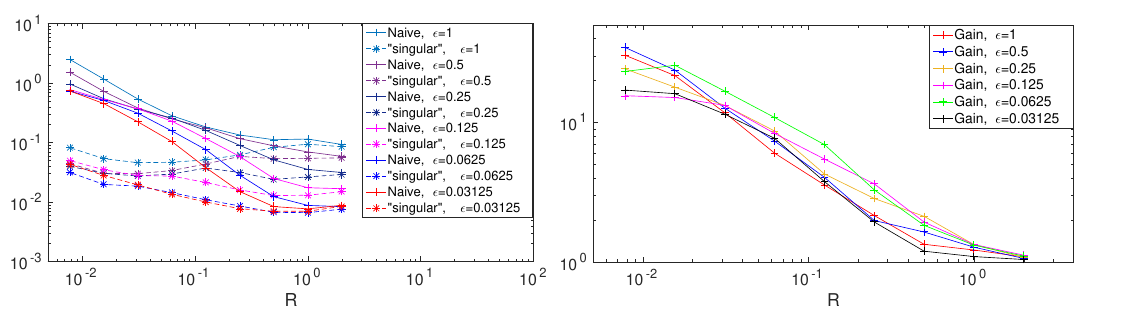}
	\end{center}
	\caption{
	On the left: $\Energ^{0,\epsilon}(R)$ and $\Energ^{1,\epsilon}(R)$ for various values of $\epsilon$.
	On the right: 
	Gain $\Energ^{0,\epsilon}(R)/\Energ^{1,\epsilon}(R)$ for various values of $\epsilon$.
}
	\label{Fig:Energ1}
\end{figure}

As can be seen above, in our simulations the hybrid 2-scale expansion always performed better than the classical 2-scale expansion.
Since the hybrid 2-scale expansion is designed to be particularly efficient near 0, it is natural that the advantage from using it gradually becomes more prominent as $R \downarrow 0$.
In the right graph of Figure \ref{Fig:Energ1} we see that a gain of precision of the 2-scale expansion larger than $10$ is obtained for radii $R \leq 0.02$.

Further numerical investigation shows that the gain scales like $R^{-0.6}$, which is not far from the expected value $R^{-\frac{\pi}{\omega}}$ since $\frac{\pi}{\omega} \simeq 0.51$.
The latter can be justified as follows :
On the one hand, remark that the 2-scale expansion \eqref{2scaleNum} only uses $N=N_0=1$ corner corrector and not the required $N>N_0$  (for $N_0$ defined by \eqref{Num:7029}). Hence, we may not directly use Theorem \ref{Th:Error}. However, a close inspection of the proof of Theorem \ref{Th:Error} (in particular Step 3) shows that we have to replace the \rhs of \eqref{Num:7034} by $\epsilon \ln^{\tilde{\nu}'}(\epsilon^{-1})|x|^{2(\rhobar_1-1)}$.
On the other hand, the gradient of the classical 2-scale expansion \eqref{naiveNum} suffers from a singularity $\phi_{\epsilon,i}^{\mathfrak{D},\tilDom} \partial_i \nabla\taubar_1$, which scales like $\epsilon |x|^{\rhobar_1-2}$.
Thus, neglecting the logarithms, we see that the expected gain scales like $|R|^{-\rhobar_1}$ as announced.

\subsection{Discussion of our results}\label{Sec:Discuss}

For the efficiency of the proofs and the simplicity of our results, we have used quite restrictive assumptions; 
however, these can be relaxed or even removed.

In Assumption \ref{AssumpId}, two hypotheses can be removed.
First, as already mentioned in Section \ref{Sec:ClassRes}, it is pure convenience to assume that $\abar=\Id$ : this can be removed by a simple change of coordinates (or equivalently, by doing so on the level of the singular functions $\taubar_n$).
Next, the symmetry of $a$ is by no means necessary.
However, if $a$ is not symmetric, it is necessary to also define transposed extended correctors $(\phi^*,\sigma^*)$ related to $a^*$ via \eqref{Decomposition}.
These transposed extended correctors would then also have to satisfy Assumption \ref{Assump2}.

For obtaining Theorem \ref{ThAL} and thus Corollary \ref{ThLiouville} for a given realization $a$, Assumption \ref{Assump2} can be relaxed.
Indeed, as in \cite{Fischer_Raithel_2017}, it is sufficient to assume that the extended correctors associated to the realization $a$ satisfy
\begin{equation}
\label{CorrSubDef-bis}
\begin{aligned}
&\displaystyle \sup_{r>0}  \lt(\fint_{\Boule_r} \lt|(\phi,\sigma) (x)- (\phi,\sigma) (y) \rt|^{2}\rt)^{1/2}
\leq C_* r^{1-\nu} \qquad \text{ for a given } \nu \in (0,1] \text{ and } C_*>0.
\end{aligned}
\end{equation}
We refer to Proposition \ref{Prop:Quench} below which goes in this direction.

In Theorem \ref{Th:OptiGR}, we only consider a first-order 2-scale expansion in the bulk of the domain.
However, considering a higher-order version of \eqref{2scale} might be interesting, hopefully retrieving a superior convergence rate.
(More precisely, we would employ classical higher-order correctors in order to approximate more precisely the regular part of $u$.)
In comparison, notice that the expansions of \cite{Costabel_2017} are provided up to any order.
Moreover, it is reasonable to think that, resorting to higher-order correctors, a more precise ansatz for $\phiC$ could be proposed.
However, such an approach could anyway not overcome the barrier in stochastic homogenization, which limits the control of the oscillation error to $O(\epsilon^{d/2})$ (see, \textit{e.g.}, \cite{DuerinckxOtto_2019}).

We believe that our method can be applied in more complex geometric situations as mentioned in Section \ref{Sec:GenDisc}; \textit{e.g.}, in higher dimensions or in situations involving interfaces as well as corners (see, \textit{e.g.} \cite[Chap.\ 5 \& 6]{Dauge_Livre_1988} and \cite{Costabel_2000}).
Since we employ tools that are also available for systems (our use of the maximum principle is a mere practical matter to get the pointwise \eqref{Opti_Decay}, by no means necessary), our results may also be applicable to systems with some adaptations.
Also, \cite{Kenig_Prange_2018} suggests that our results may still hold if we were to consider a domain $\Dom$ that looks like a corner on the macroscopic scale, but that oscillates on the small scale.

Last, we point out an intriguing technical detail :
In order to retrieve that the corner corrector $\phiC_1$ satisfies \eqref{Opti_Decay}, we need a separation of at least $\nu$ between the growth rate of the first $\abar$-harmonic function $\taubar_1$ and the decay rate of the first dual $\abar$-harmonic function $\taubar^\star_1$.
Hence, the assumption $\omega \leq 2\pi$ appears crucial to obtain \eqref{Opti_Decay} (see \eqref{Bizarre2} below).
This has practical consequences : In the case of ``corners'' displaying angles $\omega> 2 \pi$, we could only obtain deteriorated estimates on the corner corrector $\phiC_1$.
Even though this situation is less common than $\omega \leq 2\pi$, it is not unlikely (and may happen for a plate cut along a line, and bent so that one fold covers the other one).
Nevertheless, we do not know whether it is a mere technical restriction or a fundamental limitation.

\section{Strategy of proof}
\label{strategy}

\subsection{Ansatz for the corner correctors}\label{Sec:Better_Ansatz}

Before we present the precise ansatz for $\phiC_n$ let us briefly describe the idea : 
By standard calculations (using \eqref{Decomposition} and the skew-symmetry of the $\sigma_i$), 
the classical 2-scale expansion $\taubar_n+\phi_i\partial_i\taubar_n$ of the $\abar$-harmonic function $\taubar_n$ formally satisfies
\begin{align}
\label{equation_ansatz}
\nabla \cdot a\nabla (\taubar_n+\phi_i\partial_i\taubar_n)=\nabla \cdot (a\phi_i-\sigma_i)\nabla \partial_i\taubar_n=:\nabla\cdot \widehat h \dans \Dom,
\end{align}
where by \eqref{taubar_Homog} and \eqref{CorrSubDef} the growth  of $\widehat h$ is well-controlled.
This suggests an ansatz of the form
\begin{equation}\label{eq:ansatzsupernaiv}
 \phiC_n=\phi_i\partial_i\taubar_n+ \widehat\phi_n^{\mathfrak C},
\end{equation} 
where the correction $\widehat\phi_n^{\mathfrak C}$ satisfies an equation in $\Dom$ with \rhs $-\nabla \cdot \widehat h$ (determined by \eqref{equation_ansatz} and \eqref{Def:psi}).
Ansatz \eqref{eq:ansatzsupernaiv} has two drawbacks :
\begin{itemize}
\item $\widehat h$ is singular at the corner due to the blowup of $\nabla^2\taubar_n$,
\item $\widehat \phi_n^{\mathfrak C}$ satisfies the nontrivial boundary conditions $\widehat \phi_n^{\mathfrak C}=-\phi_i\partial_i\taubar_n$ on $\Gamma$.
\end{itemize}
To overcome the first problem, we simply introduce a cut-off at the tip of the corner. 
To address the second problem, we do not directly use the whole-space correctors $\phi_i$ but rather appeal to the half-space correctors associated to the ``upper'' boundary $\R_+(\cos\omega,\sin\omega)$ and to the ``lower'' boundary $\R_+ e_1$.\footnote{For consistency, we define all the correctors $\phiup_i$ and $\phidown_i$; however, in our arguments we only make use of $\phiup \cdot (\sin\omega,-\cos\omega)$ and $\phidown_2$, for the upper and the lower boundary respectively.
In the other directions, we are able to make use of the whole-space correctors $\phi_i$.}
These are defined as sublinear solutions to
\begin{equation}\label{Def:phidown}
\lt\{
\begin{aligned}
&-\nabla \cdot a \lt( \nabla \phiup_i + e_i \rt)=0 && \dans \mathbb{H}_\omega,
\\
&\phiup_i = 0 && \sur \partial\mathbb{H}_\omega,
\end{aligned}
\rt.
\qquad
\lt\{
\begin{aligned}
&-\nabla \cdot a \lt( \nabla \phidown_i + e_i \rt)=0 && \dans \R^2 \setminus \mathbb{H}_0,
\\
&\phidown_i = 0 && \sur \partial\mathbb{H}_0,
\end{aligned}
\rt.
\end{equation}
where $\mathbb{H}_\omega :=\{x \in \R^2, x\cdot(-\sin\omega,\cos\omega) \leq 0\}$.
Analogously to the whole-space case, flux-correctors, 
namely skew-symmetric $\R^{2 \times 2}$-valued fields $\sigup_i$ and $\sigdown_i$, can be defined through
\begin{equation} \label{Def:sigs}
a\nabla \phiup_i = (\abar-a) e_i + \nabla \cdot \sigup_i,
\qquad \et 
a\nabla \phidown_i = (\abar-a) e_i + \nabla \cdot \sigdown_i
\end{equation}
on the appropriate domains; \textit{i.e.}, on $\mathbb{H}_\omega$ for $\sigup_i$ and $\R^2 \setminus \mathbb{H}_0$ for $\sigdown_i$.
For use later on, we extend the half-space correctors and flux-correctors to the whole space by $0$. 
(In practice, they will be multiplied by a cut-off function so that their discontinuities due to the extension are irrelevant.)

The half-space correctors can be built and estimated by techniques similar to \cite[Th.\ 1.4]{JosienRaithel_2019} (see also \cite{BellaFischerJosienRaithel}) and we state without proof\footnote{The logarithm exponent $\tilde{\nu}+3$ of \eqref{Borne:HalfCorr} is suboptimal in most cases (it corresponds to the worst case $\nu=1$ for the flux correctors $\sigup$ and $\sigdown$).} : 
\begin{proposition}[Similar to Theorem~1.4 of \cite{JosienRaithel_2019}]\label{Prop:growhthscor}
Under Assumptions \ref{Assump1}, \ref{AssumpId}, and \ref{Assump2}, there exist extended half-space correctors $(\phiup,\sigup)$ and $(\phidown,\sigdown)$ that satisfy
\begin{align}\label{Borne:HalfCorr}
\begin{split}
& \langl \lt(\fint_{\Dom(x)}
 |\lt(\phiup,\phidown,\sigup,\sigdown \rt)|^2\rt)^{\frac{p}{2}} \rangl^{\frac{1}{p}}\\
& \qquad \lesssim_{\Xi,p}
(|x|+1)^{1-\nu} \ln^{\tilde{\nu} + 3}(|x|+2) \pourtout x \in \Dom \et  p \in [1,\infty).
\end{split}
\end{align}
\end{proposition}

Equipped with these extended half-space correctors, we are in a position to propose an efficient ansatz for the corner correctors.
Towards this aim, we introduce a partition of unity $\etaup + \etabulk + \etadown = 1$ in $\Dom$, such that all these functions, $\etaup$, $\etabulk$, and $\etadown$, are nonnegative and only depend on the angle $\theta$ in a smooth way, with $\etaup = 1$ for $\theta$ near $\omega$ and vanishing for $\omega - \theta > \omega/4$, $\etadown = 1$ for $\theta$ near $0$ and vanishing for $\theta > \omega/4$.

\begin{lemma}\label{*Lem:ansatzphi}
Then, using the following ansatz :
\begin{equation}
\label{*ansatz}
\phiC_n=(1-\eta_{\Boule,1})\lt( \etaup \phiup_i + \etabulk \phi_i  + \etadown \phidown_i \rt) \partial_i \taubar_n + \phitilde_n,
\end{equation}
the remainder $\phitilde_n$ satisfies
\begin{equation}
\label{*Def:psitilde2}
\lt\{
\begin{aligned}
-\nabla \cdot a \nabla \phitilde_{n} =~& \nabla \cdot h && \dans \Dom,
\\
\phitilde_{n}=~&0 && \sur \Gamma,
\end{aligned}
\rt.
\end{equation}
for
\begin{equation}
\label{*Def:h1h2}
\begin{aligned}
h :=\eta_{\Boule,1} (a-\abar) \nabla \taubar_n 
&+\lt(a\phiup_i- \sigup_i\rt)
\nabla \lt( (1-\eta_{\Boule,1}) \etaup  \partial_i \taubar_n \rt)
\\
&+\lt(a\phi_i- \sigma_i\rt)
\nabla \lt( (1-\eta_{\Boule,1}) \etabulk  \partial_i \taubar_n \rt)
\\
&+
\lt(a\phidown_i- \sigdown_i\rt)
\nabla \lt( (1-\eta_{\Boule,1})\etadown  \partial_i \taubar_n\rt).
\end{aligned}
\end{equation}
\end{lemma}

The \rhs of \eqref{*Def:psitilde2}, \textit{i.e.} $h$ defined above, enjoys the appealing property that, by \eqref{Borne:HalfCorr}, \eqref{CorrSubDef}, and \eqref{taubar_Homog}, and since $|\nabla \etaup(x)|+|\nabla \etabulk(x)|+|\nabla \etadown(x)|\lesssim |x|^{-1}$,  it satisfies
\begin{equation}\label{Borne:h}
 \langl \lt(\fint_{\Dom(x)} |h|^2\rt)^{\frac{p}{2}} \rangl^{\frac{1}{p}}
\lesssim_{\Xi,p}
(|x|+1)^{\rhobar_n-1-\nu} \ln^{\tilde{\nu}'}(|x|+2)\qquad\mbox{for all $x \in \Dom$},
\end{equation}
for $\tilde{\nu}' = \tilde{\nu} + 3$.
Hence, up to some logarithmic losses, we may expect that $\nabla \phitilde(x)$ is bounded by $(|x|+1)^{\rhobar_n-1-\nu}$.
Integrating this yields Theorem \ref{Th:OptiGR}.

\begin{remark}
Instead of \eqref{*ansatz}, we may consider the following ansatz (as in \cite{JosienRaithel_2019,Josien_InterfPer_2018}) :
\begin{equation}
\label{ansatz}
\phiC_n=\eta \phi_i \partial_i\taubar_n + \phitilde_n,
\end{equation}
where $\eta$ is a smooth cut-off function that is equal to $0$ on the boundary $\Gamma$.
However, this would necessitate dealing with the layer $\partial \Dom$, which, as argued in  \cite[Sec.\ 2.3]{JosienRaithel_2019}, causes losses when using
the $\LL^2$-type estimates at the core of the approach of \cite{Fischer_Raithel_2017}.
In \cite{JosienRaithel_2019}, this difficulty is circumvented by appealing to bounds on the mixed derivatives of the Green function.
Such a strategy is feasible but is less straightforward, \textit{cf.} the Acknowledgment section.
\end{remark}

\begin{proof}[Proof of Lemma \ref{*Lem:ansatzphi}]
We first recall the following identity for an arbitrary vector field $f$ :
\begin{equation}
\label{Num:4003}
-\nabla \cdot a \nabla \lt(\phiup_i f_i\rt)
=\nabla \cdot \lt( \sigup_i- a\phiup_i\rt)
\nabla f_i + \nabla \cdot (a - \abar) f. 
\end{equation}
A similar identity holds replacing $\phiup \rightsquigarrow \phidown$ and $\sigup \rightsquigarrow\sigdown$, and also replacing $\phiup \rightsquigarrow \phi$ and $\sigup \rightsquigarrow\sigma$.

For the convenience of the reader, we recall the argument for \eqref{Num:4003}.
We begin with
\begin{align}\label{Num:4001}
-\nabla \cdot a \nabla \lt(\phiup_i f_i\rt)
=&
-\nabla \cdot a \phiup_i \nabla f_i
-\nabla \cdot f_i a \nabla \phiup_i.
\end{align}
By \eqref{Def:sigs}, the rightmost term of \eqref{Num:4001} reads
\begin{equation*}
-\nabla \cdot f_i a \nabla \phiup_i
\overset{\eqref{Def:sigs}}=
\nabla \cdot (a-\abar) f - \nabla \cdot (\nabla \cdot \sigup_i) f_i
=
\nabla \cdot (a-\abar) f + \nabla \cdot \sigup_i \nabla f_i,
\end{equation*}
where we have made use of the antisymmetry of $\sigup_i$.
Inserting this into \eqref{Num:4001} yields \eqref{Num:4003}.

Here comes the proof of \eqref{*Def:psitilde2}.
We compute
\begin{align*}
-\nabla \cdot a \nabla \phitilde_n 
\overset{\eqref{*ansatz}}{=}~&
-\nabla \cdot a \nabla \phiC_n
+
\nabla \cdot a \nabla \lt((1-\eta_{\Boule,1}) \etaup  \phiup_i \partial_i \taubar_n\rt)
+\nabla \cdot a \nabla \lt((1-\eta_{\Boule,1}) \etabulk \phi_i \partial_i \taubar_n\rt)
\\&
+
\nabla \cdot a \nabla \lt( (1-\eta_{\Boule,1})\etadown \phidown_i \partial_i \taubar_n  \rt)
\\
\overset{\eqref{Def:psi},\eqref{Num:4003}}{=}~&
\nabla \cdot a \nabla \taubar_n
+\nabla \cdot (\abar - a) (1- \eta_{\Boule,1}) \nabla \taubar_n
-
\nabla \cdot \lt( \sigup_i- a\phiup_i\rt)
\nabla \lt( (1-\eta_{\Boule,1}) \etaup  \partial_i \taubar_n \rt)
\\
&-
\nabla \cdot \lt( \sigma_i- a\phi_i\rt)
\nabla \lt( (1-\eta_{\Boule,1}) \etabulk  \partial_i \taubar_n \rt)
-
\nabla \cdot \lt( \sigdown_i- a\phidown_i\rt)
\nabla \lt( (1-\eta_{\Boule,1}) \etadown  \partial_i \taubar_n \rt).
\end{align*}
By \eqref{Def:taubar}, the first two \rhs terms combine to
\begin{align*}
\nabla \cdot a \nabla \taubar_n
+\nabla \cdot (\abar - a) (1- \eta_{\Boule,1}) \nabla \taubar_n
=
\nabla \cdot \eta_{\Boule,1} (a-\abar) \nabla \taubar_{n}.
\end{align*}
This yields \eqref{*Def:psitilde2} and \eqref{*Def:h1h2}.
\end{proof}

\subsection{Conditional regularity for $a$-harmonic functions at the corner}

As has become classical in homogenization following the work of Avellaneda and Lin \cite{AvellanedaLin} on large-scale Lipschitz estimates, the workhorse behind our results is a large-scale regularity theorem.
Assuming that the growth rate of the corner (and extended) correctors is well-controlled, it provides an algebraic decay at the corner for a suitable renormalization of an $a$-harmonic function (in the spirit of Theorem~\ref{ThAL}).
Notice that, unlike Theorem~\ref{ThAL}, this result is purely deterministic.

\begin{lemma}\label{Lem:heterog}
Assume that the coefficient field $a$ satisfies Assumptions \ref{Assump1} and \ref{AssumpId}.
Let $N \in \mathbb{N}$ and $\rho \in [\rhobar_{N},\rhobar_{N+1})$.
Then, for $0<\delta \ll_{\omega,\lambda,N,\rho} 1$, $C_0 \lesssim_{\omega,\lambda,N,\rho} 1$
and $C_0' \lesssim_{\omega,\lambda,N} 1$, the following property holds :
Let $\rstar \geq 1$.
Assume that the extended whole-space correctors $\lt(\phi,\sigma\rt)$ and, for $n \in \{1,\dots,N\}$, the corner correctors $\phiC_n$ satisfy the following estimate :
\begin{equation}
\label{Sublin**}
\sup_{r \geq \rstar}  \lt[ \frac{1}{r} \lt( \fint_{\Boule_r} \lt|(\phi,\sigma)\rt|^2  \rt)^{\frac{1}{2}}
+\sum_{n=1}^N\frac{1}{r^{\rhobar_n}} \lt( \fint_{\Dom_r} \lt|\phiC_n\rt|^2  \rt)^{\frac{1}{2}}\rt] 
\leq \delta.
\end{equation}
Let $\rmax \geq 1$.
If $u$ satisfies
\begin{equation}\label{Num:011}
\begin{cases}
-\nabla \cdot a \nabla u=0 & \dans \Dom_{\rmax},
\\
u=0 & \sur \Gamma_{\rmax},
\end{cases}
\end{equation}
then there exist coefficients $\gamma_1, \cdots, \gamma_N \in \R$ such that there holds
\begin{equation}\label{Num:119}
\lt(\fint_{\Dom_r} \lt| u - \sum_{n=1}^N \gamma_n \lt( \taubar_n + \phiC_n\rt) \rt|^2 \rt)^{\frac{1}{2}}
\leq C_0
\max\lt\{1, \lt(\frac{\rstar}{r}\rt)^{1+\rho}\rt\}
\lt(\frac{r}{\rmax} \rt)^\rho \lt(\fint_{\Dom_{\rmax}} |u|^2\rt)^{\frac{1}{2}},
\end{equation}
for all $r \in [1,\rmax]$, along with
\begin{equation}
\label{BorneGamma}
\lt|\gamma_n\rt| \leq C_0 \rmax^{-\rhobar_n} 
\lt(\fint_{\Dom_{\rmax}} \lt| u \rt|^2\rt)^{\frac{1}{2}}.
\end{equation}
In particular, in the case $N\geq1$ it holds that
\begin{equation}\label{Num:119:lip}
\lt(\fint_{\Dom_r} \lt| u - \sum_{n=1}^{N-1} \gamma_n \lt( \taubar_n + \phiC_n\rt) \rt|^2 \rt)^{\frac{1}{2}}
\leq C_0'
\max\lt\{1, \lt(\frac{\rstar}{r}\rt)^{1+\rhobar_N}\rt\}
\lt(\frac{r}{\rmax} \rt)^{\rhobar_N} \lt(\fint_{\Dom_{\rmax}} |u|^2\rt)^{\frac{1}{2}},
\end{equation}
for $r\geq1$.
\end{lemma}

By the Caccioppoli estimate on the l.~h.~s.\ and the Poincaré inequality on the r.~h.~s., we may replace \eqref{Num:119} by
\begin{equation}\label{Num:119_bis}
\lt(\fint_{\Dom_r} \lt| \nabla u - \sum_{n=1}^N \gamma_n \lt( \nabla \taubar_n + \nabla \phiC_n\rt) \rt|^2 \rt)^{\frac{1}{2}}
\leq C_0
\max\lt\{1, \lt(\frac{\rstar}{r}\rt)^{1+\rho}\rt\}
\lt(\frac{r}{\rmax} \rt)^{\rho-1} \lt(\fint_{\Dom_{\rmax}} |\nabla u|^2\rt)^{\frac{1}{2}}.
\end{equation}

\subsection{Conditional quasi-optimal growth rates in weaker spatial norm} \label{Sec:QuenchRes}

Using Lemma \ref{Lem:heterog}, we can prove a deterministic version of Theorem \ref{Th:OptiGR} with estimates in a weaker spatial norm.

\begin{proposition} \label{Prop:Quench}
Let $a$ satisfy Assumption \ref{Assump1} and  \ref{AssumpId}.
We assume that the extended whole-space and half-space correctors (characterized above in Section \ref{Sec:Better_Ansatz}) are such that there exist exponents $\nu \in (0,1]$ and $\tilde{\nu} \geq 0$, with
\begin{equation}\label{Num:40101}
\lt( \fint_{\Dom_{r}} \lt|\phi,\phiup,\phidown,\sigma,\sigup,\sigdown \rt|^2 \rt)^{\frac{1}{2}}
\leq
r^{1-\nu} \ln^{\tilde{\nu}}(r+1)
\pourtout r \geq 1.
\end{equation}
Then, there exist an exponent $\tilde{\nu}' \lesssim_{\tilde{\nu}} 1$ and a constant $\Cstar \lesssim_{\omega,\lambda,\nu,\tilde{\nu},n} 1$,
such that, for any $n \in \mathbb{N} \backslash\{0\}$, there exists a corner corrector $\phiC_n$ that is decomposed as in~\eqref{*ansatz} and for which $\phitilde_{n}$ satisfies the following estimate :
\begin{equation}\label{Num:4011}
\lt( \fint_{\Dom_{2r} \backslash \Dom_{r}} \lt|\nabla \phitilde_n \rt|^2 \rt)^{\frac{1}{2}}
\leq
\Cstar r^{\rhobar_n-1-\nu} \ln^{\tilde{\nu}'}(r+1)
\pourtout r \geq 1.
\end{equation}
In particular, we also have
\begin{equation}\label{Num:4010}
\lt( \fint_{\Dom_{2r} \backslash \Dom_r} \lt|\phiC_n \rt|^2 \rt)^{\frac{1}{2}}
\leq
\Cstar r^{\rhobar_n-\nu} \ln^{\tilde{\nu}'}(r+1)
\pourtout r \geq 1.
\end{equation}
Similarly, for $j = 1, 2$, there exist Dirichlet correctors $\phiD_j$ that may be decomposed as in \eqref{*ansatz}, replacing $\partial_i\taubar_n$ by the Kronecker symbol $\delta_{ij}$ and $\phitilde_n$ by $\phiDtilde_j$, so that there holds
\begin{equation}\label{Num:4011-phiD}
\lt( \fint_{\Dom_{2r} \backslash \Dom_{r}} \lt|\nabla \phiDtilde_j \rt|^2 \rt)^{\frac{1}{2}}
\leq
\Cstar r^{-\nu} \ln^{\tilde{\nu}'}(r+1) 
\quad\text{ and }
\lt( \fint_{\Dom_{2r} \backslash \Dom_r} \lt|\phiD_j \rt|^2 \rt)^{\frac{1}{2}}
\leq
\Cstar r^{1-\nu} \ln^{\tilde{\nu}'}(r+1)
\qquad
\pourtout r \geq 1.
\end{equation}
\end{proposition}

Proposition \ref{Prop:Quench} and Lemma \ref{Lem:heterog} immediately imply a Liouville principle, in form of the following deterministic version of Corollary \ref{ThLiouville}.
\begin{corollary}[Liouville principle]\label{ThLiouville-2}
	We place ourselves under the assumptions of Proposition \ref{Prop:Quench}.
	Let $N \in \mathbb{N}$ and $\rho \in (\rhobar_{N},\rhobar_{N+1})$.
	If $u$ is a subalgebraic $a$-harmonic function in $\Dom$ in the sense of \eqref{Hypo :ThLiouville}, then there exist coefficients $\gamma_1, \ldots, \gamma_N \in \R$ such that we may decompose $u$ as follows :
	\begin{equation*}
	u=\sum_{n=1}^N \gamma_n (\taubar_n+\phiC_n).
	\end{equation*}
\end{corollary}

Once Lemma \ref{Lem:heterog}, Proposition \ref{Prop:Quench}, and Corollary \ref{ThLiouville-2} are established, Theorems \ref{ThAL} and \ref{Th:OptiGR} and Corollary \ref{ThLiouville} follow easily.

\subsection{Error estimate for the nonstandard 2-scale expansion}

In order to estimate the error of the hybrid 2-scale expansion, we need an analogue to the classical equation for the standard 2-scale expansion.
This requires us to build and estimate the flux-correctors $\sigmaD_i$ associated to the Dirichlet correctors $\phiD_i$ :
\begin{lemma}\label{Lem:sig}
	Assume that $\omega \in (0,2\pi)$.
	We place ourselves under the Assumptions \ref{Assump1}, \ref{AssumpId}, and \ref{Assump2}.
	Then there exists a Dirichlet flux-corrector on $\Dom$, which is a skew-symmetric tensor field $\sigmaD_i : \Dom \rightarrow \R^{2 \times 2}$, where $i \in \{1,2\}$, such that
	\begin{equation}\label{Def:sigD}
	\nabla \cdot \sigmaD_i= a(\nabla \phiD_i+e_i) - \abar e_i.
	\end{equation}
	This Dirichlet flux-corrector is decomposed as
	\begin{equation}\label{ansatz:sig}
	\sigmaD_i = (1-\eta_{\Boule,1})(\etaup \sigup_i + \etabulk \sigma_i+ \etadown\sigdown_i) + \tisigmaD_i,
	\end{equation}
	where $\etaup$, $\etabulk$, and $\etadown$ are cut-off functions as in Section \ref{Sec:Better_Ansatz}, and there exists $\tilde{\nu}'\lesssim_{\tilde{\nu}} 1$ such that, for any $p \in [1,\infty)$, there holds
	\begin{equation}\label{M:0005}
	\langl\lt( \fint_{\Dom(x)} \lt| \sigmaD_i \rt|^2 \rt)^{\frac{p}{2}} \rangl^{\frac{1}{p}} 
	\lesssim_{\Xi,p}
	(|x|+1)^{1-\nu} \ln^{\tilde{\nu}'}(|x|+2) \pourtout x \in \Dom.
	\end{equation}	
\end{lemma}

\begin{remark}[The case $\omega = 2\pi$] 
Notice that in Lemma \ref{Lem:sig} we have purposefully not included that case $\omega = 2 \pi$.
This is because in our construction of $\sigmaD_i$ we are required to extend the vector-field on the \rhs of \eqref{equation_sigma_tilde} in a divergence-free way (this is done in Lemma \ref{Lem:DivFree} in the Appendix) and our construction of the extensions fails in this case.

\end{remark}
For simplicity, in the sequel, we denote the rescaled quantities :
\begin{align*}
&a_\epsilon := a\lt(\frac{\cdot}{\epsilon}\rt), \quad
\phi_{\epsilon,i} := \epsilon \phi_i\lt(\frac{\cdot}{\epsilon}\rt),
\quad
\phiC_{\epsilon,n} := \epsilon^{\rhobar_n}\phiC_{n}\lt(\frac{\cdot}{\epsilon}\rt),
\quad
\phiD_{\epsilon,i} := \epsilon \phiD_i\lt(\frac{\cdot}{\epsilon}\rt),
\quad
\sigma_{\epsilon,i} := \epsilon \sigma_i\lt(\frac{\cdot}{\epsilon}\rt),
\quad \sigmaD_{\epsilon,i} := \epsilon\sigmaD_i\lt(\frac{\cdot}{\epsilon}\rt).
\end{align*}
We set $\chi:=\eta_{\Boule,1}$, and we extend \eqref{Decompose_ubar} and \eqref{2scale} from $\Dom_{1/2}$ to $\Dom$ as follows (thus redefining $\ubar_{\reg}^N$ and $\utieps^N$):
\begin{align}\label{Num:7003}
\ubar &= \ubar_{\reg}^N + \sum_{n=1}^N \bar{\gamma}_n \taubar_n \chi,
\\
\label{2-scale-New}
\utieps^N &:= (1+\phiD_{\epsilon,i}\partial_i) \ubar_{\reg}^N + \sum_{n=1}^N \bar{\gamma}_n (\taubar_n + \phiC_{\epsilon,n})\chi + \sum_{n=1}^N \bar{\gamma}_n \taubar_n \phiD_{\epsilon,i} \partial_i \chi.
\end{align}
(Notice that \eqref{Decompose_ubar} and \eqref{Num:7003} on the one hand, and \eqref{2scale} and \eqref{2-scale-New} on the other hand indeed coincide in $\Dom_{1/2}$.
However, the introduction of cut-off functions in \eqref{Num:7003} and \eqref{2-scale-New} counterbalances the growth at infinity of the functions $\taubar_n$.)

Then, we may express the error $\nabla \ueps-\nabla\utieps^N$ as the solution of an elliptic equation:

\begin{lemma}\label{Lem:2scale}
	Let $N \in \mathbb{N}$.
	Assume that $\ueps$ and $\bar{u}$ satisfy \eqref{Num:7001}, and define $\utieps^N$ and $\ubar_{\reg}^N$  by \eqref{2-scale-New} and \eqref{Num:7003}, respectively.
	Then, there holds
	\begin{equation}\label{E:1-0}
	-\nabla \cdot a_\epsilon (\nabla \utieps^{N}  - \nabla \ueps)
	=
	\nabla \cdot h^{N}_\epsilon,
	\end{equation}
	where
	\begin{equation}\label{E:1-bis-0}
	h^{N}_\epsilon :=\lt(\sigmaD_{\epsilon,i} - a_\epsilon \phiD_{\epsilon,i} \rt) \partial_i \nabla \bar{v}^{N}
	-\sum_{n=1}^N \bar{\gamma}_n a_\epsilon \nabla \lt((\chi-1)(\phiC_{\epsilon,n} - \phiD_{\epsilon,i} \partial_i \taubar_n)\rt),
	\end{equation}
	for
	\begin{equation}\label{Def:barv-0}
	\bar{v}^{N} :=\ubar - \sum_{n=1}^N \bar{\gamma}_n \taubar_n.
	\end{equation}	
\end{lemma}
Next, we estimate the \rhs of \eqref{E:1-0}, $h^N_\epsilon$, in the annealed $\LL^\infty$-like norm analogous to that in \eqref{Num:7034}, \textit{cf.} Lemma \ref{Lem:Prelim}.
By post-processing the Lipschitz-like estimates of Theorem \ref{ThAL} to accommodate a \rhs, \textit{cf.} Section \ref{Sec:Linfty}, we may transfer these estimates on the level of $\nabla \utieps^{N}  - \nabla \ueps$ and establish Theorem \ref{Th:Error}.

\section{Conditional regularity at the corner: Argument for Lemma \ref{Lem:heterog}}
\label{proof_lemma3.3}

In this section we give the argument for Lemma \ref{Lem:heterog} : This requires the iterative use of Lemma \ref{Lem:Hcv}, which we state and then prove below. Once we have access to Lemma \ref{Lem:Hcv}, we can iterate it over various scales to obtain Lemma \ref{Lem:heterog}.

\subsection{Iteration Lemma}
\label{iterative_lemma}

\begin{lemma}[Iteration Lemma]\label{Lem:Hcv}
Let $N \in \mathbb{N}$ and $\rho< \rhobar_{N+1}$.
There exists a constant $C_{\omega,\lambda,N}$
and we may choose
$\theta \ll_{\omega,\lambda,N,\rho} 1$ and $\delta \ll_{\omega,\lambda,N,\rho} 1$ such that the following property holds : 
Let $\radius>0$ be given.
Assume that the extended corrector $(\phi,\sigma)$ and the corner correctors $\phiC_n$, for $n \in \{1,\dots,N\}$, satisfy
\begin{equation}\label{Sublin*}
\frac{1}{\radius} \lt( \fint_{\Boule_\radius} \lt|(\phi,\sigma)\rt|^2  \rt)^{\frac{1}{2}}+\sum_{n=1}^N \frac{1}{\radius^{\rhobar_n}} \lt( \fint_{\Dom_\radius} \lt|\phiC_n\rt|^2  \rt)^{\frac{1}{2}} \leq \delta.
\end{equation}
Then, for any solution $u$ to \eqref{Num:011} there exist coefficients $\gamma_n$ bounded as follows :
\begin{equation} 
\label{Estim :gamma}
\lt|\gamma_n\rt|\leq C_{\omega,\lambda,N} r^{-\rhobar_n} \lt(\fint_{\Dom_{\radius}} |u|^2\rt)^{\frac{1}{2}},
\end{equation}
and such that there holds :
\begin{equation}\label{Num:012_quad}
\lt(\fint_{\Dom_{\theta \radius}} \lt| u - \sum_{n=1}^{N} \gamma_n \lt(\taubar_n+ \phiC_n\rt) \rt|^2 \rt)^{\frac{1}{2}}
\leq
\theta^{\rho} \lt(\fint_{\Dom_{\radius}} |u|^2\rt)^{\frac{1}{2}}.
\end{equation}
\end{lemma}

\begin{proof}[Proof of Lemma \ref{Lem:Hcv}]

The core of this proof is to make use of H-convergence in order to establish that
\begin{equation}\label{Num:012}
\lt(\fint_{\Dom_{\theta \radius}} \lt| u - \sum_{n=1}^{N} \gamma_n \taubar_n \rt|^2 \rt)^{\frac{1}{2}}
\leq
\frac{\theta^{\rho}}{2} \lt(\fint_{\Dom_{\radius}} |u|^2\rt)^{\frac{1}{2}},
\end{equation}
where, for $\ubar$ defined below, the coefficients $\gamma_n$ are given by
\begin{equation}\label{M:0030}
\gamma_n=-\frac{1}{n \pi} \int_{\Dom}  \ubar\lt(2 \nabla \taubar^\star_n \cdot \nabla\eta_{\Boule,\radius/2} +\taubar^\star_n \Delta \eta_{\Boule,\radius/2}  \rt).
\end{equation}

To begin, we establish \eqref{Num:012} in two steps : 
In Step 1, we show that the solution $\ubar$ of the homogeneous equation
\begin{equation}
\begin{cases}
-\Delta \ubar =0 &\dans \Dom_{\radius/2},
\\
\ubar =u &\sur \partial \Dom_{\radius/2},
\end{cases}
\end{equation}
satisfies
\begin{equation}
\label{Num:203}
\lt(\fint_{\Dom_{\theta \radius}} \lt| \ubar - \sum_{n=1}^{N} \gamma_n \taubar_n \rt|^2 \rt)^{\frac{1}{2}}
\lesssim_{\omega,N}
\theta^{\rhobar_{N+1}} \radius\lt(\fint_{\Dom_{\radius/2}} |\nabla u|^2\rt)^{\frac{1}{2}}.
\end{equation} 
Then, in Step 2, as a consequence of \eqref{Sublin*},  we justify by H-convergence that \eqref{Num:203} can be ``transferred'' to the heterogeneous coefficient operator, obtaining \eqref{Num:012}.
To finish, in Step 3, we show \eqref{Estim :gamma} and employ once more \eqref{Sublin*} to establish~\eqref{Num:012_quad}.

\paragraph{Step 1 : Estimate in the homogeneous setting.}
We first rescale $\ubar_\radius := \ubar(\radius \cdot /2)$, on which we apply Theorem~\ref{Th:homog}.
Therefore, for any $\theta \leq 1/4$ we have
\begin{equation}\label{Num:013}
\lt(\fint_{\Dom_{2\theta}} \lt| \ubar_\radius - \sum_{n=1}^{N} \gamma^\radius_n \taubar_n  \rt|^2 \rt)^{\frac{1}{2}}
\lesssim_{\omega,N}
\theta^{\rhobar_{N+1}} \lt(\fint_{\Dom_{1}} |\ubar_\radius|^2\rt)^{\frac{1}{2}}
\quad \text{with }\gamma^\radius_n \overset{\eqref{Formula :gamma3}}{=}
-\frac{1}{n\pi}\int_{\Dom_1}\ubar_\radius \lt(2 \nabla  \taubar^\star_n \cdot \nabla \eta_{\Boule,1/2} + \taubar^\star_n \Delta \eta_{\Boule,1/2}\rt). 
\end{equation}
By the change of variables $x \rightsquigarrow (2/\radius)x$ (recall \eqref{Def:taubar0} and \eqref{Def:taubarstar0}), we get  $\gamma^\radius_n = (\radius/2)^{\rhobar_n} \gamma_n$ for $\gamma_n$ defined by \eqref{M:0030}, and thus
\begin{equation}\label{Num:012_bis}
\lt(\fint_{\Dom_{\theta \radius}} \lt| \ubar - \sum_{n=1}^{N} \gamma_n \taubar_n \rt|^2 \rt)^{\frac{1}{2}}
\lesssim_{\omega,N}
\theta^{\rhobar_{N+1}} \lt(\fint_{\Dom_{\radius/2}} |\ubar|^2\rt)^{\frac{1}{2}}.
\end{equation}
Moreover, using successively the Poincaré inequality and the energy estimate on $u-\ubar$, we get
\begin{equation}\label{M:0031}
\lt(\fint_{\Dom_{\radius/2}} |\ubar|^2\rt)^{\frac{1}{2}}
\lesssim_{\omega}
\radius\lt(\fint_{\Dom_{\radius/2}} |\nabla \ubar|^2\rt)^{\frac{1}{2}}
\lesssim
\radius\lt(\fint_{\Dom_{\radius/2}} |\nabla u|^2\rt)^{\frac{1}{2}}.
\end{equation}
Thus, \eqref{Num:012_bis} upgrades to \eqref{Num:203}.

\paragraph{Step 2 : H-convergence.}
By an argument from  \cite[Prop. 2.1. (17)]{JosienOtto_2019} based on H-convergence and involving \eqref{Sublin*}, there exist constants $C_{\lambda,\omega,\delta}$ satisfying $C_{\lambda,\omega,\delta} \rightarrow0$
as  $\delta \downarrow 0$, such that
\begin{equation}\label{Num:201}
\lt(\fint_{\Dom_{\radius/2}} \lt| u - \ubar \rt|^2\rt)^{\frac{1}{2}} \leq C_{\lambda,\omega,\delta}\radius \lt(\fint_{\Dom_{\radius/2}} \lt|\nabla u\rt|^2\rt)^{\frac{1}{2}}.
\end{equation}

By a triangle inequality involving \eqref{Num:203} and \eqref{Num:201}, we obtain
\begin{equation*}
\lt(\fint_{\Dom_{\theta \radius}} \lt| u - \sum_{n=1}^{N} \gamma_n \taubar_n \rt|^2 \rt)^{\frac{1}{2}}
\lesssim_{\lambda,\omega,N}
\lt(\theta^{\rhobar_{N+1}}+ C_{\lambda,\omega,\delta}\theta^{-1}\rt)\radius \lt(\fint_{\Dom_{\radius/2}} \lt|\nabla u\rt|^2\rt)^{\frac{1}{2}}.
\end{equation*} 
Moreover, the Caccioppoli inequality yields
\begin{equation*}
\radius \lt(\fint_{\Dom_{\radius/2}} \lt|\nabla u\rt|^2\rt)^{\frac{1}{2}}
\lesssim_{\lambda,\omega} \lt(\fint_{\Dom_{\radius}} \lt|u\rt|^2\rt)^{\frac{1}{2}}.
\end{equation*}
As a consequence, we get
\begin{equation*}
\lt(\fint_{\Dom_{\theta \radius}} \lt| u - \sum_{n=1}^{N} \gamma_n \taubar_n \rt|^2 \rt)^{\frac{1}{2}}
\lesssim_{\lambda,\omega,N}
\lt(\theta^{\rhobar_{N+1}}+ C_{\lambda,\omega,\delta}\theta^{-1}\rt)
\lt(\fint_{\Dom_{\radius}} \lt|u\rt|^2\rt)^{\frac{1}{2}}.
\end{equation*}
Therefore, we may successively choose $\theta \ll_{\lambda,\omega,N,\rho} 1$ and $\delta  \ll_{\lambda,\omega,N,\rho} 1$ such that \eqref{Num:012} holds.

\paragraph{Step 3: Proof of \eqref{Estim :gamma} and \eqref{Num:012_quad}}
We get \eqref{Estim :gamma} as a direct corollary of the Cauchy-Schwarz inequality applied to \eqref{M:0030} (recalling \eqref{taubar_Homog}), using then \eqref{M:0031} and the Caccioppoli inequality :
\begin{align*}
|\gamma_n| \lesssim \radius^{-\rhobar_n-2} \lt(\int_{\Dom_{\radius/2}} |\ubar|^2\rt)^{\frac{1}{2}} \overset{\eqref{M:0031}}{\lesssim}
\radius^{-\rhobar_n-1}\lt(\fint_{\Dom_{\radius/2}} |\nabla u|^2\rt)^{\frac{1}{2}}
\lesssim
\radius^{-\rhobar_n}\lt(\fint_{\Dom_{\radius}} |u|^2\rt)^{\frac{1}{2}}.
\end{align*}

Then, combining \eqref{Sublin*} and \eqref{Estim :gamma}, we get
\begin{align*}
\lt(\fint_{\Dom_{\theta \radius}} \lt| \sum_{n=1}^{N} \gamma_n \phiC_n \rt|^2 \rt)^{\frac{1}{2}}
\lesssim
\sum_{n=1}^N \theta^{-1} \radius^{-\rhobar_n} \lt(\fint_{\Dom_{\radius}} |u|^2\rt)^{\frac{1}{2}}
\delta \radius^{\rhobar_n}
\lesssim \theta^{-1}\delta \lt(\fint_{\Dom_{\radius}} |u|^2\rt)^{\frac{1}{2}}.
\end{align*}
Up to choosing $\delta$ smaller than in the previous step, we deduce from above
\begin{equation}\label{Num:205}
\lt(\fint_{\Dom_{\theta \radius}} \lt| \sum_{n=1}^{N} \gamma_n \phiC_n \rt|^2 \rt)^{\frac{1}{2}}
\leq
\frac{\theta^\rho}{2}\lt(\fint_{\Dom_{\radius}} |u|^2\rt)^{\frac{1}{2}}.
\end{equation}
Finally, using the triangle inequality and appealing to \eqref{Num:012} and \eqref{Num:205} yields \eqref{Num:012_quad}.
\end{proof}

\subsection{Inductive argument: Proof of Lemma \ref{Lem:heterog}}
\label{proof_3.3}

\begin{proof}[Proof of Lemma \ref{Lem:heterog}]
We set $C_{\omega,\lambda, N}$, $\theta$ and $\delta$ as in Lemma \ref{Lem:Hcv}.
We split our proof into three steps : The inductive use of Lemma \ref{Lem:Hcv} is relegated to the first two steps. In these steps it is shown that, for any $m \geq 0$ such that $\theta^m R>\rstar$, we may define $\gamma_n(m) \in \R$ such that there hold
\begin{equation}\label{HR :1}
\lt(\fint_{\Dom_{\theta^m \rmax}} \lt| u - \sum_{n=1}^N \gamma_n(m) \lt( \taubar_n + \phiC_n\rt) \rt|^2 \rt)^{\frac{1}{2}}
\leq \theta^{m \rho}
\lt(\fint_{\Dom_{\rmax}} \lt| u \rt|^2 \rt)^{\frac{1}{2}},
\end{equation} 
and
\begin{equation}\label{HR :2}
\lt|\gamma_n(m)\rt| \leq C_{\omega,\lambda, N} \rmax^{-\rhobar_n} \lt(\sum_{m'=1}^m \theta^{(m'-1)(\rho-\rhobar_n)}\rt) \lt(\fint_{\Dom_{\rmax}} \lt| u \rt|^2 \rt)^{\frac{1}{2}}.
\end{equation}
In the last step, we iterate \eqref{HR :1} and \eqref{HR :2} in order to conclude our argument.

\paragraph{Step 1 : Initialization.}
For $m=0$, we may take $\gamma_n(0)=0$ by which both \eqref{HR :1} and \eqref{HR :2} become tautological.

\paragraph{Step 2 : Inductive Step.}
Assume that $m\geq 0$ is such that $\theta^{m} \rmax> \rstar$, and that \eqref{HR :1} and \eqref{HR :2} hold.
By assumption \eqref{Sublin**}, we may apply Lemma \ref{Lem:Hcv} to
\begin{equation}\label{Num:2007}
v :=u - \sum_{n=1}^N \gamma_n(m) \lt( \taubar_n + \phiC_n\rt),
\end{equation}
which satisfies \eqref{Num:011}, in the domain $\Dom_{\theta^m \rmax}$.
Hence, there exist $N$ coefficients $\tilde{\gamma}_n(m+1) \in \R$ satisfying
\begin{align}\label{Num:1001}
\lt|\tilde{\gamma}_n(m+1)\rt| \overset{\eqref{Estim :gamma}}{\leq}
C_{\omega,\lambda, N} (\theta^m\rmax)^{-\rhobar_n} \lt(\fint_{\Dom_{\theta^m\rmax}} |v|^2\rt)^{\frac{1}{2}}
\overset{\eqref{Num:2007}, \eqref{HR :1}}{\leq}
C_{\omega,\lambda, N} \theta^{m (\rho-\rhobar_n)} \rmax^{-\rhobar_n}
\lt(\fint_{\Dom_{\rmax}} \lt| u \rt|^2 \rt)^{\frac{1}{2}},
\end{align}
and such that
\begin{align}\label{Num:1002}
\lt(\fint_{\Dom_{\theta^{m+1} \rmax}} \lt| v - \sum_{n=1}^N \tilde{\gamma}_n(m+1) \lt( \taubar_n + \phiC_n\rt) \rt|^2 \rt)^{\frac{1}{2}}
\overset{\eqref{Num:012_quad}}{\leq}
\theta^\rho
\lt(\fint_{\Dom_{\theta^{m} \rmax}} \lt| v\rt|^2 \rt)^{\frac{1}{2}}.
\end{align}
Defining $\gamma_n(m+1) :=\gamma_n(m) + \tilde{\gamma}_n(m+1)$, recalling the definition of $v$ and inserting \eqref{HR :1} into the r.~h.~s.\ of \eqref{Num:1002}, we immediately obtain that $u$ satisfies \eqref{HR :1} for $m$ replaced by $m+1$.
Moreover, appealing to the triangle inequality, using \eqref{HR :2} and \eqref{Num:1001} yields
\begin{align*}
\lt|\gamma_n(m+1)\rt| 
\leq & \lt|\gamma_n(m)\rt| +  \lt|\tilde{\gamma}_n(m+1)\rt|
\leq 
C_{\omega,\lambda, N} \rmax^{-\rhobar_n} \lt(\sum_{m'=1}^{m+1} \theta^{(m'-1)(\rho-\rhobar_n)}\rt)
\lt(\fint_{\Dom_{\rmax}} \lt| u \rt|^2 \rt)^{\frac{1}{2}}.
\end{align*}
This establishes \eqref{HR :2} for $m$ replaced by $m+1$ and concludes the inductive step.

\paragraph{Step 3 : Conclusion.}
By induction, for any $m \geq0$ such that $\theta^{m} \rmax> \rstar$, we have built a sequence $\lt(\gamma_n(m)\rt)_{n \in \{1, \ldots N \}}$, such that \eqref{HR :1} and \eqref{HR :2} are satisfied.
We let $M \in \mathbb{N}$ be such that  $\theta^{M} \rmax> \rstar \geq \theta^{M+1} \rmax$ and define $\gamma_n :=\gamma_n(M)$.
By \eqref{HR :2}, we get \eqref{BorneGamma}.

Next, we pick $r \in [1,\rmax]$, and choose $m\leq M$ maximal such that $r < \theta^m R$.
Therefore, we get by enlarging the integration domain
\begin{align}\label{Num:2010}
\lt(\fint_{\Dom_r} \lt| u - \sum_{n=1}^N \gamma_n \lt( \taubar_n + \phiC_n\rt) \rt|^2 \rt)^{\frac{1}{2}}
\leq
\theta^{-1}\max \lt\{1, \frac{\rstar}{r}\rt\}
\lt(\fint_{\Dom_{\theta^m \rmax}} \lt| u - \sum_{n=1}^N \gamma_n \lt( \taubar_n + \phiC_n\rt) \rt|^2 \rt)^{\frac{1}{2}}.
\end{align}
By a triangle inequality, we immediately estimate the above \rhs as follows
\begin{equation}
\label{Num:2008}
\begin{aligned}
\lt(\fint_{\Dom_{\theta^m \rmax}} \lt| u - \sum_{n=1}^N \gamma_n \lt( \taubar_n + \phiC_n\rt) \rt|^2 \rt)^{\frac{1}{2}}
\leq~&
\lt(\fint_{\Dom_{\theta^m \rmax}} \lt| u - \sum_{n=1}^N \gamma_n(m) \lt( \taubar_n + \phiC_n\rt) \rt|^2 \rt)^{\frac{1}{2}}
\\
&+
\sum_{n=1}^N |\gamma_n - \gamma_n(m)|\lt(\fint_{\Dom_{\theta^m \rmax}}  \lt| \taubar_n + \phiC_n \rt|^2 \rt)^{\frac{1}{2}}.
\end{aligned}
\end{equation}
Recalling from Step 2 that $\gamma_n = \sum_{m=1}^{M} \tilde{\gamma}_n(m)$, we easily estimate the second \rhs as follows
\begin{equation*}
\begin{aligned}
\sum_{n=1}^N |\gamma_n - \gamma_n(m)|\lt(\fint_{\Dom_{\theta^m \rmax}}  \lt| \taubar_n + \phiC_n \rt|^2 \rt)^{\frac{1}{2}}
&\overset{\eqref{taubar_Homog},\eqref{Sublin**}}{\lesssim}
\sum_{n=1}^N \lt(\sum_{m'=m+1}^{M} |\tilde{\gamma}_n(m')|\rt)\lt(\theta^m \rmax\rt)^{\rhobar_n}
\\
&\overset{\eqref{Num:1001}}{\lesssim}
\sum_{n=1}^N  \theta^{m \rhobar_n} \lt(\sum_{m'=m+1}^{M} \theta^{m' (\rho-\rhobar_n)}\rt)
\lt(\fint_{\Dom_{\rmax}} \lt| u \rt|^2 \rt)^{\frac{1}{2}}
\\
&\lesssim \theta^{m\rho} \lt(\fint_{\Dom_{\rmax}} \lt| u \rt|^2 \rt)^{\frac{1}{2}}.
\end{aligned}
\end{equation*}
Inserting this as well as \eqref{HR :1} into \eqref{Num:2008}, which we use in turn in \eqref{Num:2010}, and recalling that $\theta^m \leq \theta^{-1} \max\{\rstar,r\}/\rmax$ yields \eqref{Num:119} in form of
\begin{equation}\label{Num:2009}
\begin{aligned}
\lt(\fint_{\Dom_r} \lt| u - \sum_{n=1}^N \gamma_n \lt( \taubar_n + \phiC_n\rt) \rt|^2 \rt)^{\frac{1}{2}}
&\lesssim
\max \lt\{1, \frac{\rstar}{r} \rt\}
\theta^{m\rho} \lt(\fint_{\Dom_{\rmax}} \lt| u \rt|^2 \rt)^{\frac{1}{2}}
\\
&\lesssim
\lt(\max \lt\{1, \frac{\rstar}{r}\rt\}\rt)^{1+\rho}
\lt(\frac{r}{\rmax} \rt)^\rho\lt(\fint_{\Dom_{\rmax}} \lt| u \rt|^2 \rt)^{\frac{1}{2}}.
\end{aligned}
\end{equation}

Finally, the remaining estimate \eqref{Num:119:lip} for $N\geq1$ follows from \eqref{Num:119}, \eqref{BorneGamma}, and the triangle inequality. Indeed, appealing to \eqref{Num:119} (with $\rho=\frac12(\rhobar_N+\rhobar_{N+1})\geq\rhobar_N$), \eqref{BorneGamma} together with \eqref{Sublin**} and $\fint_{D_r}|\taubar_N|^2=r^{2\rhobar_N}\fint_{D_1}|\taubar_N|^2$, we obtain for $r_*\leq r<R$
\begin{align*}
\lt(\fint_{\Dom_r} \lt| u - \sum_{n=1}^{N-1} \gamma_n \lt( \taubar_n + \phiC_n\rt) \rt|^2 \rt)^{\frac{1}{2}}
\leq&\lt(\fint_{\Dom_r} \lt| u - \sum_{n=1}^{N} \gamma_n \lt( \taubar_n + \phiC_n\rt) \rt|^2 \rt)^{\frac{1}{2}}+|\gamma_N|\lt(\fint_{\Dom_r} | \taubar_N + \phiC_N|^2 \rt)^{\frac{1}{2}}\\
\leq&C_0\lt(\lt(\frac{r}{\rmax} \rt)^{\frac12(\rhobar_N+\rhobar_{N+1})}+\lt(\frac{r}{\rmax} \rt)^{\rhobar_N}\lt(\lt(\fint_{D_1}|\taubar_N|^2\rt)^\frac12+\delta\rt)\rt) \lt(\fint_{\Dom_{\rmax}} |u|^2\rt)^{\frac{1}{2}},
\end{align*}
which implies \eqref{Num:119:lip} (for $r_*\leq r<R$, but the general claim easily follows).

\end{proof}

\section{Conditional quasi-optimal growth rates in weaker spatial norm: Arguments for Proposition \ref{Prop:Quench} and Corollary \ref{ThLiouville-2}}
\label{proof_prop_3.4}

\subsection{Conditional growth rates on the corner correctors: Argument for Proposition \ref{Prop:Quench}}\label{Sec:GeomSetting}

The difficulty that we encounter when building the corner corrector $\phiC_n$ via the ansatz \eqref{*ansatz} is that the remainder $\phitilde_{n}$ satisfies equation \eqref{*Def:psitilde2}, which is set on the whole unbounded domain $\Dom$ with a \rhs $h$ that is, in general, not in $\LL^2(\Dom)$.
Therefore, the Lax-Milgram theorem cannot be directly used.
To circumvent this issue, we appeal to a dyadic argument, in which we employ large-scale regularity results.

Departing
 from \eqref{*ansatz}, we split the \rhs on dyadic rings.
That is, we define $\phitilde_{n,m}$ as the Lax-Milgram solution to
\begin{equation}
\label{Def:psitilde2 :m}
\lt\{
\begin{aligned}
-\nabla \cdot a \nabla \phitilde_{n,m} =~& \nabla \cdot \lt( \eta_{\Boule,r_{m}} - \eta_{\Boule,r_{m-1}}\rt) h  && \dans \Dom,
\\
\phitilde_{n,m}=~&0 && \sur \Gamma,
\end{aligned}
\rt.
\end{equation}
for $r_m :=2^mr_0$ if $m \geq 0$ and $r_{m}=0$ if $m<0$, where $r_0$ will be fixed precisely afterwards.
This leads to the following decomposition of $\phitilde_{n}$ :
\begin{equation}\label{Num:4033}
\phitilde_{n} =  \sum_{m=0}^\infty \lt(\phitilde_{n,m} - \sum_{n'=1}^{N} \gamma_{n,m,n'} \lt(\taubar_{n'}+\phiC_{n'}\rt)\rt),
\end{equation}
where $N<n$ and the coefficients $\gamma_{n,m,n'}$ will be determined afterwards and correspond to a renormalization of $\phitilde_{n,m}$.
The addition of $a$-harmonic functions $\taubar_{n'}+\phiC_{n'}$ of lesser growth is not seen on the level of \eqref{*Def:psitilde2}, but is necessary in order to control the growth of $\phiC_{n,m}$.
In \eqref{Num:4033}, the index $N$ is chosen as being maximal such that all the renormalizing terms $\taubar_{n'}+\phiC_{n'}$ display growth rates $\rhobar_{n'}$ that are smaller than the desired $\rhobar_n-\nu$ that we expect for $\phiC_{n}$.
In other words, it is defined through
\begin{equation}\label{Def:N}
N :=\min\{n' \in \mathbb{N},\rhobar_{n}-\rhobar_{n'+1}-\nu < 0  \} < n.
\end{equation}

The aim is to establish that the series \eqref{Num:4033} converges absolutely.
To achieve this convergence, we use Lemma \ref{Lem:heterog} in order to finely estimate the behavior of the renormalized version of each $\nabla\phitilde_{n,m}$ in a suitable weighted $\LL^2$ norm.
Notice that, unless $n$ is small (depending on $\nu$ and $\omega$), in which case $N=0$, our strategy for estimating $\phiC_n$ requires us to assume that there already exist corner correctors for lesser indices $n'\leq N$, the growth of which is controlled.
Hence, we establish Proposition \ref{Prop:Quench} by induction on the index $n$.
The inductive argument, which is the core of the proof, is isolated in Lemma \ref{Lem:Iter2} below due to its technicality.

\begin{lemma}\label{Lem:Iter2}
We place ourselves under the assumptions of Proposition \ref{Prop:Quench}.
Let $n \in \NN \backslash\{0\}$, $m \in \mathbb{N}$ and define $N \in \mathbb{N}$ by~\eqref{Def:N}.
Next, we pick  $\tilde{\rho}\in (0,\rhobar_1)$, and $\rho \in (\rhobar_N,\rhobar_{N+1})$ sufficiently large so that
\begin{equation}\label{Contrainte :rho}
\rhobar_{n}-\rho-\nu < 0.
\end{equation}

Assume that there exist corner correctors $\phiC_{n'}$ for $n' \leq N$ such that \eqref{Sublin**} is satisfied for a given $\rstar \geq 2$.
Then, defining $r_m := 2^m\rstar$, there exist coefficients $\gamma_{n,m,n'}$ for $n' \in \{1,\dots,N\}$ such that the following renormalized version of $\nabla \phitilde_{n,m}$, for $\phitilde_{n,m}$ defined by \eqref{Def:psitilde2 :m}, satisfies the far-field estimate :
\begin{equation}\label{NearField}
\lt(\fint_{\Dom_r} \lt|\nabla \phitilde_{n,m} - \sum_{n'=1}^{N} \gamma_{n,m,n'} \nabla \lt( \taubar_{n'} + \phiC_{n'}\rt) \rt|^2 \rt)^{\frac{1}{2}}
\lesssim_{\omega, \lambda, n,\nu,\tilde{\nu},\rho}
r_m^{\rhobar_n-\rho-\nu}  \ln^{\tilde{\nu}}(r_m) r^{\rho-1} \pourtout r \in [r_0,r_{m}],
\end{equation}
and the near-field estimate :
\begin{equation}\label{FarField}
\lt(\fint_{\Dom_{2r} \backslash \Dom_{r}} \lt|\nabla \phitilde_{n,m} - \sum_{n'=1}^{N} \gamma_{n,m,n'} \nabla \lt( \taubar_{n'} + \phiC_{n'}\rt) \rt|^2 \rt)^{\frac{1}{2}}
\lesssim_{\omega,\lambda, n,\nu,\tilde{\nu},\tilde{\rho}} 
\begin{cases}
r_m^{\rhobar_{n}+\tilde{\rho}-\nu} \ln^{\tilde{\nu}}(r_m) r^{-\tilde{\rho}-1} & \si N=0,
\\
r_m^{\rhobar_n-\rhobar_{N}-\nu} \ln^{\tilde{\nu}}(r_m) r^{\rhobar_N-1} & \si N\geq 1,
\end{cases}
\end{equation}
for all $r \geq r_m$.
\end{lemma}

We remark that, following \cite{JosienRaithel_2019, Fischer_Raithel_2017, Raithel_2017}, in Lemma \ref{Lem:Iter2}, we have split the various contributions $\phitilde_{n,m}$ into near-field and far-field terms depending on $m$.
As we will see, the far-field contributions are handled directly with Lemma \ref{Lem:heterog}, whereas the near-field contributions are estimated via a dualized version of this argument.

In estimates \eqref{NearField} and \eqref{FarField}, the exponent $\rho = \rhobar_{N+1}$ cannot be reached.
However, assuming that we have already built and estimated the corrector $\phiC_1$, we can actually achieve the following improvement, which will prove useful
in the proof of Proposition \ref{Prop:Quench} for the case $\omega=2\pi$ and $\nu =1$ :
Under the assumptions of Lemma \ref{Lem:Iter2}, suppose that there exists a corner corrector $\phiC_{1}$ such that \eqref{Sublin**} is satisfied with $N=1$ for a given $\rstar \geq 2$.
Then, defining $r_m := 2^m\rstar$, there holds
\begin{equation}\label{NearField:lip}
\lt(\fint_{\Dom_r} |\nabla \phitilde_{1,m}|^2  \rt)^{\frac{1}{2}}
\lesssim_{\omega, \lambda,\nu,\tilde{\nu}}
r_m^{-\nu}  \ln^{\tilde{\nu}}(r_m) r^{\rhobar_1-1} \pourtout r \in [r_0,r_{m}],
\end{equation}
and
\begin{equation}\label{FarField:lip}
\lt(\fint_{\Dom_{2r} \backslash \Dom_{r}} |\nabla \phitilde_{1,m} |^2 \rt)^{\frac{1}{2}}
\lesssim_{\omega,\lambda,\nu,\tilde{\nu}} r_m^{2\rhobar_{1}-\nu} \ln^{\tilde{\nu}}(r_m) r^{-\rhobar_1-1}\quad\mbox{for all $r\geq r_m$.}
\end{equation}
These two above estimates are established in the proof of Lemma \ref{Lem:Iter2} below.

Equipped with Lemma \ref{Lem:Iter2}, we are in a position to establish Proposition \ref{Prop:Quench}.

\begin{proof}[Proof of Proposition \ref{Prop:Quench}]
The argument for obtaining \eqref{Num:4011-phiD} is exactly the same as for \eqref{Num:4011} and \eqref{Num:4010}.
The only thing to do is to replace $\taubar_n$ by the coordinate function $x_i$.
Therefore, in the sequel, we only show \eqref{Num:4011} and \eqref{Num:4010}.
In this proof, the constant $\tilde{\nu}'\lesssim_{\tilde{\nu}} 1$ may change from line to line.
	
We establish Proposition \ref{Prop:Quench} via induction on $n \in \NN\backslash\{0\}$, invoking at each step Lemma \ref{Lem:Iter2} in order to establish the absolute convergence of the gradient of the series \eqref{Num:4033}.

We assume first that $\omega < 2\pi$ or $\nu<1$ :
this has the beneficial consequence that we may choose $\tilde{\rho} < \rhobar_1$, so that
\begin{equation}\label{Bizarre2}
\tilde{\rho} + \rhobar_n \geq \nu \qquad \pourtout n \geq 1.
\end{equation}
(Indeed $\omega=2\pi$ corresponds to the case $\rhobar_1 = 1/2$; thus, in the case that $\nu<1$, we have $\nu-\rhobar_1<1/2$.)
Inequality \eqref{Bizarre2} appears crucial when dealing with the near-field contributions of the renormalization of $\nabla \phitilde_{n}$, \textit{cf.} Step 2 below.
We discuss the more delicate case $\omega=2\pi$ and $\nu=1$ at the end of the proof (in Step 3).
In that case, two stages are required : First, we establish a suboptimal growth rate on $\phiC_1$, which is then, in turn, used to obtain the desired quasi-optimal growth rate on $\phiC_1$.

\paragraph{Step 1 : From \eqref{Num:4011} to \eqref{Num:4010}.}
As a preliminary step, we explain how to get \eqref{Num:4010} from \eqref{Num:4011}.
Let $r \geq 1$.
By the Poincaré inequality into which we insert \eqref{Num:4011}, we have
\begin{equation}\label{Num:4040}
\lt( \fint_{\Dom_{2r} \backslash \Dom_{r}} \lt|\phitilde_n \rt|^2 \rt)^{\frac{1}{2}}
\leq \Cstar r^{\rhobar_n-\nu} \ln^{\tilde{\nu}'}(r+1).
\end{equation}
Then, recalling estimates \eqref{Num:40101} and \eqref{taubar_Homog}, we get
\begin{equation}\label{Num:4041}
\lt( \fint_{\Dom_{2r} \backslash \Dom_{r}}
\lt|(1-\eta_{\Boule,1})\lt( \etaup \phiup_i + \etabulk \phi_i + \etadown\phidown_i \rt) \partial_i \taubar_n\rt|^2\rt)^{\frac{1}{2}}
\lesssim r^{\rhobar_n-\nu} \ln^{\tilde{\nu}'}(r+1).
\end{equation}
Inserting \eqref{Num:4040} and \eqref{Num:4041} into the triangle inequality
applied to \eqref{*ansatz} yields \eqref{Num:4010}.

\paragraph{Step 2 : Induction.}
Set $n \in \mathbb{N} \backslash \{0\}$, and define $N$ by \eqref{Def:N} accordingly.
We assume that, for all $n' < n$, \eqref{Num:4011} holds replacing $n \rightsquigarrow n'$ (if $n = 1$, this is trivially true, since the statement is empty), and we establish that \eqref{Num:4011} holds for $n$.

First, we fix the constants $\gamma_{n,m,n'}$ in \eqref{Num:4033}.
By the inductive hypothesis, \eqref{Num:4011} and thus \eqref{Num:4010} hold for $n \rightsquigarrow n' \leq N$, since $N \leq n-1$.
Moreover, by the Poincare and Caccioppoli inequality for equation \eqref{Def:psi}, we also have
\begin{equation*}
\lt(\int_{D_1}\lt|\phiC_{n'} \rt|^2 \rt)^{\frac{1}{2}}\lesssim \lt(\int_{D_1}\lt|\nabla \phiC_{n'} \rt|^2 \rt)^{\frac{1}{2}}\lesssim \lt(\int_{D_2\setminus D_1}\lt|\phiC_{n'} \rt|^2 \rt)^{\frac{1}{2}}+ \lt(\int_{D_2}\lt|\nabla \taubar_{n'} \rt|^2 \rt)^{\frac{1}{2}}\stackrel{\eqref{Num:4011}}{\lesssim} \Cstar+1.
\end{equation*}
Hence, there is a constant $\Cstar$ so that
\begin{equation*}
\lt( \fint_{\Dom_{2^k}} \lt|\phiC_{n'} \rt|^2 \rt)^{\frac{1}{2}}
\stackrel{\eqref{Num:4010}}{\lesssim} 2^{-k} \Cstar\sum_{k'=0}^{k-1} 2^{k'(1+\rhobar_{n'}-\nu)} \ln^{\tilde{\nu}'}(2^{k'+1})
\lesssim \lt(\Cstar 2^{-k \nu} (k+1)^{\tilde{\nu}'}\rt) 2^{k\rhobar_{n'}}.
\end{equation*}
Since $\Cstar 2^{-k \nu} (k+1)^{\tilde{\nu}'}$ tends to $0$ when $k\uparrow \infty$, there exists a radius $r_0=\rstar \lesssim_{\Cstar,n,\tilde{\nu},\nu} 1$ such that \eqref{Sublin**} holds for $n \rightsquigarrow n'$, for any $n' \leq N$.
(Hereinafter, we neglect the dependence in $\rstar$ of our estimates.)
Hence we may apply Lemma \ref{Lem:Iter2} and associate coefficients $\gamma_{n,m,n'}$ for $n' \leq N$ to any function $\phitilde_{n,m}$ (which are defined by \eqref{Def:psitilde2 :m} with $h$ given in \eqref{*Def:h1h2}).

Let $k \in \mathbb{N}$.
We estimate separately the far-field and near-field contributions of the renormalization of $\phitilde_{n,m}$ on $\Dom_{2^{k+1}\rstar} \backslash \Dom_{2^k\rstar}$;
they are dealt with using \eqref{NearField} and \eqref{FarField}, respectively.
All the far-field contributions sum up to
\begin{align*}
\sum_{m=k+1}^\infty
\lt(\fint_{\Dom_{2^{k+1}\rstar}}
\lt|\nabla\phitilde_{n,m} - \sum_{n'=1}^N \gamma_{n,m,n'}\nabla \lt( \taubar_{n'} + \phiC_{n'}\rt)\rt|^2
\rt)^{\frac{1}{2}}
\overset{\eqref{NearField}}{\lesssim}~&
\sum_{m=k+1}^\infty  2^{m(\rhobar_n-\rho-\nu)}
(m+1)^{\tilde{\nu}} 2^{k(\rho-1)}
\\
\overset{\eqref{Contrainte :rho}}{\lesssim}~&
2^{k(\rhobar_n-\rho-\nu)} (k+1)^{\tilde{\nu}} 2^{k(\rho-1)}
\\
\lesssim~~&2^{k(\rhobar_n - 1 -\nu)} (k+1)^{\tilde{\nu}}.
\end{align*}
If $N=0$, all the near-field contributions sum up to\footnote{Here, we are suboptimal by one logarithmic term, but this logarithmic term will show up in the case $\omega=2\pi$ and $\nu=1$.} (recalling \eqref{Bizarre2})
\begin{align*}
\sum_{m=0}^k
\lt(\fint_{\Dom_{2^{k+1}\rstar} \backslash \Dom_{2^k\rstar}}
\lt|\nabla\phitilde_{n,m} \rt|^2
\rt)^{\frac{1}{2}}
\overset{\eqref{FarField}}{\lesssim}~&
\sum_{m=0}^k
2^{m(\rhobar_n+\tilde{\rho}-\nu)}(m+1)^{\tilde{\nu}} 2^{-k(\tilde{\rho}+1)}
\\
\overset{\eqref{Bizarre2}}{\lesssim}~& 2^{k(\rhobar_n+\tilde{\rho}-\nu)}(k+1)^{1+\tilde{\nu}} 2^{-k(\tilde{\rho}+1)}
\\
\lesssim~~&  2^{k(\rhobar_n-\nu-1)}(k+1)^{\tilde{\nu}+1}
\end{align*}
whereas, if $N\geq 1$, we have
\begin{align*}
\sum_{m=0}^k
\lt(\fint_{\Dom_{2^{k+1}\rstar} \backslash \Dom_{2^k\rstar}}
\lt|\nabla\phitilde_{n,m} - \sum_{n'=1}^N \gamma_{n,m,n'}\nabla \lt( \taubar_{n'} + \phiC_{n'}\rt)\rt|^2
\rt)^{\frac{1}{2}}
\overset{\eqref{FarField}}{\lesssim}~&
\sum_{m=0}^k
2^{m(\rhobar_n-\rhobar_N-\nu)} (m+1)^{\tilde{\nu}}
2^{k(\rhobar_N-1)}
\\
\overset{\eqref{Def:N}}{\lesssim}~&
2^{k(\rhobar_n-\rhobar_N-\nu)} (k+1)^{\tilde{\nu}+1}
2^{k(\rhobar_N-1)}
\\
\lesssim~~& 
2^{k(\rhobar_n-\nu-1)} (k+1)^{\tilde{\nu}+1} .
\end{align*}
In any case, by the triangle inequality, we obtain the absolute convergence of the gradient of the \rhs of \eqref{Num:4033} on any bounded subdomain of $\Dom$ not containing a neighborhood of $0$, as well as the following estimate :
\begin{equation}\label{Num:4032}
\lt(\fint_{\Dom_{2^{k+1}\rstar} \backslash \Dom_{2^k\rstar}}
|\nabla\phitilde_{n}|^2
\rt)^{\frac{1}{2}}
\lesssim 
2^{k(\rhobar_n-\nu-1)} (k+1)^{\tilde{\nu}+1}.
\end{equation}
As a consequence, this establishes \eqref{Num:4011} for all $r \geq \rstar$.
Moreover, \eqref{Num:4011} can be proved for $r \in [1,\rstar]$ by the same reasoning as above, but using only the near-field estimate (at the price of a constant depending algebraically in $\rstar$).
This concludes the induction step.	

\paragraph{Step 3 : The special case $\omega=2\pi$ and $\nu=1$.}
Notice that the only case that cannot be treated as in Step 2 is the case $n=1$, because we cannot choose $\tilde{\rho}<\rhobar_1$ such that \eqref{Bizarre2} holds for $n=1$.
However, by the proof above, we have already dealt with the case $\omega=2\pi$ and $\nu=1/2$.
Hence, we have established \eqref{Num:4011} and thus \eqref{Num:4010}, replacing $\nu \rightsquigarrow 1/2$.

As a consequence, we may find a radius $\rstar=:r_0$ such that
\eqref{Sublin**} holds for $N=1$.
Hence, we can appeal to \eqref{NearField:lip} and \eqref{FarField:lip} which in this case (that is $\rhobar_1=\frac12$ and $\nu=1$) read
\begin{align}\label{NearField-1}
\lt(\fint_{\Dom_r} \lt|\nabla \phitilde_{1,m}\rt|^2 \rt)^{\frac{1}{2}}
\lesssim_{\omega,\lambda,N,\tilde{\nu}}~&
r_m^{-1}  \ln^{\tilde{\nu}}(r_m) r^{-\frac{1}{2}}&&\pourtout r \in [r_0,r_{m}],
\\
\label{FarField-1}
\lt(\fint_{\Dom_{2r} \backslash \Dom_{r}} \lt|\nabla \phitilde_{1,m} \rt|^2 \rt)^{\frac{1}{2}}
\lesssim_{\omega,\lambda,N,\tilde{\nu}}~& \ln^{\tilde{\nu}}(r_m) r^{-\frac{3}{2}} && \pourtout r \geq r_m.
\end{align}
Thus, we can execute the same argument as in Step 2 and get that $\phitilde_{1}$ satisfies \eqref{Num:4011}.
\end{proof}

\begin{proof}[Proof of Lemma \ref{Lem:Iter2}]
First, we show that
\begin{equation}\label{Num:4021}
\lt(\int_{\Dom} \lt|\nabla \phitilde_{n,m} \rt|^2\rt)^{\frac{1}{2}} 
\lesssim
r_m^{\rhobar_n-\nu} \ln^{\tilde{\nu}}(r_m).
\end{equation}	
The latter is a consequence of the energy estimate applied to \eqref{Def:psitilde2 :m}, combined with
\begin{equation}\label{Num:4025}
\lt(\int_{\Dom} \lt|\lt( \eta_{\Boule,r_{m}} - \eta_{\Boule,r_{m-1}}\rt) h \rt|^2\rt)^{\frac{1}{2}}
\lesssim
r_m^{\rhobar_n-\nu} \ln^{\tilde{\nu}}(r_m),
\end{equation}
which itself derives from inserting \eqref{Num:40101} and \eqref{taubar_Homog} into \eqref{*Def:h1h2}.

The proof of  \eqref{NearField} and  \eqref{FarField} is divided into three steps : 	
In Step 1, we appeal to Lemma \ref{Lem:heterog} to renormalize $\phitilde_{n,m}$ and establish the far-field estimate \eqref{NearField} for $r \in [r_0, r_{m-2}]$.
In Step 2, we extend \eqref{NearField} to the whole range $r \in [\max\{r_0,r_{m-2}\}, r_m]$.
In Step 3, we make use of a dualization argument and prove the near-field estimate \eqref{FarField}.

Finally, in Step~4 we gather the needed adjustments to obtain the estimates \eqref{NearField:lip} and \eqref{FarField:lip}.

\paragraph{Step 1 : Far-field estimate for  $r \in [r_0, r_{m-2}]$.} 
We may assume that $m \geq 2$, since otherwise the statement is empty (since $r_{m-2} < r_0$).
By Lemma \ref{Lem:heterog}, there exist coefficients $(\gamma_{n,m,n'})_{n' \in \{ 1, \ldots, N \}}$ such that for any $r \in [r_0,r_{m-2}]$, we obtain \eqref{NearField} in form of
\begin{align*}
\lt(\fint_{\Dom_r} \lt| \nabla \phitilde_{n,m} - \sum_{n'=1}^N \gamma_{n,m,n'} \nabla \lt(\taubar_{n'} + \phiC_{n'}\rt) \rt|^2 \rt)^{\frac{1}{2}}
\overset{\eqref{Num:119_bis}}{\lesssim}~&
\lt(\frac{r}{r_{m-2}} \rt)^{\rho-1} \lt(\fint_{\Dom_{r_{m-1}}} |\nabla \phitilde_{n,m}|^2\rt)^{\frac{1}{2}}
\\
\overset{\eqref{Num:4021}}{\lesssim}~& r_m^{\rhobar_n-\rho-\nu}
\ln^{\tilde{\nu}}(r_m) r^{\rho-1}.
\end{align*}
Moreover, taking into account \eqref{BorneGamma} and appealing to the Poincaré inequality as well as \eqref{Num:4021} yields
\begin{equation}\label{Num:4022}
\lt|\gamma_{n,m,n'}\rt| \overset{\eqref{BorneGamma}}{\lesssim} r_{m}^{-\rhobar_{n'}} 
\lt(\fint_{\Dom_{r_{m-1}}} \lt| \phitilde_{n,m} \rt|^2\rt)^{\frac{1}{2}}
\lesssim
r_{m}^{1-\rhobar_{n'}} 
\lt(\fint_{\Dom_{r_{m-1}}} \lt| \nabla \phitilde_{n,m} \rt|^2\rt)^{\frac{1}{2}}
\overset{\eqref{Num:4021}}{\lesssim}
r_m^{\rhobar_n-\rhobar_{n'}-\nu} \ln^{\tilde{\nu}}r_m.
\end{equation}

\paragraph{Step 2 :  Far-field estimate for  $r \in [\max\{r_0,r_{m-2}\}, r_m]$.}
We shall establish that, for any  $m\geq 2$ it holds that
\begin{equation}\label{Num:4023}
\lt(\fint_{\Dom_r} \lt| \sum_{n'=1}^N \gamma_{n,m,n'} \nabla \lt(\taubar_{n'} + \phiC_{n'}\rt) \rt|^2 \rt)^{\frac{1}{2}}
\lesssim r_m^{\rhobar_n-\rhobar_{N}-\nu} \ln^{\tilde{\nu}}(r_m) r^{\rhobar_N-1}.
\end{equation}
Inserting the above \eqref{Num:4023} and \eqref{Num:4021} into the triangle inequality yields the desired \eqref{NearField}.
Notice that in the case $m \leq 1$, \eqref{Num:4021} is sufficient for obtaining \eqref{NearField} and $\gamma_{n,m,n'} =0$.

We obtain \eqref{Num:4023} by inserting \eqref{Num:4022}, a Caccioppoli inequality, and, finally, \eqref{taubar_Homog} and \eqref{Sublin**}
together with the assumption $1\lesssim \frac{r}{r_m}$
into the triangle inequality as follows :
\begin{equation*}
\begin{aligned}
\lt(\fint_{\Dom_r} \lt| \sum_{n'=1}^N \gamma_{n,m,n'} \nabla \lt(\taubar_{n'} + \phiC_{n'}\rt) \rt|^2 \rt)^{\frac{1}{2}}
\overset{\eqref{Num:4022}}{\lesssim}~~&
\sum_{n'=1}^N
r_m^{\rhobar_n-\rhobar_{n'}-\nu} \ln^{\tilde{\nu}}(r_m) r^{-1}\lt(\fint_{\Dom_{2r}} \lt|\taubar_{n'} + \phiC_{n'}\rt|^2 \rt)^{\frac{1}{2}}
\\
\overset{\eqref{taubar_Homog},\eqref{Sublin**}}{\lesssim}~&
\sum_{n'=1}^N
r_m^{\rhobar_n-\rhobar_{n'}-\nu} \ln^{\tilde{\nu}}(r_m) r^{\rhobar_{n'}-1}
\lesssim r_m^{\rhobar_n-\rhobar_{N}-\nu} \ln^{\tilde{\nu}}(r_m) r^{\rhobar_N-1}.
\end{aligned}
\end{equation*}

\paragraph{Step 3 : Near-field estimate for $r \geq r_m$.}
We claim that, for $r \geq r_m$,
\begin{equation}\label{Num:4024}
\lt(\fint_{\Dom_{2r} \backslash \Dom_r} \lt| \nabla \phitilde_{n,m} \rt|^2 \rt)^{\frac{1}{2}}
\lesssim
r_m^{\rhobar_{n}+\tilde{\rho}-\nu} \ln^{\tilde{\nu}}(r_m) r^{-\tilde{\rho}-1}.
\end{equation}
Inserting \eqref{Num:4024} and \eqref{Num:4023} into the triangle inequality yields the desired estimate \eqref{FarField} (recall that, if $N=0$, there is no renormalization term).

Here comes the argument for \eqref{Num:4024}.
We proceed by duality and pick a vector field $f \in \LL^2(\Dom)$ with support in $\Dom_{2r} \backslash \Dom_{r}$.
We define $v$ as the Lax-Milgram solution to
\begin{equation*}
-\nabla \cdot a \nabla v = \nabla \cdot f \dans \Dom,
\qquad \text{with} \quad
v=0 \sur \Gamma.
\end{equation*}
Thus, since $a$ is symmetric thanks to Assumption \ref{Assump1}, we may rewrite
\begin{align*}
\int_{\Dom} \nabla \phitilde_{n,m} \cdot f
=
-\int_{\Dom} \nabla \phitilde_{n,m} \cdot a \nabla v 
\overset{\eqref{Def:psitilde2 :m}}{=}
\int_{\Dom} \lt( \eta_{\Boule,r_{m}} - \eta_{\Boule,r_{m-1}}\rt) h \cdot \nabla v.
\end{align*}
Since $v$ is a-harmonic inside $\Dom_{r}$, we may appeal to Lemma \ref{Lem:heterog} (with $N \rightsquigarrow 0$, $\rmax \rightsquigarrow r$, $r \rightsquigarrow r_m$, and $\rho \rightsquigarrow \tilde{\rho} $) to the effect of
\begin{align}\label{est:L51:s3}
\lt(\int_{\Dom_{r_m}} |\nabla v|^2 \rt)^{\frac{1}{2}}
\overset{\eqref{Num:119_bis}}{\lesssim} r_m^{\tilde{\rho}} r^{-\tilde{\rho}}
\lt(\int_{\Dom_{r}} |\nabla v|^2 \rt)^{\frac{1}{2}}.
\end{align}
Moreover, appealing to the energy estimate, we get
\begin{align*}
\int_{\Dom_{r}} |\nabla v|^2 \lesssim \int_{\Dom_{2r} \backslash \Dom_r} |f|^2.
\end{align*}
Therefore, a duality argument produces \eqref{Num:4024} through
\begin{equation}\label{Num:4026}
\lt(\int_{\Dom_{2r} \backslash \Dom_r} \lt|\nabla\phitilde_{n,m}\rt|^2\rt)^{\frac{1}{2}}
\lesssim
r_m^{\tilde{\rho}} r^{-\tilde{\rho}} \lt(\int_{\Dom_{r_m}} \lt|\lt( \eta_{\Boule,r_{m}} - \eta_{\Boule,r_{m-1}}\rt)h\rt|^2\rt)^{\frac{1}{2}}
\overset{\eqref{Num:4025}}{\lesssim}
r_m^{\rhobar_n+\tilde{\rho}-\nu} \ln^{\tilde{\nu}}(r_m) r^{-\tilde{\rho}}.
\end{equation}

\paragraph{Step 4 : Estimates \eqref{NearField:lip} and \eqref{FarField:lip}.}
We start with the argument for \eqref{NearField:lip}.
Since we assume that there exists a corner corrector $\phiC_{1}$ such that \eqref{Sublin**} is satisfied with $N=1$, we can apply estimate \eqref{Num:119:lip} (with $N=1$) on $\phitilde_{1,m}$ in $\Dom_{r_{m-1}}$ and, with applications of the Caccioppoli and Poincar\'e inequalities, obtain
\begin{align*}
\lt(\fint_{\Dom_r} \lt| \nabla \phitilde_{1,m}  \rt|^2 \rt)^{\frac{1}{2}}\lesssim~&
\lt(\frac{r}{r_{m-2}} \rt)^{\rhobar_1-1} \lt(\fint_{\Dom_{r_{m-1}}} |\nabla \phitilde_{1,m}|^2\rt)^{\frac{1}{2}}
\overset{\eqref{Num:4021}}{\lesssim}~ r_m^{-\nu}
\ln^{\tilde{\nu}}(r_m) r^{\rhobar_1-1},
\end{align*}
for any $r \in [r_0,r_{m-2}]$.
The proof of \eqref{FarField:lip} follows by the same argument as for \eqref{Num:4024} with the only change being that we use \eqref{Num:119:lip} (with $N=1$).
The effect of this is that instead of \eqref{est:L51:s3}, we have 
\begin{align*}
\lt(\int_{\Dom_{r_m}} |\nabla v|^2 \rt)^{\frac{1}{2}}
{\lesssim} r_m^{\rhobar_1} r^{-\rhobar_1}
\lt(\int_{\Dom_{r}} |\nabla v|^2 \rt)^{\frac{1}{2}}
\end{align*}
which leads to the desired estimate \eqref{FarField:lip}.
\end{proof}

\subsection{Conditional Liouville principle : Argument for Corollary \ref{ThLiouville-2}}\label{Sec:ProofLiouville}

The Liouville principle is a straightforward consequence of the following reformulation of Lemma~\ref{Lem:heterog} as an excess decay statement :

\begin{corollary}[Excess decay]\label{C:excessdecay}
Assume that the coefficient field $a$ satisfies Assumptions \ref{Assump1} and \ref{AssumpId}.
Let $N \in \mathbb{N}$ and $\rho \in (\rhobar_{N},\rhobar_{N+1})$.
Then, for $\delta \ll_{\omega,\lambda,N,\rho} 1$ and $C_0 \lesssim_{\omega,\lambda,N,\rho} 1$, the following property holds :
Let $\rstar \geq 1$.
Assume that the extended correctors $\lt(\phi,\sigma\rt)$ and the corner correctors $\phiC_n$, for $n \in \{1,\dots,N\}$, satisfy \eqref{Sublin**}.
Let $\rmax \geq 1$ and $u$ satisfy \eqref{Num:011}. Then, the tilt-excess defined by \eqref{excess} satisfies:
 \begin{equation}\label{est:Exc}
 \Exc_N(u;r)
\leq C_0
\max\lt\{1, \lt(\frac{\rstar}{r}\rt)^{2(1+\rho)}\rt\}
\lt(\frac{r}{\rmax} \rt)^{2\rho-2}  \Exc_N(u;R)\qquad\forall r\in[1,R].
 \end{equation}
Furthermore, the infimum in \eqref{excess} is actually attained and we denote by $\gamma^r\in\R^N$ an optimal choice (for the radius $r$).
Then, for $r\in[r_*,R]$ and $n \in \{ 1, \ldots, N\}$, it holds that
 \begin{equation}\label{eq:gammarR}
  R^{\overline \rho_n-1}|\gamma_n^r-\gamma_n^R|\lesssim_{\omega,\lambda,N} \lt(\Exc_N(u,R)\rt)^{\frac{1}{2}}.
 \end{equation}
\end{corollary}

\begin{proof}[Proof of Corollary~\ref{C:excessdecay}]
Estimate \eqref{est:Exc} is a direct consequence of Lemma~\ref{Lem:heterog} applied to the solution $u-\sum_{n=1}^N\gamma_n^R(\taubar_n+\phiC_n)$ of \eqref{Num:011}, \textit{cf.}~\eqref{Num:119_bis}.
Hence it is left to establish the estimate \eqref{eq:gammarR}.\\

Instead of \eqref{eq:gammarR}, we prove the more general
\begin{equation}\label{est:diffgammarR}
	\sum_{n=1}^N {R'}^{2\rhobar_n-2}|\gamma_n^r-\gamma_n^{R'}|^2
	\lesssim_{\omega,N,\lambda}  \Exc_N(u;R')  \pourtout  r^*\leq r<R'\leq R. 
\end{equation}
We first consider the case $r<R'\leq 2r$.
Using $r^{2\rhobar_n}\fint_{D_1}|\taubar_n|^2=\fint_{D_r}|\taubar_n|^2$ and the $L^2(D_r)$-orthogonality of the functions $\taubar_n$,
we obtain
     \begin{align}\label{est:diffgammarR:s1}
     \sum_{n=1}^N {r}^{2\rhobar_n}(\gamma_n^r-\gamma_n^{R'})^2
     &\lesssim_{\omega,N} \sum_{n=1}^N \fint_{D_r}|(\gamma_n^r-\gamma_n^{R'})\taubar_n|^2
     = \fint_{D_r}\biggl|\sum_{n=1}^N(\gamma_n^r-\gamma_n^{R'})\taubar_n\biggr|^2 \notag
     \\
     &\lesssim  \fint_{D_r}\biggl|\sum_{n=1}^N(\gamma_n^r-\gamma_n^{R'})(\taubar_n+\phiC_n)\biggr|^2
     +
     \fint_{D_r}\biggl|\sum_{n=1}^N(\gamma_n^r-\gamma_n^{R'})\phiC_n\biggr|^2.
     \end{align}
     A combination of the Cauchy-Schwarz inequality and \eqref{Sublin**} yields
     \begin{equation*}
     \fint_{D_r}\biggl|\sum_{n=1}^N(\gamma_n^r-\gamma_n^{R'})\phiC_n\biggr|^2
     \leq 
    \left( \sum_{n=1}^N {r}^{2\rhobar_n}(\gamma_n^r-\gamma_n^{R'})^2\right)\left(
     \sum_{n=1}^Nr^{-2\rhobar_n}\fint_{D_r}|\phiC_n|^2\right)
     \lesssim \delta^2 \sum_{n=1}^Nr^{2\rhobar_n}(\gamma_n^r-\gamma_n^{R'})^2.
     \end{equation*}
     Hence, the second term on the \rhs in \eqref{est:diffgammarR:s1} can  be absorbed into the \lhs provided $\delta \ll_{\omega,N} 1$. 
	Thus, it remains to estimate the first term :
We smuggle in $u$ and obtain with help of the Poincar\'e inequality, using as well $r \geq R'/2$, the relation
     \begin{align*}
      \fint_{D_r}\biggl|\sum_{n=1}^N(\gamma_n^r-\gamma_n^{R'})(\taubar_n+\phiC_n)\biggr|^2
      \lesssim r^2\Exc_N(u;r)+{R'}^2\Exc_N(u;R')
      \stackrel{\eqref{est:Exc}}{\lesssim} {R'}^2\Exc_N(u;R').
     \end{align*}
     This proves the claimed estimate \eqref{est:diffgammarR} in the case $r<R'\leq 2r$. 
     
	The general case follows by a standard dyadic argument,
	similar to Step 3 of the proof of Lemma \ref{Lem:heterog} :
	Consider $r^*\leq r<R'\leq R$ and fix
	$M\in \mathbb N$ such that $2^{-(M+1)}R'<r\leq 2^{-M}R'$.
	Then, we have for all $n\in\{1,\dots,N\}$
     \begin{align*}
     &(2^{-M}R')^{\rhobar_n-1}|\gamma_n^r-\gamma_n^{2^{-M}R'}|\lesssim \lt(\Exc_N(u;2^{-M}R')\rt)^{\frac{1}{2}},
      \\
      & (2^{-m}R')^{\rhobar_n-1}|\gamma_n^{2^{-(m+1)}R'}-\gamma_n^{2^{-m}R'}|\lesssim  \lt(\Exc_N(u;2^{-m}R')\rt)^{\frac{1}{2}},
     \end{align*}
     for $m \leq M$.
Hence, for any $n\in\{1,\dots,N\}$, \eqref{est:Exc} with $\rho=\frac12(\rho_N+\rho_{N+1})$ yields
     \begin{align*}
	R'^{\rhobar_n-1}|\gamma_n^r-\gamma_n^{R'}|
	\lesssim&\sum_{m=0}^M 2^{(\rhobar_n-1) m} \lt(\Exc_N(u;2^{-m}R')\rt)^{\frac{1}{2}}
	\overset{\eqref{est:Exc}}{\lesssim} \sum_{m=0}^M 2^{(\rhobar_n-\rho) m} \lt(\Exc_N(u;R')\rt)^{\frac{1}{2}}
	\lesssim \lt(\Exc_N(u;R')\rt)^{\frac{1}{2}},
     \end{align*}
which finishes the proof.

\end{proof}

\begin{proof}[Proof of Corollary \ref{ThLiouville-2}]
Let $\tilde{\rho} \in (\rho,\rhobar_{N+1})$ and $1 \leq r \leq R/2$. 
By Proposition \ref{Prop:Quench}, there exists a radius $\rstar\geq 1$ above which \eqref{Sublin**} holds, so that we may apply Corrollary \ref{C:excessdecay} to $u$, which is $a$-harmonic in $\Dom_{R}$.
Hence, by the Caccioppoli inequality and \eqref{Hypo :ThLiouville}, we get
\begin{align*}
\Exc_N(u;r)\leq C_0\max\lt\{1, \lt(\frac{\rstar}{r}\rt)^{2+2\tilde{\rho}}\rt\}
\lt(\frac{r}{R}\rt)^{2\tilde{\rho}-2}\fint_{\Dom_R} |\nabla u|^2\lesssim C_0\max\lt\{1, \lt(\frac{\rstar}{r}\rt)^{2+2\tilde{\rho}}\rt\}
\frac{r^{2\tilde{\rho}-2}}{R^{2\tilde\rho}}\fint_{\Dom_{2R}} | u|^2\overset{\eqref{Hypo :ThLiouville}}{\underset{R\uparrow\infty}{\to}}0
\end{align*}
and thus $\Exc_N(u;r)=0$ and $u=\sum_{n=1}^N\gamma_n^r(\taubar_n+\phiC_n)$ on $D_r$ for any $r\geq1$.
Moreover, it follows from \eqref{eq:gammarR} that the coefficients $\gamma_n^r$ do not depend on $r$ for $r\geq r_*$ (since $|\gamma_n^r-\gamma_n^{R}|\lesssim R^{-(\rhobar_n-1)}\lt(\Exc_N(u,R)\rt)^{\frac{1}{2}} = 0$ for all $r_*\leq r<R$). This concludes the proof.
\end{proof}

\section{Large-scale Lipschitz regularity and quasi-optimal corner corrector estimates : Proofs of Theorems \ref{ThAL} and \ref{Th:OptiGR} }

\label{main_thms}

This section contains the proofs of Theorems \ref{ThAL} and \ref{Th:OptiGR}.
As previously discussed in Section \ref{strategy}, the plan is to convert the conditional results Lemma \ref{Lem:heterog} and Proposition \ref{Prop:Quench} into Theorems \ref{ThAL} and \ref{Th:OptiGR}, respectively,
where Proposition \ref{Prop:Quench} must also be localized.
In the course of the proof of Theorem \ref{Th:OptiGR}, we will make use of an $\LL^\infty$-like annealed estimate, which we derive as a consequence of Theorem \ref{ThAL}.
In both proofs, the constant $\tilde{\nu}'\lesssim_{\tilde{\nu}} 1$ may change from line to line.

\subsection{Large-scale regularity at the corner: Proof of Theorem \ref{ThAL}}

\begin{proof}[Proof of Theorem \ref{ThAL}]	
	The proof of Theorem \ref{ThAL} is divided in three steps :
	In Step 1, we establish
	\begin{align}\label{Borne:HalfCorr_Q}
	\lt(\fint_{\Dom_r}
	\lt|\lt(\phi,\phiup,\phidown,\sigma,\sigup,\sigdown \rt)\rt|^2\rt)^{\frac{1}{2}}
	\leq \Cstar
	r^{1-\nu} \ln^{\tilde{\nu}'}(r+1)
	\quad \pourtout r \geq 1,
	\end{align}
	where the random constant $\Cstar \geq 2$ satisfies \eqref{Borne_CN-}.
	The proof relies on a classical manipulation, based on estimates in weighted $\ell^p$ spaces.
	In Step 2, using Proposition \ref{Prop:Quench}, we show that  there exist corner correctors $\phiC_n$, for $n \in \mathbb{N} \setminus \{ 0 \}$, and Dirichlet correctors $\phiD_j$, for $j =1,2$,
	that satisfy \eqref{Num:4011}, \eqref{Num:4010}, and \eqref{Num:4011-phiD}, where $\tilde{\nu}'\lesssim_{\tilde{\nu}} 1$ and $\Cstar \geq 2$ satisfies \eqref{Borne_CN-} for $N \rightsquigarrow n$.
	Last, in Step 3, we prove \eqref{Num:301_L2-1}.
	
	\paragraph{Step 1: Argument for \eqref{Borne:HalfCorr_Q}.}
	Since the proof of \eqref{Borne:HalfCorr_Q} is similar for $\phi, \phiup, \phidown, \sigma, \sigup$, and $\sigdown$, which satisfy \eqref{CorrSubDef} or \eqref{Borne:HalfCorr}, we only provide it for $\phi$ (for $\tilde{\nu}':=\tilde{\nu}+2$).
	Appealing to a dyadic argument and the triangle inequality, we obtain \eqref{Borne:HalfCorr_Q} in form of
	\begin{align*}
	\langl  \lt(\sup_{r \geq 1}
	\frac{1}{r^{1-\nu} \ln^{\tilde{\nu}+2}(r+1) }
	\lt(\fint_{\Dom_r}
	\lt|\phi\rt|^2\rt)^{\frac{1}{2}}
	\rt)^p 
	\rangl
	\lesssim~&
	\sum_{k=1}^\infty
	\langl \lt(
	\frac{1}{2^{k(1-\nu)} k^{\tilde{\nu}+2}}
	\lt(\fint_{\Dom_{2^k}}
	\lt|\phi\rt|^2\rt)^{\frac{1}{2}}
	\rt)^p 
	\rangl
	\overset{\eqref{CorrSubDef}}{\lesssim}
	\sum_{k=1}^\infty \frac{1}{k^{2p}} \lesssim 1.
	\end{align*}
	
	\paragraph{Step 2: Argument for \eqref{Num:4011}, \eqref{Num:4010}, and \eqref{Num:4011-phiD} with $\Cstar \geq 2$ satisfying \eqref{Borne_CN-}.}
	Since \eqref{Borne:HalfCorr_Q} holds, if we replace the coefficient field by a rescaled version $a\rightsquigarrow a_\epsilon :=a(\cdot/\epsilon)$, then the extended correctors (either half-space or not) are rescaled according to the rule $\phi \rightsquigarrow \phi_\epsilon=\epsilon \phi(\cdot/\epsilon)$.
	Thus, \eqref{Borne:HalfCorr_Q} is turned into
	\begin{equation}\label{Borne:HalfCorr_Q-eps}
	\begin{aligned}
	\lt(\fint_{\Dom_r}
	\lt|\lt(\phi_\epsilon, \phiup_\epsilon, \phidown_\epsilon, \sigma_\epsilon, \sigup_\epsilon, \sigdown_\epsilon \rt)\rt|^2\rt)^{\frac{1}{2}}
	&\leq \Cstar
	\epsilon^\nu r^{1-\nu} \ln^{\tilde{\nu}'}(r/\epsilon)
	\\
	&\leq \Cstar
	\epsilon^\nu \ln^{\tilde{\nu}'}(\epsilon^{-1}) r^{1-\nu} \ln^{\tilde{\nu}'}(r),
	\text{ provided } \ln(\epsilon^{-1}), \ln(r) \geq 2.	
	\end{aligned}
	\end{equation}
	Hence (as $\Cstar \geq 2$) we may choose $q \gg 1$ such that $\epsilon := (\Cstar)^{-q}$ satisfies
	\begin{align*}
	\Cstar
	\epsilon^\nu \ln^{\tilde{\nu}'}(\epsilon^{-1})
	= q^{\tilde{\nu}'}\Cstar^{1-q\nu} \ln^{\tilde{\nu}'}(\Cstar) < 1.
	\end{align*}
For such a choice of $q$ and $\epsilon$, $\lt(\phi_\epsilon, \phiup_\epsilon, \phidown_\epsilon, \sigma_\epsilon, \sigup_\epsilon, \sigdown_\epsilon \rt)$ satisfy \eqref{Num:40101}, where we replace $\tilde{\nu} \rightsquigarrow \tilde{\nu}'$ (since $\Cstar$ is finite almost surely, we have $\epsilon>0$ almost surely).
	As a consequence, we may apply Proposition~\ref{Prop:Quench} to the rescaled coefficient field $a_\epsilon$, which therefore possesses corner correctors $\phiC_{\epsilon,n}=\epsilon^{\rhobar_n} \phiC_n(\cdot/\epsilon)$ and Dirichlet correctors $\phiD_\epsilon =\epsilon \phiD(\cdot/\epsilon)$ that satisfy estimates \eqref{Num:4011}, \eqref{Num:4010} and \eqref{Num:4011-phiD} for some $\tilde{\nu}'$ and $\Cstar\rightsquigarrow\tilde{\Cstar} \lesssim_{\omega,\lambda,\nu,\tilde{\nu},n} 1$.
	Rescaling back these estimates yields  \eqref{Num:4011}, \eqref{Num:4010}, and \eqref{Num:4011-phiD} with $\tilde{\Cstar} \lesssim_{\omega,\lambda,\nu,\tilde{\nu},n} \Cstar^\frac2\nu$ with $\Cstar$ from \eqref{Borne:HalfCorr_Q}.
	Since $\Cstar$ satisfies \eqref{Borne_CN-} the random constant $\tilde{\Cstar}$ also satisfies \eqref{Borne_CN-}.
	
	\paragraph{Step 3: Argument for \eqref{Num:301_L2-1}.}
	By \eqref{Borne:HalfCorr_Q}, and \eqref{Num:4011}, and \eqref{Num:4010}, (using the Poincaré and Caccioppoli estimates for estimating $\lt( \fint_{\Dom_1} \lt|\phiC_n\rt|^2  \rt)^{\frac{1}{2}}$),
we may select a radius $\rstar \geq 1$ satisfying \eqref{Borne_CN-} for $\Cstar \rightsquigarrow \rstar$ such that \eqref{Sublin**} is satisfied, so that we may apply Lemma~\ref{Lem:heterog}.
	This provides constants $\gamma_1,\cdots,\gamma_{N}$ such that,
	\begin{equation}\label{Num:209-1}
	\begin{aligned}
	\lt(\fint_{\Dom_{r}} \lt| \nabla u - \sum_{n=1}^{N} \gamma_n \nabla\lt( \taubar_n +  \phiC_n\rt) \rt|^2 \rt)^{\frac{1}{2}}
	& \overset{\eqref{Num:119:lip}}{\lesssim_{\omega,\lambda,N,\rho}}
	\lt(1+\lt(\frac{\rstar}{r}\rt)^{1+\rhobar_{N+1}}\rt)\lt(\frac{r}{R}\rt)^{\rhobar_{N+1}-1}\lt(\fint_{\Dom_{R}} |\nabla u|^2\rt)^{\frac{1}{2}}. 
	\end{aligned}
	\end{equation}
\end{proof}

\subsection{$\LL^\infty$-like annealed estimates}\label{Sec:Linfty}

We first state a direct consequence of \cite[Th.\ 1 \& Th.\ 2]{Fischer_Raithel_2017}:
\begin{corollary}\label{C:LSL}
	Let $x$ and $R$ be such that $2\leq R \leq |x|/2$. 
	Under Assumptions \ref{Assump1}, \ref{AssumpId}, and \ref{Assump2}, if $u$ satisfies
	\begin{equation}\label{Eq :u-aharmLSLC}
	-\nabla \cdot a \nabla u = \nabla \cdot h \dans \Dom_R(x) \qquad \et u = 0 \sur \Gamma_R(x),
	\end{equation}
	then there holds
	\begin{align}\label{M:0002}
	\langl \lt( \fint_{\Dom_{1}(x)} \lt| \nabla u  \rt|^2 
	\rt)^{\frac{p}{2}} \rangl^{\frac{1}{p}}
	\lesssim_{\Xi, p, p'}
	\langl  \lt(\fint_{\Dom_R(x)} |\nabla u|^2\rt)^{\frac{p'}{2}} \rangl^{\frac{1}{p'}}+ \ln(R) \sup_{x' \in \Dom_R(x)} \langl\lt(\fint_{\Dom_1(x')}|h|^2\rt)^{\frac{p'}{2}} \rangl^{\frac{1}{p'}},
	\end{align}
	for $p \in [1, \infty)$ and $p'>p$.
\end{corollary}
Notice that, in Corollary \ref{C:LSL}, the geometrical situation involves only flat boundaries and not the tip of the corner.
A similar result can be derived from Theorem \ref{ThAL} :
\begin{corollary}\label{Cor:Annealed}
	Let $r \geq 2$ and $K \in \mathbb{N}\backslash\{0\}$.
	Under Assumptions \ref{Assump1}, \ref{AssumpId}, and \ref{Assump2}, if $u$ satisfies 
	\begin{equation}\label{M:0011}
	-\nabla \cdot a \nabla u = \nabla \cdot h\dans \Dom_{2^Kr}  \qquad \et u = 0 \sur \Gamma_{2^Kr},
	\end{equation}
	there holds
	\begin{equation}\label{M:0012}
	\begin{aligned}
	\langl \lt( \fint_{\Dom_r} |\nabla u|^2 \rt)^{\frac{p}{2}} \rangl^{\frac{1}{p}}
	\lesssim_{\Xi,p,p'} &
	\langl \lt( \fint_{\Dom_{r}} |h|^2 \rt)^{\frac{p'}{2}} \rangl^{\frac{1}{p'}}
	+
	\sum_{k=1}^{K-1} 2^{k(1-\rhobar_1)} 
	\langl \lt( \fint_{\Dom_{2^{k+1}r} \backslash \Dom_{2^{k}r}} |h|^2 \rt)^{\frac{p'}{2}} \rangl^{\frac{1}{p'}}
	\\
	&+
	2^{K(1-\rhobar_1)}
	\langl \lt( \fint_{\Dom_{2^Kr}} |\nabla u|^2 \rt)^{\frac{p'}{2}} \rangl^{\frac{1}{p'}} \quad \pourtout p \in [1, \infty) \et p'>p.
	\end{aligned}
	\end{equation}
\end{corollary}

Both Corollaries \ref{C:LSL} and \ref{Cor:Annealed} are obtained from quenched regularity results by similar proofs, which are based on a dyadic argument.
Hence, we only establish Corollary \ref{Cor:Annealed}, which is a genuine new result (compared to Corollary \ref{C:LSL}).

\begin{proof}[Proof of Corollary \ref{Cor:Annealed}]
	We extend $h$ by $0$ outside $\Dom_{2^Kr}$ and appeal to a dyadic argument.
	For this, we introduce $v = \sum_{k=0}^{K-1} w_k$, where, for all $k\in\{0,\dots,K-1\}$, $w_k$ satisfies $w_k=0$ on $\Gamma$ and
	\begin{align}
	\label{M:0020}
	&-\nabla \cdot a \nabla w_0 = \nabla \cdot \lt(\eta_{\Boule,r} h\rt) \dans \Dom,
	\\
	\label{M:0021}
	&-\nabla \cdot a \nabla w_k = \nabla \cdot \lt((\eta_{\Boule,2^{k+1}r} -\eta_{\Boule,2^{k}r}) h\rt) \dans \Dom, \qquad \pourtout k\in\{1,\dots,K-1\}.
	\end{align}
	Notice that there holds
	\begin{equation}\label{M:0022}
	-\nabla \cdot a \nabla (u-v) = 0 \dans \Dom_{2^Kr}  \qquad \et u-v = 0 \sur \Gamma_{2^Kr}.
	\end{equation}
	
	\paragraph{Step 1: Estimates on $w_k$.}
	By taking the $\LL^p_{\langl\cdot\rangl}$ moment of the energy estimate applied to \eqref{M:0020}, we get
	\begin{equation}\label{M:0023}
	\langl \lt( \fint_{\Dom_r} |\nabla w_0|^2 \rt)^{\frac{p}{2}} \rangl^{\frac{1}{p}}
	\lesssim
	\langl \lt( \fint_{\Dom_r} |h|^2 \rt)^{\frac{p}{2}} \rangl^{\frac{1}{p}}.
	\end{equation}
	Similarly, for $k\in\{1,\dots,K-1\}$, we get
	\begin{equation}\label{M:0024}
	\langl \lt( \int_{\Dom} |\nabla w_k|^2 \rt)^{\frac{p}{2}} \rangl^{\frac{1}{p}}
	\lesssim
	\langl \lt( \int_{\Dom_{2^{k+1}r} \backslash \Dom_{2^{k-1}r}} |h|^2 \rt)^{\frac{p}{2}} \rangl^{\frac{1}{p}}.
	\end{equation}
	
	Next, we may apply Theorem \ref{ThAL} to $w_k$, which is $a$-harmonic inside $\Dom_{2^{k-1}r}$.
	Notice that using $N=0$ in \eqref{Num:301_L2-1} yields
	\begin{align*}
	\lt(\fint_{\Dom_r} |\nabla w_k|^2\rt)^{\frac{1}{2}}
	&\lesssim
	C_* 2^{k(1-\rhobar_1)}
	\lt(\fint_{\Dom_{2^k r}} |\nabla w_k|^2\rt)^{\frac{1}{2}}.
	\end{align*}
	Hence, taking the $\LL^p_{\langl\cdot\rangl}$ moment, using the Hölder inequality, and then \eqref{M:0024} (replacing $p$ by $p'>p$) and \eqref{Borne_CN-}, we obtain
	\begin{align}\label{M:0026}
	\langl \lt(\fint_{\Dom_r} |\nabla w_k|^2\rt)^{\frac{p}{2}}\rangl^{\frac{1}{p}}
	\lesssim
	2^{k(1-\rhobar_1)}
	\langl \lt( \fint_{\Dom_{2^{k+1}r} \backslash \Dom_{2^{k-1}r}} |h|^2 \rt)^{\frac{p'}{2}} \rangl^{\frac{1}{p'}}.
	\end{align}
	
	\paragraph{Step 2: Estimate on $u$.}
	Since $u-v$ is $a$-harmonic inside $\Dom_{2^Kr}$ (\textit{cf.} \eqref{M:0022}), we may apply the same reasoning as in Step 1, getting from Theorem \ref{ThAL} that
	\begin{align*}
	\langl \lt(\fint_{\Dom_r} |\nabla (u-v)|^2\rt)^{\frac{p}{2}}\rangl^{\frac{1}{p}}
	\lesssim
	2^{K(1-\rhobar_1)} \langl \lt(\fint_{\Dom_{2^Kr}} |\nabla (u-v)|^2\rt)^{\frac{p'}{2}}\rangl^{\frac{1}{p'}}.
	\end{align*}
	Using the triangle inequality (and the energy estimate for $v$), the above estimate transforms into
	\begin{equation}\label{M:0025}
	\langl \lt(\fint_{\Dom_r} |\nabla (u-v)|^2\rt)^{\frac{p}{2}}\rangl^{\frac{1}{p}}
	\lesssim
	2^{K(1-\rhobar_1)} \langl \lt(\fint_{\Dom_{2^Kr}} |\nabla u|^2\rt)^{\frac{p'}{2}}\rangl^{\frac{1}{p'}}
	+
	2^{K(1-\rhobar_1)}
	\langl \lt(\fint_{\Dom_{2^Kr}} |h|^2\rt)^{\frac{p'}{2}}\rangl^{\frac{1}{p'}}.	
	\end{equation}
	Hence, using once more the triangle inequality, into which we insert \eqref{M:0023}, \eqref{M:0026}, and \eqref{M:0025}, we obtain
	\begin{align*}
	\langl \lt(\fint_{\Dom_r} |\nabla u|^2\rt)^{\frac{p}{2}}\rangl^{\frac{1}{p}}
	\lesssim~&
	\sum_{k=1}^{K-1} 2^{k(1-\rhobar_1)}
	\langl \lt( \fint_{\Dom_{2^{k+1}r} \backslash \Dom_{2^{k-1}r}} |h|^2 \rt)^{\frac{p'}{2}} \rangl^{\frac{1}{p'}}
	+
	\langl \lt( \fint_{\Dom_r} |h|^2 \rt)^{\frac{p'}{2}} \rangl^{\frac{1}{p'}}
	\\
	&+2^{K(1-\rhobar_1)} \langl \lt(\fint_{\Dom_{2^Kr}} |\nabla u|^2\rt)^{\frac{p'}{2}}\rangl^{\frac{1}{p'}}
	+
	2^{K(1-\rhobar_1)}
	\langl \lt(\fint_{\Dom_{2^Kr}} |h|^2\rt)^{\frac{p'}{2}}\rangl^{\frac{1}{p'}}.
	\end{align*}
	By observing that the fourth term on the \rhs may be absorbed by the first two, we retrieve \eqref{M:0012}.
\end{proof}

\subsection{Local annealed estimates on the corner correctors: Proof of Theorem \ref{Th:OptiGR}}

\begin{proof}[Proof of Theorem \ref{Th:OptiGR}]
As in the proof of Proposition~\ref{Prop:Quench}, the proof of \eqref{Opti_Decay} can be readily adapted to establish \eqref{Opti_Decay_Dir}.
Therefore, we only prove the former.
	
Let $x \in \Dom$.
By Step 2 of the proof of Theorem \ref{ThAL}, we already have that $\nabla \phitilde_n$ satisfies \eqref{Num:4011}, where $\Cstar$ satisfies \eqref{Borne_CN-}.

Localizing this estimate by means of Corollary \ref{C:LSL}, we establish
\begin{equation}\label{key1}
\langl
\lt( 
\fint_{\Dom(x)} \lt|\nabla \phitilde_n \rt|^2 
\rt)^{\frac{p}{2}}
\rangl^{\frac{1}{p}}
\lesssim_{\Xi,p,n} (|x|+1)^{\rhobar_n-1-\nu} \ln^{\tilde{\nu}'}(|x|+2).
\end{equation}
Then, integrating this estimate along a path and appealing to the Minkowski inequality, we get that
\begin{equation}\label{key2}
\langl
\lt| \fint_{\Dom(x)} \phitilde_n \rt|^p
\rangl^{\frac{1}{p}}
\lesssim_{\Xi,p,n} (|x|+1)^{\rhobar_n-\nu} \ln^{\tilde{\nu}'}(|x|+2).
\end{equation}
Last, recalling the decomposition \eqref{*ansatz} and invoking the maximum principle,
we finally obtain the desired \eqref{Opti_Decay}.
In our proof below, the random field $\Cstar(x)$ may change from line to line, but is always assumed to satisfy \eqref{Borne_CN}.

\paragraph{Step 1 : Argument for \eqref{key1}.}
W.~l.~o.~g., we may assume that $|x| \geq 8$.
Recall that $\phitilde_n$ satisfies \eqref{*Def:psitilde2}, the r.~h.~s.\ of which is controlled in a very strong norm (almost $\LL^\infty$).
We apply Corollary \ref{C:LSL} for $u \rightsquigarrow \phitilde_n$, $R \rightsquigarrow |x|/2$, and then recall \eqref{Num:4011} and \eqref{Borne:h}, obtaining \eqref{key2} as follows:
\begin{align*}
\langl \lt( \fint_{\Dom_{1}(x)} \lt| \nabla \phitilde_n  \rt|^2 
\rt)^{\frac{p}{2}} \rangl^{\frac{1}{p}}
\lesssim~~~&
\langl  \lt(\fint_{\Dom_{|x|/2}(x)} |\nabla \phitilde_n |^2\rt)^{\frac{p'}{2}} \rangl^{\frac{1}{p'}}+ \ln(|x|) \sup_{x' \in \Dom_{|x|/2}(x)} \langl\lt(\fint_{\Dom_1(x')}|h|^2\rt)^{\frac{p'}{2}} \rangl^{\frac{1}{p'}}
\\
\overset{\eqref{Num:4011},\eqref{Borne:h}}{\lesssim}& |x|^{\rhobar_n-1-\nu} \ln^{\tilde{\nu}'}(|x|).
\end{align*}

\paragraph{Step 2 : Argument for \eqref{key2}.}
First, by the Poincaré inequality (since $\phiC_n=0$ on the boundary $\Gamma$), for any $x \in \Dom$ such that $\dist(x,\Gamma) \leq 1$, we have that
\begin{equation*}
\langl\lt|\int_{\Dom(x)} \phiC_n\rt|^p\rangl^{\frac{1}{p}}
\lesssim \langl\lt(\int_{\Dom(x)} |\nabla \phiC_n|^2\rt)^{\frac{p}{2}}\rangl^{\frac{1}{p}}
\overset{\eqref{key1}}{\lesssim}_{\Xi,p,n} (|x|+1)^{\rhobar_n - 1 - \nu} \ln^{\tilde{\nu}'}(|x|+2).
\end{equation*}
Next, for any $x \in \Dom$ with $\dist(x,\Gamma) > 1$,
let $x'$ be such that $|x'| = |x|$ and $\dist(x',\Gamma)=1$,
and we denote by $\mathcal{A}$ the arc from $x$ to $x'$.
Hence, we have
\begin{align*}
\lt|\int_{\Boule(x)} \phiC - \int_{\Boule(x')} \phiC\rt|
=
\lt|\int_{\mathcal{A}} \lt(\int_{\Boule(z)} \nabla \phiC\rt) \cdot \dd z\rt|.
\end{align*}
Taking the $\LL^p_{\langl\cdot\rangl}$-moment for $p\geq 2$ and appealing to the Cauchy-Schwarz estimate and the Minkowski inequality yields
\begin{align*}
\langl\lt|\int_{\Boule(x)} \phiC - \int_{\Boule(x')} \phiC\rt|^p \rangl^{\frac{1}{p}}
\lesssim
\int_{\mathcal{A}} \langl
\lt(\int_{\Boule(z)} |\nabla \phiC|^2\rt)^{\frac{p}{2}}\rangl^{\frac{1}{p}} |\dd z|.
\end{align*}
Last, inserting \eqref{key1} into the above estimate, we easily get \eqref{key2}.

\paragraph{Step 3 : Conclusion.}
By the Poincaré-Wirtinger estimate, invoking \eqref{key1} and \eqref{key2}, we have
\begin{equation}\label{key7}
\langl
\lt( \fint_{\Dom(x)} |\phitilde_n|^2 \rt)^{\frac{p}{2}}
\rangl^{\frac{1}{p}}
\lesssim_{\Xi,p,n} (|x|+1)^{\rhobar_n-\nu} \ln^{\tilde{\nu}'}(|x|+2).
\end{equation}
Next, notice that the function $u(\cdot) :=\taubar_n(\cdot)+\phiC_{n}(\cdot) -\taubar_n(x)$
is $a$-harmonic in $\Dom(x)$.
Hence, we may apply the De Giorgi-Nash-Moser theorem \cite[Th.\ 8.25 p.\ 202]{GT} to $u$ in the form $|\phiC_{n}(x)|=|u(x)|\lesssim_{\lambda,\omega} (\fint_{\Dom(x)}|u|^2)^\frac12$.
Appealing to the triangle inequality on the decomposition \eqref{*ansatz},
into which we insert \eqref{taubar_Homog}, \eqref{Borne:HalfCorr}, and \eqref{key7}, we finally obtain \eqref{Opti_Decay}.
\end{proof}

\section{Error estimate for the nonstandard $2$-scale expansion: Argument for Theorem \ref{Th:Error}}

\label{error_est_sec}

\subsection{Construction of the corner flux-correctors: Proof of Lemma \ref{Lem:sig}}

\begin{proof}[Proof of Lemma \ref{Lem:sig}]
In this proof, the constant $\tilde{\nu}'\lesssim_{\tilde{\nu}} 1$ may change from line to line.
Let $i = 1,2$. This proof proceeds in three steps : 
In Step 1, we construct $\tisigmaD_i$; this relies on the extension of a given divergence-free field, that we establish in Step 2.
In Step 3, we establish
\begin{equation}\label{E:nablasigD-2}
\langl  \Big| \fint_{\Dom(x)} |\nabla \tisigmaD_i |^2 \Big|^{\frac{p}{2}} \rangl^{\frac{1}{p}}
\lesssim  (|x|+1)^{-\nu} \ln^{\tilde{\nu}'}(|x|+2) \qquad \pourtout p \in [1,\infty) \et x \in \Dom.
\end{equation}
Last, in Step 4, we prove \eqref{M:0005}.

\paragraph{Step 1 : Construction of $\tisigmaD_i$.}
We begin by constructing the skew-symmetric correction $\tisigmaD_i$ in \eqref{ansatz:sig}.
From \eqref{Def:sigD} we see that $\tisigmaD_i$ must satisfy
\begin{align}
\label{equation_sigma_tilde_1}
\nabla \cdot \tisigmaD_i = a (\nabla \phiD_i + e_i) - \abar e_i - \nabla \cdot \lt( (1 -\eta_{\Boule,1} ) \lt( \etaup \sigup_i + \etabulk \sigma_i + \etadown \sigdown_i\rt) \rt) \dans \Dom.
\end{align}
Using the ansatz for $\phiD_i$ given by \eqref{*ansatz} with $\bar{\tau}_n$ replaced by $x_i$, \textit{i.e.} 
\begin{align*}
\phiD_i=(1-\eta_{\Boule,1})\lt( \etaup \phiup_i + \etabulk \phi_i + \etadown \phidown_i \rt) + \phiDtilde_i,
\end{align*}
the equations \eqref{Def:sigs} satisfied by $\sigup_i$ and $\sigdown_i$, and, similarly to \eqref{*Def:h1h2}, the notation
\begin{equation}
\label{defn_hi}
\begin{aligned}
h_i := \eta_{\Boule,1} (a-\abar) e_i
&+\lt(a\phiup_i- \sigup_i\rt)
\nabla \lt( (1-\eta_{\Boule,1}) \etaup   \rt) 
+ \lt(a\phi_i - \sigma_i \rt) \nabla \lt( (1-\eta_{\Boule,1}) \etabulk \rt) 
\\
&+\lt(a\phidown_i- \sigdown_i\rt)
\nabla \lt( (1-\eta_{\Boule,1})\etadown \rt),
\end{aligned}
\end{equation} 
the \rhs of \eqref{equation_sigma_tilde_1} becomes 
\begin{align}
\label{rewrite_rhs}
\begin{split}
&a \nabla \phiDtilde_i + a \phiup_i \nabla( (1 -\eta_{\Boule,1} ) \etaup) + a \phi_i \nabla( (1 -\eta_{\Boule,1} ) \etabulk) + a \phidown_i \nabla( (1 -\eta_{\Boule,1} ) \etadown) 
\\
&+ (1 -\eta_{\Boule,1} )\lt( \etaup a \nabla \phiup_i + \etabulk a \nabla \phi_i + \etadown a \nabla \phidown_i \rt) + (a - \abar) e_i
\\ 
&- \sigup_i \nabla( (1 -\eta_{\Boule,1} ) \etaup) 
- \sigma_i \nabla( (1 -\eta_{\Boule,1} ) \etabulk)
- \sigdown_i \nabla( (1 -\eta_{\Boule,1} )\etadown)
\\
&-(1-\eta_{\Boule,1}) \lt(\etaup \lt(a\nabla\phiup_i + (a-\abar)e_i\rt)
+\etabulk \lt(a\nabla\phi_i + (a-\abar)e_i\rt)
+\etadown \lt(a\nabla\phidown_i + (a-\abar)e_i\rt)\rt)
\\
&= a \nabla \phiDtilde_i + h_i.
\end{split}
\end{align}
Combining \eqref{equation_sigma_tilde_1} and \eqref{rewrite_rhs}, we obtain that $\tisigmaD_i$ is required to solve 
\begin{align}
\label{equation_sigma_tilde}
\nabla \cdot \tisigmaD_i =a \nabla \phiDtilde_i + h_i \dans \Dom.
\end{align}
By adapting \eqref{*Def:psitilde2} to $\phiDtilde_i$, we immediately obtain that the \rhs of \eqref{equation_sigma_tilde} is divergence-free.
To construct a solution of \eqref{equation_sigma_tilde} we emulate previous constructions of adapted vector potentials, \textit{e.g.} in \cite{Fischer_Raithel_2017} or \cite{JosienRaithel_2019}, and for now take for granted the existence of an extension $g_i$ of $a\nabla \phiDtilde_i + h_i$ to $\R^2$ that is divergence-free and satisfies 
\begin{align}
\label{growth_g}
 \langl \lt( \fint_{B_1(x)} \lt| g_i \rt|^2 \rt)^{\frac{p}{2}} \rangl^{\frac{1}{p}}
\lesssim (|x|+1)^{-\nu} \ln^{\tilde{\nu}'}( |x| +2) \qquad \pourtout x \in \R^2 \et p \in [1, \infty).
\end{align}

Then, for $j=1,2$, we seek strictly subquadratic solutions to 
\begin{align}
\label{def_Nji}
-\Delta N_{ij}  = g_i \cdot e_j \dans \R^2
\end{align} 
and take the ansatz
\begin{align}
\label{sigma_correction_ansatz}
\tisigmaD_{ijk} = \partial_j N_{ik} - \partial_k N_{ij}.
\end{align}
 Taking the divergence of the ansatz for $\tisigmaD_{ijk}$, we see that 
\begin{align*}
\partial_k \tisigmaD_{ijk} 
= \partial_k\partial_j N_{ik}  - \partial_k\partial_k N_{ij} = g_i \cdot e_j \dans \R^2,
\end{align*}
where we have used \eqref{def_Nji}, and the Liouville principle for harmonic functions and that $g$ is divergence-free to deduce that $\partial_k N_{ik}$ is a constant.

To complete our construction of $\tisigmaD_{i}$ it remains to obtain strictly subquadratic solutions of the equations \eqref{def_Nji} and to show the existence of the extension $g$. To accomplish the former task we notice that because $g$ is divergence-free, \eqref{def_Nji} can be re-written in divergence-form as
\begin{align}
\label{def_Nji_2}
-\Delta N_{ij}  = \nabla \cdot \lt(x \cdot e_j g_i\rt) \dans \R^2.
\end{align} 
Using a covering argument and \eqref{growth_g} we see that the relation
\begin{align}
\label{growth_g_2}
\langl \lt( \fint_{\Boule_{r}} \lt| x \cdot e_j g_i \rt|^2 \rt)^{\frac{p}{2}} \rangl^{\frac{1}{p}}
\lesssim
\sup_{x' \in \Boule_r}
\langl \lt( \fint_{\Boule_{1}(x')} \lt| x \cdot e_j g_i \rt|^2 \rt)^{\frac{p}{2}} \rangl^{\frac{1}{p}}
\overset{\eqref{growth_g}}{\lesssim} r^{1-\nu} \ln^{\tilde{\nu}'}(r+1)
\end{align}
holds for $r\geq 1$ and $p\in [1, \infty)$.

By \cite[Lemma A.1]{JosienRaithel_2019}, there exists almost-surely a distributional solution $N_{ij} \in H^{1}_{\loc}(\R^2)$ of \eqref{def_Nji_2} satisfying 
\begin{align}
\label{subquadratic}
\limsup_{r \uparrow \infty} r^{-1} \lt( \fint_{\Boule_{r}} \lt|\nabla N_{ij} - \fint_{\Boule_r} \nabla N_{ij}  \rt|^2 \, \dd x \rt)^{\frac{1}{2}} =0.
\end{align}
Furthermore, it also follows from \cite[Lemma A.1]{JosienRaithel_2019}, using \eqref{growth_g_2}, that this solution satisfies 
\begin{align}
\label{estimate_Nji_1}
\langl \lt( \fint_{\Boule_r} \lt|\nabla N_{ij} - \fint_{\Boule_r} \nabla N_{ij}  \rt|^2 \, \dd x \rt)^{\frac{p}{2}} \rangl^{\frac{1}{p}} \lesssim_{\nu} r^{1-\nu} \ln^{\tilde{\nu}'}(r+1) \pourtout r\geq 1.
\end{align}

\paragraph{Step 2 : Divergence-free extension of $a\nabla \phiDtilde_i + h_i$.}
By Lemma \ref{Lem:DivFree} of the Appendix,
we may construct the extension $g_i$ of $a\nabla \phiDtilde_i + h_i$  to $\R^2$ that has been used in Step 1.

We then check that \eqref{growth_g} is satisfied. 
Let $x \in \Dom$ and $p \in [1,\infty)$.
By \eqref{defn_hi}, using Proposition \ref{Prop:growhthscor}, we see that the analogue of \eqref{Borne:h} is
\begin{align}
\label{analogue_1}
\langl \lt(\fint_{\Dom(x)} |h_i|^2\rt)^{\frac{p}{2}} \rangl^{\frac{1}{p}}
\lesssim
(|x|+1)^{-\nu} \ln^{\tilde{\nu}'}(|x|+2).
\end{align}
Furthermore, using the argument in Step 1 of the proof of Theorem \ref{Th:OptiGR} and \eqref{Num:4011-phiD}, we see that the analogue of \eqref{key1} is 
\begin{align}
\label{analogue_2}
\langl \lt(\fint_{\Dom(x)} |\nabla \phiDtilde_i |^2\rt)^{\frac{p}{2}} \rangl^{\frac{1}{p}}
\lesssim
(|x|+1)^{-\nu} \ln^{\tilde{\nu}'}(|x|+2).
\end{align}
Thus, obtaining \eqref{growth_g} is a simple matter of combining \eqref{analogue_1} and \eqref{analogue_2} with the definition of $\bar{H}$ in Lemma \ref{Lem:DivFree}. 

\paragraph{Step 3: Argument for \eqref{E:nablasigD-2}.}

W.~l.~o.~g., we assume that $|x|\geq 8$.
By the ansatz \eqref{sigma_correction_ansatz}, the Caccioppoli estimate, and then
\eqref{estimate_Nji_1} and \eqref{growth_g} combined with a covering argument in which we use the Minkowski inequality, we obtain that 
\begin{align*}
\langl \Bigg( \fint_{\Boule_{\frac{|x|}{4}}(x)} |\nabla \tisigmaD_i|^2 \Bigg)^{\frac{p}{2}} \rangl^{\frac{1}{p}} & \lesssim \langl \Bigg( \fint_{\Boule_{\frac{|x|}{4}}(x)} |\nabla^2 N_i |^2 \Bigg)^{\frac{p}{2}} \rangl^{\frac{1}{p}}\\
&  \lesssim \frac{1}{|x|}  \langl \Bigg( \fint_{\Boule_{\frac{|x|}{2}}(x)} |\nabla N_i - \fint_{\Boule_{\frac{|x|}{2}}(x)} \nabla N_i |^2 \Bigg)^{\frac{p}{2}} \rangl^{\frac{1}{p}}  +  \langl \Bigg( \fint_{\Boule_{\frac{|x|}{2}}(x)} |g_i |^2 \Bigg)^{\frac{p}{2}} \rangl^{\frac{1}{p}}\\
&\lesssim |x|^{-\nu} \ln^{\tilde{\nu}'}|x|.
\end{align*}
To finish the argument for \eqref{E:nablasigD-2} we resort to the following estimate:
\begin{align}
\label{MVP_sigma}
\begin{split}
\langl \Bigg( \fint_{\Boule(x)} |\nabla  \tisigmaD_i |^2 \Bigg)^{\frac{p}{2}} \rangl^{\frac{1}{p}} & \lesssim \langl \Bigg(\fint_{\Boule_{\frac{|x|}{4}}(x)} |\nabla  \tisigmaD_i |^2  \Bigg)^{\frac{p'}{2}} \rangl^{\frac{1}{p'}} +  |x|^{-\nu} \ln^{\tilde{\nu}'}|x| \qquad \pourtout p' > p,
\end{split}
\end{align}
which is proved below.
Indeed, the conjunction of the two above estimates gives \eqref{E:nablasigD-2}.

Here comes the proof of \eqref{MVP_sigma}.
By applying the Laplacian to \eqref{sigma_correction_ansatz} and using \eqref{def_Nji}, we obtain
\begin{align*}
	-\Delta \tisigmaD_{ijk} = - \partial_j \Delta N_{ik} + \partial_k \Delta N_{ij} = \partial_j (g_i e_k) - \partial_k (g_i e_j).
\end{align*}
Thus, we may apply Corollary \ref{C:LSL}, 
in the easier case where $a=\Id$ and where the sets $\Dom_r$ are replaced by balls $\Boule_r$ (furthermore, no boundary condition needs to be considered).
This yields \eqref{MVP_sigma} through
\begin{align}
\label{M:0040}
\langl \lt( \fint_{\Boule(x)} |\nabla  \tisigmaD_i |^2 \rt)^{\frac{p}{2}} \rangl^{\frac{1}{p}} 
& \lesssim \langl \lt(\fint_{\Boule_{\frac{|x|}{4}}(x)} |\nabla  \tisigmaD_i |^2  \rt)^{\frac{p'}{2}} \rangl^{\frac{1}{p'}}
+\ln\left( |x| \right)
\sup_{x' \in \Boule_{\frac{|x|}{2}}(x)}
\langl \lt( \fint_{\Boule(x')} |g_i|^2 \rt)^{\frac{p'}{2}} \rangl^{\frac{1}{p'}},
\end{align}
into which we insert \eqref{growth_g}.

\paragraph{Step 4: Argument for \eqref{M:0005}.}
Let $x \in \Dom$.
We integrate
\begin{align*}
\lt| \fint_{\Boule(x)} \tisigmaD_i -  \fint_{\Boule} \tisigmaD_i \rt|
\leq
|x|\lt| \int_{0}^{1} \lt(\fint_{\Boule(tx)} \nabla \tisigmaD_i\rt)  \dd t \rt|.
\end{align*}
Hence, by the Minkowski inequality and \eqref{E:nablasigD-2}, we get
\begin{align*}
\langl\lt| \fint_{\Boule(x)} \tisigmaD_i -  \fint_{\Boule} \tisigmaD_i \rt|^{p} \rangl^{\frac{1}{p}}
\leq &
|x|\int_{0}^{1} 
\langl \lt(\fint_{\Boule(tx)} |\nabla \tisigmaD_i|^2\rt)^{\frac{p}{2}} \rangl^{\frac{1}{p}} \dd t
\\
\lesssim&
(|x|+1)^{1-\nu} \ln^{\tilde{\nu}'}(|x|+2) \int_0^1 t^{-\nu} \dd t
\lesssim (|x|+1)^{1-\nu} \ln^{\tilde{\nu}'}(|x|+2).
\end{align*}
Using the Poincaré inequality as well as \eqref{E:nablasigD-2}, this implies
\begin{align*}
	\langl\lt( \fint_{\Boule(x)} \lt|\tisigmaD_i -  \fint_{\Boule} \tisigmaD_i\rt|^2 \rt)^{\frac{p}{2}} \rangl^{\frac{1}{p}}
	\lesssim (|x|+1)^{1-\nu} \ln^{\tilde{\nu}'}(|x| +2).
\end{align*}
Up to substracting a constant in the definition of $\tisigmaD_{i}$, we finally get \eqref{M:0005}
by appealing to the decomposition \eqref{ansatz:sig},
into which we insert \eqref{CorrSubDef} and \eqref{Borne:HalfCorr}.
\end{proof}

\subsection{Equation satisfied by the nonstandard 2-scale expansion \eqref{2-scale-New}: Proof of Lemma \ref{Lem:2scale}}

\begin{proof}[Proof of Lemma \ref{Lem:2scale}]
	For use later on, in order to avoid singularities at the tip of the sector (the small scales),
we show a generalization of \eqref{E:1-0}.
	In particular, as is seen in \eqref{2-scale-New-bis} below,
we resort to the use of additional singular functions $\taubar_n$ with a very sharp cut-off near the corner.
	
	\paragraph{Step 1 : Generalizing \eqref{E:1-0}.}
	We choose $M \in \NN$, $\chi:=\eta_{\Boule,1}$, and the cut-off function $\chi_\epsilon := \eta_{\Boule,\epsilon}$, and we generalize \eqref{2-scale-New} by the following 2-scale expansion :
	\begin{equation}\label{2-scale-New-bis}
	\utieps^{N,M} := (1+\phiD_{\epsilon,i}\partial_i) \ubar_{\reg}^{N,M} + \sum_{n=1}^N \bar{\gamma}_n (\taubar_n + \phiC_{\epsilon,n})\chi + \sum_{n=1}^N \bar{\gamma}_n \taubar_n \phiD_{\epsilon,i} \partial_i \chi
	+ \sum_{n=N+1}^{N+M} \bar{\gamma}_n \taubar_n  \chi_\epsilon,
	\end{equation}
	with
	\begin{equation}\label{Num:7003-bis}
	\ubar_{\reg}^{N,M} := \ubar - \sum_{n=1}^N \bar{\gamma}_n \taubar_n \chi - \sum_{n=N+1}^{N+M} \bar{\gamma}_n \taubar_n  \chi_\epsilon.
	\end{equation}
	We shall establish that
	\begin{equation}\label{E:1}
	-\nabla \cdot a_\epsilon (\nabla \utieps^{N,M}  - \nabla \ueps)
	=
	\nabla \cdot h^{N,M}_\epsilon \quad \dans \Dom, \qquad \text{with} \quad \utieps^{N,M} = \ueps = 0 \sur \Gamma,
	\end{equation}
	where
	\begin{equation}\label{E:1-bis}
	h^{N,M}_\epsilon :=\lt(\sigmaD_{\epsilon,i} - a \phiD_{\epsilon,i} \rt) \partial_i \nabla \bar{v}^{N,M}
	-\sum_{n=1}^N \bar{\gamma}_n a_\epsilon \nabla \lt((\chi-1)(\phiC_{\epsilon,n} - \phiD_{\epsilon,i} \partial_i \taubar_n)\rt)
	-\sum_{n=N+1}^{N+M} \bar{\gamma}_n \lt( a_\epsilon- \abar \rt) \nabla \lt(\chi_\epsilon \taubar_n \rt),
	\end{equation}
	for	
	\begin{equation}\label{Def:barv}
	\bar{v}^{N,M} :=\ubar - \sum_{n=1}^N \bar{\gamma}_n \taubar_n -\sum_{n=N+1}^{N+M} \bar{\gamma}_n \taubar_n  \chi_\epsilon  =\ubar_{\reg}^{N,M} + \sum_{n=1}^N \bar{\gamma}_n \taubar_n (\chi-1).
	\end{equation}
	Clearly, taking $M=0$ in the above identities establishes Lemma \ref{Lem:2scale}.

	\paragraph{Step 2 : Proof of \eqref{E:1}.}
	We compute
	\begin{equation}\label{Num:9001}
	\begin{aligned}
	-\nabla \cdot a_\epsilon (\nabla \utieps^{N,M}  - \nabla \ueps)
	\overset{\eqref{Num:7001}}{=}~& 
	-\nabla \cdot \lt( a_\epsilon \nabla \utieps^{N,M}  - \abar \nabla \ubar \rt)
	\\
	\overset{\eqref{2-scale-New-bis}, \eqref{Num:7003-bis}}{=}~&
	- \nabla \cdot \lt( a_\epsilon \nabla  \lt( (1+\phiD_{\epsilon,i}\partial_i) \ubar_{\reg}^{N,M} \rt)  - \abar \nabla \ubar_{\reg}^{N,M} \rt)
	\\
	&-\sum_{n=1}^N \bar{\gamma}_n  \nabla \cdot \lt( a_\epsilon \nabla \lt( (\taubar_n + \phiC_{\epsilon,n})\chi + \taubar_n \phiD_{\epsilon,i} \partial_i \chi \rt) - \abar \nabla \lt( \taubar_n \chi \rt) \rt)
	\\
	&- \sum_{n=N+1}^{N+M} \bar{\gamma}_n \nabla \cdot (a_\epsilon - \abar) \nabla (\chi_\epsilon \taubar_n).
	\end{aligned}
	\end{equation}
	But, since $\taubar_n$ and $\taubar_n + \phiC_{\epsilon,n}$ are $\abar$- and $a_\epsilon$-harmonic, respectively,
the summands in the second term of the \rhs read
	\begin{align*}
	&\bar{\gamma}_n\nabla \cdot \lt( a_\epsilon \nabla \lt( (\taubar_n + \phiC_{\epsilon,n})\chi + \taubar_n \phiD_{\epsilon,i} \partial_i \chi \rt) - \abar \nabla (\taubar_n \chi) \rt)
	\\
	&\qquad \qquad 
	=
	\bar{\gamma}_n \nabla \cdot \lt( a_\epsilon \nabla \lt( (\taubar_n + \phiC_{\epsilon,n})(\chi-1) + \taubar_n \phiD_{\epsilon,i} \partial_i (\chi-1) \rt) - \abar \nabla (\taubar_n (\chi-1)) \rt).
	\end{align*}
	Moreover
	\begin{align*}
	(\taubar_n + \phiC_{\epsilon,n})(\chi-1) + \taubar_n \phiD_{\epsilon,i} \partial_i (\chi-1)
	=(1+\phiD_{\epsilon,i}\partial_i)(\taubar_n (\chi-1))
	+(\phiC_{\epsilon,n} - \phiD_{\epsilon,i} \partial_i \taubar_n)(\chi-1).
	\end{align*}
	Thus, recalling \eqref{Def:barv} we may reformulate \eqref{Num:9001} as 
	\begin{equation*}
	\begin{aligned}
	-\nabla \cdot a_\epsilon (\nabla \utieps^{N,M}  - \nabla \ueps)
	=~&
	- \nabla \cdot \lt( a_\epsilon \nabla  \lt( (1+\phiD_{\epsilon,i}\partial_i) \bar{v}^{N,M} \rt)  - \abar \nabla \bar{v}^{N,M} \rt)
	\\
	&-\sum_{n=1}^N \bar{\gamma}_n \nabla \cdot a_\epsilon \nabla \lt( (\phiC_{\epsilon,n} - \phiD_{\epsilon,i} \partial_i \taubar_n)(\chi-1) \rt)
	\\
	&- \sum_{n=N+1}^{N+M} \bar{\gamma}_n \nabla \cdot (a_\epsilon - \abar) \nabla (\chi_\epsilon \taubar_n).
	\end{aligned}
	\end{equation*}
	By \eqref{Num:4003} (replacing $\phiup \rightsquigarrow \phiD$ and $\sigup \rightsquigarrow\sigmaD$), the first term on the \rhs can be expressed as
	\begin{align*}
	-\nabla \cdot \lt( a_\epsilon \nabla \lt( (1+\phiD_{\epsilon,i}\partial_i) \bar{v}^{N,M} \rt) - \abar \nabla \bar{v}^{N,M} \rt)
	=-\nabla \cdot \lt(a_\epsilon \phiD_{\epsilon,i} - \sigmaD_{\epsilon,i} \rt) \partial_i \nabla \bar{v}^{N,M},
	\end{align*}
	so that we get \eqref{E:1} along with \eqref{E:1-bis}.
\end{proof}

\subsection{Preliminary estimates}

\begin{lemma}\label{Lem:Prelim}
	Under the assumptions of Theorem \ref{Th:Error}, pick $M \in \mathbb{N}$ sufficiently large so that
	\begin{equation}\label{Def:M}
	\rhobar_{M+1} \geq 2.
	\end{equation}	
	Then, there exists an exponent $\tilde{\nu}'\lesssim_{\tilde{\nu}} 1$ such that $h^{N,M}_\epsilon$ defined by \eqref{E:1-bis} satisfies
	\begin{align}\label{M:0003}
	\langl\lt( \fint_{\Dom_\epsilon(x)} |h^{N,M}_\epsilon|^2 \rt)^{\frac{p}{2}}\rangl^{\frac{1}{p}}
	\lesssim_{\Xi, N, M, p}
	\lt\{
	\begin{aligned}
	&\epsilon^\nu \ln^{\tilde{\nu}'}(\epsilon^{-1}|x|) |x|^{\rhobar_{N}-1-\nu} && \si x \in \Dom \backslash \Dom_{1},
	\\
	&\epsilon^\nu \ln^{\tilde{\nu}'}(\epsilon^{-1})  |x|^{\rhobar_{N+1}-1-\nu} && \si x \in \Dom_{1} \backslash \Dom_{4\epsilon},
	\\
	&\epsilon^{\rhobar_{N+1}-1}  && \si x \in \Dom_{4\epsilon},
	\end{aligned}
	\rt.
	\end{align}
	for $p \in [1,\infty)$.
\end{lemma}

\begin{proof}[Proof of Lemma \ref{Lem:Prelim}.]
	In this proof, the constant $\tilde{\nu}'\lesssim_{\tilde{\nu}} 1$ may change from line to line.
	
	In order to show \eqref{M:0003}, we split
	\begin{align*}
	h^{N, M}_\epsilon &~= 
	\underbrace{\lt(\sigmaD_{\epsilon,i} - a_\epsilon \phiD_{\epsilon,i} \rt) \partial_i \nabla \bar{v}^{N,M}}_{h^{N, M}_{\epsilon,1}}
	-\underbrace{\sum_{n=1}^N \bar{\gamma}_n a_\epsilon \nabla \lt((\chi-1)(\phiC_{\epsilon,n} - \phiD_{\epsilon,i} \partial_i \taubar_n)\rt)}_{h^{N, M}_{\epsilon,2}}
	-\underbrace{\sum_{n=N+1}^{N+M} \bar{\gamma}_n \lt( a_\epsilon- \abar \rt) \nabla \lt(\chi_\epsilon \taubar_n \rt)}_{h^{N, M}_{\epsilon,3}},
	\end{align*}
	and we estimate each term separately.
	Notice that $|\bar{\gamma}_n| \lesssim 1$ by \eqref{Formula :gamma3}, after applying the Cauchy-Schwarz and Poincaré inequalities, followed by the energy estimate for \eqref{Num:7001}.
	
	\paragraph{Step 1: Estimates on $\nabla^2 \bar{v}$}
	First, we establish that:
	\begin{align}
	|\nabla^2 \bar{v}^{N,M}(x)|
	 \lesssim
	\lt\{
	\begin{aligned}
	& |x|^{\rhobar_{N}-2},
	&& \pourtout x \in \Dom \backslash \Dom_1,
	\\
	&|x|^{\rhobar_{N+1}-2},
	&& \pourtout x \in \Dom_1 \backslash \Dom_\epsilon,
	\\
	& \epsilon^{\rhobar_{N+1}-2},
	&& \pourtout x \in \Dom_\epsilon.
	\end{aligned}
	\rt.
	\label{Num:7013b}
	\end{align}
	Notice that these three estimates coincide near the boundaries of their domains of validity.
	
	By the energy estimate for \eqref{Num:7001}, there holds
	\begin{equation}\label{Num:10002}
	\lt(\int_{\Dom}|\nabla \bar{u}|^2\rt)^{\frac{1}{2}} 
	\lesssim 1.
	\end{equation}
	We pick $r \geq 1$.
	Since $\ubar$ is harmonic in $\Dom \backslash \Dom_1$,
a duality argument in the spirit of Step 3 of the proof of Lemma \ref{Lem:Iter2} yields
	\begin{equation}\label{Num:10001}
	\lt(\int_{\Dom_{2r} \backslash \Dom_r}|\nabla \bar{u}|^2\rt)^{\frac{1}{2}} 
	\lesssim r^{-\rhobar_1}.
	\end{equation}
	Let $x \in \Dom \backslash \Dom_{1/4}$.
	The above estimates \eqref{Num:10002} and \eqref{Num:10001} combined with classical Schauder regularity theory applied
to the function $\ubar$ in $\Dom_{|x|/2}(x)$ implies that
	\begin{align}\label{Num:10003}
	|x|^{-1}\lt|\nabla \ubar_{\reg}^{N}(x)\rt| + \lt|\nabla^2 \ubar_{\reg}^{N}(x)\rt| \lesssim |x|^{-\rhobar_{1}-2}
	\end{align}
(note that $\ubar_{\reg}^N=\ubar$ in $\Dom\backslash \Dom_2$, \textit{cf}.~\eqref{Num:7003}).
	Recalling \eqref{Def:barv}, in which the third term on the \rhs vanishes in $\Dom \backslash \Dom_{\epsilon}$, this yields the first inequality of \eqref{Num:7013b} in form of
	\begin{align}\label{M:0004}
	|\nabla^2 \bar{v}^{N,M}(x)|
	\lesssim~& |x|^{-\rhobar_1-2} + \sum_{n=1}^{N} |\bar{\gamma}_n| |\nabla^2\taubar_n(x)|
	\overset{\eqref{taubar_Homog}}{\lesssim}
	|x|^{-\rhobar_1-2} + \sum_{n=1}^{N} |x|^{\rhobar_n -2}
	\lesssim |x|^{\rhobar_{{N}}-2}.
	\end{align}
	Next, we apply Theorem \ref{Th:homog} to $\ubar_{\reg}^N$, which is harmonic $\Dom_{1/2}$.
	Thus, for $x \in \Dom_{1/4}$, we obtain :
 	\begin{align*}
	\lt(\fint_{\Dom_{2|x|}} \lt|\ubar_{\reg}^{N}\rt|^2 \rt)^\frac{1}{2}\lesssim |x|^{\rhobar_{{N}+1}} \lt( \fint_{\Dom_1} |\ubar|^2 \rt)^{\frac{1}{2}}
	\lesssim |x|^{\rhobar_{{N}+1}}.
	\end{align*}
	Hence, the classical Schauder regularity theory applied to $\ubar_{\reg}^N$
in $\Dom_{|x|/2}(x)$ yields $\lt|\nabla^2 \ubar_{\reg}^{N}(x)\rt| \lesssim |x|^{\rhobar_{{N}+1}-2}$.
	Resorting to the triangle inequality as in \eqref{M:0004} combined with $\ubar_{\reg}^{N}=\ubar_{\reg}^{N,M}$ on $D\setminus D_\epsilon$,
this establishes the second estimate of \eqref{Num:7013b}.
	
	Last, inside $\Dom_{\epsilon}$, we have
	\begin{align*}
	\bar{v}^{{N},M} = \ubar_\reg^{{N}+M} + \sum_{n={N}+1}^M \bar{\gamma}_n (\chi_\epsilon - 1) \taubar_n.
	\end{align*}
	Hence, applying the same reasoning as above to $\ubar_\reg^{{N}+M}$, we get, for $x \in \Dom_\epsilon$, that $|\nabla^2 \ubar_\reg^{N+M}(x)| \lesssim |x|^{\rhobar_{{N}+M+1}-2}$, to the effect of
	\begin{align*}
	\lt|\nabla^2 \bar{v}^{{N},M}(x)\rt| 
	\lesssim~& |x|^{\rhobar_{{N}+M+1}-2} + \sum_{n={N}+1}^M |\bar{\gamma}_n| \lt| \nabla^2\lt( (\chi_\epsilon - 1) \taubar_n\rt)(x)\rt|
	\\
	\overset{\eqref{taubar_Homog}}{\lesssim}~& |x|^{\rhobar_{{N}+M+1}-2} + \sum_{n={N}+1}^M \epsilon^{-2+\rhobar_n} \mathds{1}(|x| > \epsilon / 2)
	\overset{\eqref{Def:M}}{\lesssim} \epsilon^{\rhobar_{{N}+1}-2}.
	\end{align*}
	This finally proves the last inequality in \eqref{Num:7013b}.
	
	\paragraph{Step 2: Estimate on $h^{N, M}_{\epsilon,1}$}
		Let $x \in \Dom$.
		By Theorem \ref{Th:OptiGR} and Lemma \ref{Lem:sig}, we get
		\begin{align*}
		\langl \lt( \fint_{\Dom_\epsilon(x)} |h^{N, M}_{\epsilon,1}|^2 \rt)^{\frac{p}{2}} \rangl^{\frac{1}{p}}
		&\lesssim
		\langl \lt( \fint_{\Dom_\epsilon(x)} \lt|\lt( \sigmaD_{\epsilon}, \phiD_{\epsilon} \rt)\rt|^2 \rt)^{\frac{p}{2}} \rangl^{\frac{1}{p}}
		\lt\|\nabla^2 \bar{v}^{N,M} \rt\|_{\LL^\infty(\Dom_\epsilon(x))}
		\\
		&\overset{\eqref{Opti_Decay_Dir}, \eqref{M:0005}}{\lesssim}
		\epsilon^{\nu}(\epsilon + |x|)^{1-\nu} \ln^{\tilde{\nu}'}\lt(\epsilon^{-1}|x| + 2\rt) \lt\|\nabla^2 \bar{v}^{N,M} \rt\|_{\LL^\infty(\Dom_\epsilon(x))}.
		\end{align*}
		Inserting \eqref{Num:7013b} entails that $h^{N, M}_{\epsilon,1}$ satisfies \eqref{M:0003}.
	
	\paragraph{Step 3: Estimate on $h^{N, M}_{\epsilon,2}$}
Since $\chi-1$ vanishes inside $\Dom_{1/2}$, we already have that $h^{N, M}_{\epsilon,2}$ satisfies third estimate in \eqref{M:0003}, and the second estimate in \eqref{M:0003} for $x \in \Dom_{1/2}$.
Hence, it remains to establish that the first estimate in \eqref{M:0003} is valid for  $x \in \Dom \backslash \Dom_{\frac{1}{2}}$.
		By \eqref{*ansatz}, using $\phiDtilde_i$, as defined in Proposition \ref{Prop:Quench}, and \eqref{taubar_Homog}, we have
		\begin{align*}
		|\nabla \lt(\phiC_{n} - \phiD_{i} \partial_i \taubar_n\rt)(x)|
		&\leq
		|\nabla\phitilde_n(x)| +
		\lt|\nabla\lt(\lt( \etaup \phiup_i  + \etabulk \phi_i + \etadown \phidown_i - \phiD_{i}  \rt) \partial_i \taubar_n\rt)(x)\rt|
		\\
		&\lesssim
		|\nabla\phitilde_n(x)|
		+
		|\nabla \phiDtilde(x)| |x|^{\rhobar_n - 1}
		+
		|\phiDtilde(x)| |x|^{\rhobar_n - 2},
		\end{align*}
		By scaling and the triangle inequality, appealing to \eqref{key1} and \eqref{analogue_2}, we get
		\begin{equation}\label{M:0006}
		\langl \lt( \fint_{\Dom_\epsilon(x)} |\nabla \lt(\phiC_{\epsilon,n} - \phiD_{\epsilon,i} \partial_i \taubar_n\rt)|^2 \rt)^{\frac{p}{2}} \rangl^{\frac{1}{p}}
		\lesssim \epsilon^\nu \ln^{\tilde{\nu}'}(\epsilon^{-1}|x|) |x|^{\rhobar_n - 1 - \nu}.
		\end{equation}
		As a consequence, this yields that $h^{N, M}_{\epsilon,2}$ satisfies the first estimate in \eqref{M:0003}.

	\paragraph{Step 4: Estimate on $h^{N, M}_{\epsilon,3}$}
		Since $\chi_\epsilon$ is supported in $\Boule_\epsilon$, $h^{N, M}_{\epsilon,3}$ satisfies the first two estimates of \eqref{M:0003}.
		Last, we use the triangle inequality to get for $x \in \Dom_\epsilon$
		\begin{align*}
			|h^{N, M}_{\epsilon,3}(x)| 
			&\lesssim \sum_{n=N+1}^{N+M} \lt(| \taubar_n(x)\nabla \chi_\epsilon(x)| +  |\chi_\epsilon(x) \nabla \taubar_n(x)| \rt)
			\overset{\eqref{taubar_Homog}}{\lesssim}
			\epsilon^{-1} |x|^{\rhobar_{{N}+1}} \mathds{1}(|x|> \epsilon/2)  + |x|^{\rhobar_{{N}+1}-1}
			\lesssim |x|^{\rhobar_{{N}+1}-1}.
		\end{align*}
		Since $\rhobar_{{N}+1}-1 \geq \rhobar_1 -1 > -1$, we have $h^{N, M}_{\epsilon,3} \in \LL^2(\Dom_{4\epsilon})$ and
		\begin{align*}
			\lt(\fint_{\Dom_{4\epsilon}}\lt|h^{N, M}_{\epsilon,3}\rt|^2 \rt)^{\frac{1}{2}}
			\lesssim
			\epsilon^{-1} \lt( \int_0^\epsilon t^{2 \rhobar_{N+1} - 1} \dd t \rt)^{\frac{1}{2}} \lesssim \epsilon^{\rhobar_{N+1}-1}.
		\end{align*}
		In other words, $h^{N, M}_{\epsilon,3}$ satisfies the third estimate in \eqref{M:0003}.
\end{proof}

\subsection{Quasi-optimal error estimate: Proof of Theorem \ref{Th:Error}}

\begin{proof}[Proof of Theorem \ref{Th:Error}]
	In this proof, the constant $\tilde{\nu}'\lesssim_{\tilde{\nu}} 1$ may change from line to line.
	The core of the proof relies on \eqref{E:1}.

	In Step 1 we establish 
	\begin{equation}\label{Num:7004}
		\langl\lt( \int_\Dom |\nabla \utieps^{N,M} - \nabla \ueps|^2 \rt)^{\frac{p}{2}}\rangl^{\frac{1}{p}}
		\lesssim_{\Xi,q,p, N, M}
		\epsilon^{\nu} \ln^{\tilde{\nu}'} (\epsilon^{-1}),
	\end{equation}
	where $\utieps^{N,M}$ is defined by \eqref{2-scale-New-bis}, for $N=N_0$ and $M \in \mathbb{N}$ satisfying \eqref{Def:M}.
	In Step 2, we generalize \eqref{Num:7004} for general $N>N_0$.
	In Step 3 we show that, for any $\epsilon < r \leq 2$, there holds
	\begin{equation}\label{M:0013}
	\langl \biggl(\fint_{\Dom_{r}}| \nabla \utieps^{N,M}  - \nabla \ueps |^{2}\biggr)^\frac{p}2 \rangl^{\frac{1}{p}}
	\lesssim_{\Xi, q, p, N, M}
	\epsilon^\nu \ln^{\tilde{\nu}'}(\epsilon^{-1})r^{\rhobar_1-1}.
	\end{equation}
	In Step 4, we localize \eqref{M:0013} near all points $x \in \Dom$
to the effect of \eqref{Num:7034}.

	\paragraph{Step 1: Proof of \eqref{Num:7004} for $N=N_0$.}
	By definition, we have  $\nabla \ueps \in \LL^2(\Dom)$.
	Moreover, we note that $\nabla \utieps^{N_0,M}\in\LL^2(\Dom)$,
\textit{cf.} \eqref{2-scale-New-bis}.
	Indeed, Theorem~\ref{Th:homog} and \eqref{Num:10003}, and $\rhobar_{N+M+1} \geq 2$ (by \eqref{Def:M}), imply that the gradient of the first term on the \rhs of \eqref{2-scale-New-bis} is in $\LL^2(\Dom)$.
Moreover, the remaining terms are in $\HH^1_{\rm{loc}}(\Dom)$ with compact support, so that $\nabla \utieps^{N_0,M}\in\LL^2(\Dom)$ follows.
	Hence, by the energy estimate applied to \eqref{E:1}, we have
	\begin{equation*}
		\langl\lt( \int_\Dom |\nabla \utieps^{N_0,M} - \nabla \ueps|^2 \rt)^{\frac{p}{2}}\rangl^{\frac{1}{p}}
		\lesssim
		\langl\lt( \int_\Dom |h^{N_0,M}_\epsilon|^2 \rt)^{\frac{p}{2}}\rangl^{\frac{1}{p}}.
	\end{equation*}
	Thus, \eqref{Num:7004} comes as a consequence of
	\begin{equation}\label{Num:7024}
	\langl\lt( \int_\Dom |h^{N_0,M}_\epsilon|^2 \rt)^{\frac{p}{2}}\rangl^{\frac{1}{p}}
	\lesssim \epsilon^{\nu} \ln^{\tilde{\nu}'} (\epsilon^{-1}).
	\end{equation}
	
	Here comes the argument for \eqref{Num:7024}.	
	W.~l.~o.~g. we may assume $p \geq 2$, so that we get from a covering argument and the Minkowski inequality	
	\begin{align}
	\label{last_label}
	\langl\lt( \int_\Dom |h^{N_0,M}_\epsilon|^2 \rt)^{\frac{p}{2}}\rangl^{\frac{1}{p}}
	\lesssim
	\lt(\int_\Dom \langl\lt( \fint_{\Dom_\epsilon(x)} |h^{N_0,M}_\epsilon|^2 \rt)^{\frac{p}{2}}\rangl^{\frac{2}{p}} \dd x\rt)^{\frac{1}{2}}.
	\end{align}
	Then, we decompose the outer integral onto three domains in order to employ \eqref{M:0003} :
	\begin{align*}
	\int_\Dom \langl\lt( \fint_{\Dom_\epsilon(x)} |h^{N_0,M}_\epsilon|^2 \rt)^{\frac{p}{2}}\rangl^{\frac{2}{p}} \dd x
	&=\lt( \int_{\Dom \backslash \Dom_1} + \int_{\Dom_1 \backslash \Dom_{4\epsilon}} + \int_{\Dom_{4\epsilon}}\rt) \lt( \langl\lt( \fint_{\Dom_\epsilon(x)} |h^{N_0,M}_\epsilon|^2 \rt)^{\frac{p}{2}}\rangl^{\frac{2}{p}} \rt) \dd x
	\\
	&\overset{\eqref{M:0003}}{\lesssim}
	\epsilon^{2\nu}
	\int_1^\infty \ln^{2\tilde{\nu}'}(\epsilon^{-1}t) t^{2 \rhobar_{N_0} - 2\nu -1} \dd t
	\\
	&\qquad + \epsilon^{2\nu} \ln^{2\tilde{\nu}'}(\epsilon^{-1})
	\int_\epsilon^1 t^{2\rhobar_{{N_0}+1}-2\nu -1} \dd t
	\\
	&\qquad + \int_0^\epsilon \epsilon^{2\rhobar_{{N_0}+1}-2} t\dd t.
	\end{align*}
	Yet, recall that $N_0$ satisfies \eqref{Num:7029}, so that $2 \rhobar_{N_0} - 2\nu -1 < -1$ and $2\rhobar_{{N_0}+1}-2\nu -1 \geq -1$.
	Hence
	\begin{align*}
	\int_\Dom \langl\lt( \fint_{\Dom_\epsilon(x)} |h^{N_0,M}_\epsilon|^2 \rt)^{\frac{p}{2}}\rangl^{\frac{2}{p}} \dd x
	&\lesssim
	\epsilon^{2\nu} \ln^{2\tilde{\nu}'}(\epsilon^{-1})
	+ \epsilon^{2\nu} \ln^{2\tilde{\nu}'+1}(\epsilon^{-1})
	+ \epsilon^{2\rhobar_{{N_0}+1}}
	\lesssim \epsilon^{2\nu} \ln^{2\tilde{\nu}'+1}(\epsilon^{-1}).
	\end{align*}
	Up to redefining $\tilde{\nu}'$, this proves \eqref{Num:7024}.
	
	\paragraph{Step 2 : Proof of \eqref{Num:7004} for general $N \geq N_0$.}
	
	For $M'=M+N-N_0 \geq M$, we estimate the difference
	\begin{align*}
		\utieps^{N,M} - \utieps^{N_0,M'} \overset{\eqref{2-scale-New-bis}}{=} \sum_{n=N_0+1}^N \bar{\gamma}_n \lt( \chi \lt(\phiC_{\epsilon,n} - \phiD_{\epsilon,i} \partial_i \taubar_n \rt)  + \phiD_{\epsilon,i}\partial_i \lt( \taubar_n\chi_\epsilon\rt)\rt).
	\end{align*}
	Towards this end, we decompose
	\begin{align*}
	\nabla \lt(\utieps^{N,M} - \utieps^{N_0,M'}\rt)
	=~&
	\sum_{n=N_0+1}^N \bar{\gamma}_n \lt( \lt(\phiC_{\epsilon,n} - \phiD_{\epsilon,i} \partial_i \taubar_n \rt) \nabla \chi
	+(\chi-\chi_\epsilon) \nabla \lt(\phiC_{\epsilon,n} - \phiD_{\epsilon,i} \partial_i \taubar_n \rt)
	\rt)
	\\
	&+
	\sum_{n=N_0+1}^N \bar{\gamma}_n
	\lt( 
	\chi_\epsilon \nabla \phiC_{\epsilon,n}
	+
	\phiD_{\epsilon,i}\partial_i \taubar_n \nabla \chi_\epsilon + \nabla \lt( \taubar_n \phiD_{\epsilon,i} \partial_i \chi_\epsilon\rt) 
	\rt).
	\end{align*}
	Thus, by the triangle inequality followed by the Poincaré inequality, we get
	\begin{align*}
		\lt(
		\int_{\Dom} \lt| \nabla \lt(\utieps^{N,M} - \utieps^{N_0,M'}\rt)\rt|^2
		\rt)^{\frac{1}{2}}
		\lesssim~&
		\sum_{n=N_0+1}^N
		\lt(
		\int_{\Dom_1 \backslash \Dom_{\epsilon/2}} \lt|\nabla \lt(\phiC_{\epsilon,n} - \phiD_{\epsilon,i} \partial_i \taubar_n \rt)\rt|^2
		\rt)^{\frac{1}{2}}
		\\
		&+\sum_{n=N_0+1}^N
		\lt( 
		\int_{\Dom_\epsilon} \lt|\lt(\epsilon^{-1}\phiD_{\epsilon}\nabla \taubar_n,\epsilon^{-1} \taubar_n\nabla\phiD_{\epsilon}, \epsilon^{-2} \phiD_{\epsilon}\taubar_n, \nabla \phiC_{\epsilon,n} \rt)\rt|^2
		\rt)^{\frac{1}{2}}.
	\end{align*}
	We estimate the summands of the first \rhs above by means of a covering argument combined with the Minkowski estimate (as in \eqref{last_label}), along with \eqref{M:0006}, to the effect of 	\begin{align*}
	\langl
	\lt(
	\int_{\Dom_1 \backslash \Dom_{\epsilon/2}} \lt|\nabla \lt(\phiC_{\epsilon,n} - \phiD_{\epsilon,i} \partial_i \taubar_n \rt)\rt|^2
	\rt)^{\frac{p}{2}}\rangl^{\frac{1}{p}}
	&\lesssim
	\lt(
	\int_{\Dom_1 \backslash \Dom_{\epsilon/2}} 
	\langl\lt(\fint_{\Dom_{\epsilon/4}}\lt|\nabla \lt(\phiC_{\epsilon,n} - \phiD_{\epsilon,i} \partial_i \taubar_n \rt)\rt|^2 \rt)^{\frac{p}{2}}\rangl^{\frac{1}{p}}
	\rt)^{\frac{1}{2}}
	\\
	&\overset{\eqref{M:0006}}{\lesssim} \epsilon^\nu \ln^{\tilde{\nu}'}(\epsilon^{-1})
	\lt(\int_{\epsilon}^1  t^{2\rhobar_n -  2\nu - 1} \dd t\rt)^{\frac{1}{2}}.
	\end{align*}
	For $n \geq N_0+1$, by definition \eqref{Num:7029}, the integral on the \rhs above is bounded by $\ln(\epsilon^{-1})$.
	Hence, also using \eqref{Opti_Decay} and \eqref{Opti_Decay_Dir}, we obtain
	\begin{align*}
	\langl \lt(
	\int_{\Dom} \lt| \nabla \lt(\utieps^{N,M} - \utieps^{N_0,M'}\rt)\rt|^2
	\rt)^{\frac{p}{2}}\rangl^{\frac{1}{p}}
	\lesssim~& \epsilon^{\nu} \ln^{\tilde{\nu}'}(\epsilon^{-1}) + \epsilon^{\rhobar_{{N_0}+1}},
	\end{align*}
	which, by \eqref{Num:7029}, yields \eqref{Num:7004}.

	\paragraph{Step 3: Averaged estimate \eqref{M:0013}.}
		W.~l.~o.~g.\ we may assume that $r=2^{-K}$ for $K \in \NN \backslash\{0\}$.
		By Corollary \ref{Cor:Annealed} applied to $\utieps^{N,M}  -\ueps$, which satisfies \eqref{E:1}, followed by \eqref{M:0003} and \eqref{Num:7004}, we have for $p'>p$ :
		\begin{align*}
		\langl \lt( \fint_{\Dom_r} |\nabla \utieps^{N,M}  - \nabla \ueps|^2 \rt)^{\frac{p}{2}} \rangl^{\frac{1}{p}}
		\lesssim~~~~&
		\langl \lt( \fint_{\Dom_{r}} |h^{N,M}_\epsilon|^2 \rt)^{\frac{p'}{2}} \rangl^{\frac{1}{p'}}
		+
		\sum_{k=1}^{K-1} 2^{k(1-\rhobar_1)} 
		\langl \lt( \fint_{\Dom_{2^{(k+1)}r} \backslash \Dom_{2^{k}r}} |h^{N,M}_\epsilon|^2 \rt)^{\frac{p'}{2}} \rangl^{\frac{1}{p'}}
		\\
		&+
		r^{\rhobar_1-1}
		\langl \lt( \fint_{\Dom_{1}} |\nabla \utieps^{N,M}  - \nabla \ueps|^2 \rt)^{\frac{p'}{2}} \rangl^{\frac{1}{p'}}
		\\
		\overset{\eqref{M:0003}, \eqref{Num:7004}}{\lesssim}&
		\epsilon^\nu \ln^{\tilde{\nu}'}(\epsilon^{-1}) 
		\lt(
		r^{\rhobar_1-1}
		+
		\sum_{k=0}^{K-1} 2^{k(1-\rhobar_1)} 
		2^{(k+1)(\rhobar_{N+1}-1-\nu)}
		r^{\rhobar_{N+1}-1-\nu}\rt)
		\\
		\lesssim~~~~&
		\epsilon^\nu \ln^{\tilde{\nu}'}(\epsilon^{-1}) 
		\lt(
		1+\sum_{k=0}^{K-1} (2^{(k+1)}r)^{\rhobar_{N}-\nu}
		\rt)r^{\rhobar_1-1}.
		\end{align*}
		Since we assume that $N > N_0$ and that $N_0$ satisfies \eqref{Num:7029}, and $\rhobar_N-\nu>0$, and since $r = 2^{-K}$,  then we get that the above geometrical sum is bounded, to the effect of \eqref{M:0013}.

	\paragraph{Step 4: Local estimate \eqref{Num:7034}.}
		We now consider $x \in \Dom_1 \backslash \Dom_{2\epsilon}$.
		Applying (a rescaled version of) Corollary \ref{C:LSL} to $\utieps^{N,M}  -\ueps$, which satisfies \eqref{E:1}, we have, for $p'>p$,
		\begin{align*}
		\langl \biggl(\fint_{\Dom_\epsilon(x)}| \nabla \utieps^{N,M}  - \nabla \ueps |^{2}\biggr)^\frac{p}2 \rangl^{\frac{1}{p}}
		\overset{\eqref{M:0002}}{\lesssim}~&
		\langl  \lt(\fint_{\Dom_{|x|/2}(x)} | \nabla \utieps^{N,M}  - \nabla \ueps|^2\rt)^{\frac{p'}{2}} \rangl^{\frac{1}{p'}}
		\\
		&+ \ln(\epsilon^{-1}|x|) \sup_{x' \in \Dom_{|x|/2}(x)} \langl\lt(\fint_{\Dom_\epsilon(x')}\lt|h^{N,M}_\epsilon\rt|^2\rt)^{\frac{p'}{2}} \rangl^{\frac{1}{p'}}
		\\
		\overset{\eqref{M:0013},\eqref{M:0003}}{\lesssim}&
		\epsilon^\nu \ln^{\tilde{\nu}'}(\epsilon^{-1})|x|^{\rhobar_1-1}
		+
		\epsilon^\nu \ln^{\tilde{\nu}'+1}(\epsilon^{-1})  |x|^{\rhobar_{N+1}-1-\nu}
		\lesssim \epsilon^\nu \ln^{\tilde{\nu}'}(\epsilon^{-1})|x|^{\rhobar_1-1}.
		\end{align*}
		Here, the last estimate comes from the fact that $N > N_0$, for $N_0$ satisfying \eqref{Num:7029}.
		Last, noticing that $\utieps^{N,M}$ differs from $\utieps^N$ only inside $\Dom_\epsilon$, we obtain \eqref{Num:7034}.
\end{proof}

\section*{Acknowledgement}

We thank Monique Dauge, with whom Marc Josien had an interesting discussion while invited at the Université de Rennes, and who gently pointed out accurate references for homogeneous elliptic operators in corners.
Also, we are indebted to Julian Fischer, who suggested to make use of an ansatz based on Dirichlet correctors instead of the one we first studied \eqref{ansatz}.
As discussed in Section \ref{Sec:Better_Ansatz}, this proved more efficient.

We wish to acknowledge the support of the Max Planck Institute \textit{Mathematik in den Naturwissenschaften} in Leipzig, at which Marc Josien and Mathias Sch\"affner held post-doc positions when this work begun, and which supported this project, in particular by funding visits of Claudia Raithel.
Claudia Raithel also gratefully acknowledges partial support from the Austrian Science Fund (FWF), grants P30000, P33010, W1245, and F65.

\def\cprime{$'$} \def\cprime{$'$} \def\cprime{$'$}

\appendix

\section{Extension of a divergence-free field}

\begin{lemma}
	\label{Lem:DivFree}
	Assume that $\omega \in (0,2\pi)$, and that $H \in \LL^1_\loc(\bar{\Dom})$ is a divergence-free vector field on $\Dom$, written in polar coordinates as 
	\begin{align*}
	H=(H_r,H_\theta) = H_r e_r + H_\theta e_\theta.
	\end{align*}
	Then it can be extended as a divergence-free vector field on the whole space $\R^2$ by
	\begin{align}
	\label{extension_2}
	\bar{H}(r,\theta)
	:=
	\lt\{
	\begin{aligned}
	&H(r,\theta) &&\pour \theta \in \lt[0,\omega \rt],
	\\
	&\tilde{H}(r,\theta) &&\pour \theta \in \lt( \omega, 2\pi \rt),
	\end{aligned}
	\rt.
	\text{ where }
	&\tilde{H}(r,\theta) := \lt( - \alpha H_r, H_\theta \rt)\lt( r, \alpha (2 \pi -\theta) \rt),
	\pour \alpha := \frac{\omega}{2 \pi - \omega}.
	\end{align}
\end{lemma}

\begin{proof}
	First, we may compute 
	\begin{align*}
	\nabla \cdot \tilde{H}(r,\theta)
	&=
	\frac{1}{r} \partial_r (r \tilde{H}_r(r,\theta)) + \frac{1}{r}\partial_\theta \tilde{H}_\theta(r,\theta)
	\\
	&=
	\Big(-\alpha \frac{1}{r} \partial_r (r H_r)  - \alpha \frac{1}{r} \partial_\theta H_\theta \Big) \lt( r, \alpha (2\pi-\theta) \rt)
	\\
	&= - \alpha \lt(\nabla \cdot H\rt)\lt( r, \alpha (2\pi-\theta) \rt) = 0.
	\end{align*}
	Moreover, since
	\begin{align*}
	\tilde{H}(r,\omega) = \lt( - \alpha H_r, H_\theta \rt)(r,\omega) \quad \textrm{and} \quad \tilde{H}(r, 2 \pi) = \lt( - \alpha H_r, H_\theta \rt)(r,0),
	\end{align*}
	the normal component $\bar{H}_\theta$ of $\bar{H}$ is continuous across $\partial\Dom \backslash\{0\}$. 
	This proves that $\nabla \cdot \bar{H}=0$ in $\R^2 \backslash \{0\}$.\footnote{The argument of continuity may not apply strictly speaking to a function $H$ that is only $\LL^1_\loc(\bar{\Dom})$; however, we may use this argument in combination with a mollifying argument.}
	
	It remains to establish that $\nabla \cdot \bar{H} = 0$ in the sense of distributions in $\R^2$ (that is, $\nabla \cdot \bar{H}$ does not suffer from any singularity in $0$).
	By the theory of distributions \cite[Th.\ XXXV p.\ 100]{Schwartz}, since $\bar{H} \in \LL^1_\loc(\R^2)$, there holds
	\begin{align*}
	\nabla \cdot \bar{H} = \delta_0 C, \quad \pour C \in \R^2.
	\end{align*}
	It remains to prove that $C=0$.
	In this perspective, let us define the radial function $\psi= \eta(r)$ where $\eta$ is compactly supported in $[-2,2]$ with $\eta = 1 $ in $[-1,1]$.
	By the theory of distributions, we have (abusively using the symbol $\int$) :
	\begin{align*}
	\int_{\R^2} \bar{H} \cdot \nabla \psi = -	\int_{\R^2}  (\nabla\cdot\bar{H}) \psi =  -\psi(0)C = - C.
	\end{align*}
	On the other hand, (since $\nabla \psi$ is supported into $\Boule_2 \backslash \Boule_1$), by a change of coordinates, we also have
	\begin{align*}
	\int_{\R^2} \bar{H} \cdot \nabla \psi
	=
	\int_{1}^2 \eta'(r) \lt( \int_0^{2\pi} \bar{H}(r,\theta) \cdot e_r \dd \theta\rt) \dd r.
	\end{align*}
	By definition \eqref{extension_2} followed by a change of variables $\tilde{\theta} = \alpha (2\pi - \theta)$, the above inner integral vanishes :
	\begin{align*}
	\int_0^{2\pi} \bar{H}(r,\theta) \cdot e_r \dd \theta
	=
	\int_0^\omega H_r(r,\theta)\dd \theta -  \alpha \int_\omega^{2\pi} H_r(r,\alpha(2\pi - \theta)) \dd \theta
	=0.
	\end{align*}
	Hence, we obtain that $C=0$, which in turn implies that $\nabla \cdot \bar{H} = 0$ in $\R^2$.
\end{proof}

\end{document}